\documentclass[sn-apa]{sn-jnl}

\usepackage{graphicx}%
\usepackage{multirow}%
\usepackage{amsmath,amssymb,amsfonts}%
\usepackage{amsthm}%
\usepackage{mathrsfs}%
\usepackage{xcolor}%
\usepackage{textcomp}%
\usepackage{manyfoot}%
\usepackage{booktabs}%
\usepackage{algorithm}%
\usepackage{algorithmicx}%
\usepackage{algpseudocode}%
\usepackage{listings}%

\usepackage{mathptmx}
\usepackage{latexsym,amssymb,amsmath,amsfonts,amsthm,xcolor,
graphicx, placeins}

\newtheorem{theorem}{Theorem}

\newtheorem{lemma}{Lemma}
\newtheorem{remark}{Remark}

\newtheorem{example}{Example} 
\newtheorem{assumption}{Assumption} 

\usepackage{array}
\newcolumntype{L}{>{\centering\arraybackslash}m{3cm}}

\def\R{\mathbb{R}}
\def\Z{\mathbb{Z}}
\def\P{{\mathbb P}}     
\def\E{{\mathbb E}}  
\def\eqd{\,{\buildrel d \over =}\,}  

\def\<{{\langle}} 
\def\>{{\rangle}} 
\DeclareMathOperator*{\esssup}{ess\,sup}
\newcommand{\vertiii}[1]{{\left\vert\kern-0.25ex\left\vert\kern-0.25ex\left\vert #1 
    \right\vert\kern-0.25ex\right\vert\kern-0.25ex\right\vert}}

\newcommand{\D}[2]{\frac{d#1}{d#2}}
\newcommand{\PD}[2]{\frac{\partial#1}{\partial#2}}

\newcommand{\PDD}[3]{\frac{\partial^{#1}{#2}}{\partial{#3}^{#1}}}

\usepackage{caption}
\title{Strong convergence with error estimates for a stochastic compartmental model of electrophysiology 
}

\date{\today}

\usepackage[sc]{mathpazo}
\usepackage[T1]{fontenc}
\usepackage[mmddyyyy]{datetime}
\usepackage{advdate}
\newdateformat{syldate}{\twodigit{\THEMONTH}/\twodigit{\THEDAY}}
\newsavebox{\MONDAY}\savebox{\MONDAY}{Mon}
\newcommand{\week}[1]{%
  \paragraph*{\kern-2ex\quad #1, \syldate{\today} - \AdvanceDate[4]\syldate{\today}:}
  \ifdim\wd1=\wd\MONDAY
    \AdvanceDate[7]
  \else
    \AdvanceDate[7]
  \fi%
}
\usepackage{setspace}
\usepackage{multicol}
\usepackage{fancyhdr,lastpage}
\usepackage{url}
\usepackage{hyperref}
\usepackage{lastpage}
\usepackage{amsmath}
\usepackage{layout}
\usepackage{subcaption}
\usepackage{array, xcolor}
\usepackage{color,hyperref}
\definecolor{clemsonorange}{HTML}{EA6A20}
\hypersetup{colorlinks,breaklinks,linkcolor=clemsonorange,urlcolor=clemsonorange,anchorcolor=clemsonorange,citecolor=black}

\begin{document}

\title[Article Title]{Strong convergence with error estimates for a stochastic compartmental model of electrophysiology 
}


\author[1]{\fnm{Wai-Tong Louis} \sur{Fan}}
\author*[2]{\fnm{Joshua A.} \sur{ McGinnis }}\email{jam887@sas.upenn.edu}
\author[2,3]{\fnm{Yoichiro} \sur{ Mori}}

\affil[1]{School of Data Science and Society at University, North Carolina-Chapel Hill, NC, 27599, USA}
\affil[2]{Department of Mathematics, University of Pennsylvania, Philadelphia, PA, 19104, USA}
\affil[3]{Department of Biology, University of Pennsylvania, Philadelphia, PA, 19104, USA}


\abstract{
This paper presents a rigorous mathematical analysis, alongside simulation studies, of a spatially extended stochastic electrophysiology model, the Hodgkin-Huxley model of the squid giant axon being a classical example. Although most studies in electrophysiology do not account for stochasticity, it is well known that ion channels regulating membrane voltage open and close randomly due to thermal fluctuations. 

We introduce a spatially extended compartmental model in which this stochastic behavior is captured through a piecewise-deterministic Markov process (PDMP). Space is discretized into $n$ compartments each of which has at most one ion channel. We also devise a numerical method to simulate this stochastic model and illustrate the numerical method by simulation studies.

We show that a classical system of partial differential equations (PDEs) approximates the stochastic system as $n\to \infty$. Unlike existing results, which focus on weak convergence or convergence in probability, we establish an almost sure convergence result with a precise error bound of order $n^{-1/3}$. Our findings broaden the current understanding of stochastic effects in spatially structured neuronal models and have potential applications in studying random ion channel configurations in neurobiology. Additionally, our proof leverages ideas from homogenization theory in PDEs and can potentially be applied to other PDMPs or accommodate other ion channel distributions with random spacing or defects.
}

\keywords{Hodgkin-Huxley model, stochastic ion channels, Piecewise-deterministic Markov process, homogenization theory, numerical methods
}
\maketitle
 
\section{Introduction}
\subsection{Model Derivation}

Electrophysiological phenomena are governed by ion channels embedded in the cell membrane that open and close in response to different cues including the membrane voltage. Most models of cellular and tissue electrical activity consist of differential equations satisfied by membrane voltage and gating variables that describe the opening and closing of ion channels \citep{keener2009mathematical}. This widely successful deterministic description rests on the assumption that the number of ion channels is so large that a population level description of ion channel dynamics is adequate. When the number of ion channels is small, however, stochastic effects of opening and closing of individual channels can play an important role. This may be particularly significant in small cells, especially neurons in the central nervous system whose ion channel distribution can be sparse.

Deterministic models of electrophysiology can be classified into those with or without spatial extent. 
Let us consider a model without spatial extent. For definiteness, we take the Hodgkin-Huxley model as an example:
\begin{align*}
    C_{\rm cap}\D{v}{t} =& -g_{Na} m^3 h (v-E_{\rm Na}) -g_{\rm K} n^4 (v-E_{K}) -g_L (v-E_L)\\
    \D{s}{t}=& \alpha_s(v) (1-s) - \beta_s(v) s, \quad s\in\{m,n, h\},
\end{align*}
Here, $v$ is the membrane voltage and $m,n,h$ are known as gating variables taking values between $0$ and $1$.
The first equation is a statement of current conservation across the cell membrane. The left hand side is the capacitive current where 
$C_{\rm cap}$ is the cell membrane capacitance. The right hand side is a sum of three terms, the Na$^+$, K$^+$ and leak currents which pass through transmembrane ion channels. 
The Na$^+$ conductance is given by $g_{\rm Na}m^3h$ where $g_{\rm Na}$ is the maximal Na$^+$ conductance and the product $m^3h$ is the proportion of ion channels that are open.
Opening of the Na$^+$ channel can thus be thought of as controlled by three $m$ gates and one $h$ gate. The membrane current itself is proportional to the difference between the membrane 
voltage $v$ and the equilibrium voltage $E_{\rm Na}$ for Na$^+$. Similar considerations apply for the K$^+$ channel. In the Hodgkin-Huxley model, 
the ionic current carrier of the leak current is not specified and the leak channel is always assumed to be open. The dynamics of the gating variable $s=m,n,h$ is governed by a differential equation
whose rates are controlled by voltage $v$.

When the above equations were proposed by Hodgkin and Huxley, the existence of ion channels was not known. Subsequent developments, especially the patch-clamp technique, showed
that individual channels open and close stochastically due to thermal fluctuations. The differential equations for the gating variables can thus be interpreted as generating a continuous-time Markov process.
To simplify the presentation, consider the following toy model with voltage $v$ and a single gating variable $z$:
\begin{equation}\label{simpleprob}
C_{\rm cap}\D{v}{t} = -gz(v-E), \quad \D{z}{t}= \alpha(v) (1-z) - \beta(v) z.
\end{equation}
We introduce the stochastic variables $V$ and $Z^{(k)}, k=1,\cdots n$ where $V$ satisfies the equation:
\begin{equation}\label{simpleprobstoch}
C_{\rm cap}\D{V}{t}=-\sum_{k=1}^n \frac{g}{n}Z^{(k)}(V-E)
\end{equation}
and $Z^{(k)}$ is a continuous time Markov process taking values $0$ or $1$ with the transition rate from $0$ to $1$ given by $\alpha(V)$ and $1$ to $0$ given by $\beta(V)$.
The above is an example of a stochastic hybrid system, in which the deterministic equation for $V$ is coupled to a stochastic process. More specifically, it is an example of a piecewise deterministic Markov process (PDMP) which already have various applications in cell biology \citep{bressloff2021stochastic,bressloff_Maclaurin}. 
It is straight forward to write down a similar stochastic model for the Hodgkin-Huxley model.
When $n$ is large, it is expected that the dynamics of the stochastic hybrid system will approach that of the deterministic equation \eqref{simpleprob}.
This is indeed the case and is well-established. Moreover, the central limit theorem is well understood \citep{FoxLu}. Using the Langevin approximation, the authors \citep{Thomas_Pu} were able to develop an efficient simulation algorithm for the full Hodgkin-Huxley model using stochastic shielding. It is even possible to study questions of large deviations in simpler PDMP models such as a PDMP version of  Morris-Lecar in the setting of equation \eqref{simpleprobstoch} \citep{Newby-Keener}, \citep{Newby}. 

The canonical electrophysiological model with spatial extent is the cable model used to describe the propagation of an action potential along an axon. 
The axon is treated as a one-dimensional cable, whose spatial coordinate we denote by $x$. In the context of the the above toy model, the voltage $v(x,t)$ and gating variables $z(x,t)$ satisfy the following equations:
\begin{equation}\label{simple_cable}
C_{\rm cap}\D{v}{t}=D\PDD{2}{v}{x}-gz(v-E), \quad  \PD{z}{t}=\alpha(v) (1-z) - \beta(v) z,
\end{equation}
where $D$ is the axial conductivity of the axon. Likewise, for the Hodgkin-Huxley model, we have:
\begin{align}\label{HH_cable}
    \partial_t v(t,x) =& D\Delta v -g_{Na}(x) m^3 h (v-E_{Na+}) -g_{K}(x) n^4 (v-E_{K+}) -g_L(x) (v-E_L)\\
    \partial_t s=& \alpha_s(v) (1-s) - \beta_s(v) s, \quad s\in\{m,n, h\},
\end{align}
The equation for $v$ can be derived using a current conservation argument applied to the one-dimensional axon. 
Note that the equation for the gating variables have no spatial coupling. 
It is only through the voltage $v$ that the gating variables are spatially coupled.

In this paper, we consider the following stochastic version of the cable model.
Divide the one-dimensional neuron into compartments of length $h$ and assume that the voltage in each compartment is approximately uniform.
Each compartment contains ion channels. Let $V^{(k)}$ be the voltage of the $k$-th compartment, 
and $Z^{(k)}$ be the stochastic gating variable for the ion channel in that compartment. The stochastic version of \eqref{simple_cable} will be:
\begin{equation}\label{comp_simple}
C_{\rm cap}\D{V^{(k)}}{t}=D\frac{V^{(k+1)}-2V^{(k)}+V^{(k-1)}}{h^2}-gZ^{(k)}(V^{(k)}-E)
\end{equation}
where $Z^{(k)}$  is a continuous time Markov process taking values $0$ or $1$ with the transition rate from $0$ to $1$ given by $\alpha(V^{(k)})$ and $1$ to $0$ given by $\beta(V^{(k)})$. 
In the above example, each compartment contains one ion channel each.
In the case of the Hodgkin-Huxley model, we may set:
\begin{align}
\label{E:HH1}
C_{\text{cap}}\dfrac{dV^{(k)}}{dt}&=D\dfrac{(V^{(k+1)}-2V^{(k)}+V^{(k-1)})}{h^2}+I^{(k)}_\text{ion}(V^{(k)}),\\
\label{E:HH2}
I^{(k)}_\text{ion}&=\,
-g^{(k)}_{\rm Na}Y^{(k)} (V^{(k)}-E_{\rm Na+})
-g^{(k)}_{\rm K}Z^{(k)} (V^{(k)}-E_{\rm K+})
-g^{(k)}_{\rm L} (V^{(k)}-E_{\rm L}).
\end{align}
where
\begin{equation}\label{E:HH3}
Y^{(k)}= M^{(k,1)}M^{(k,2)}M^{(k,3)}\,H^{(k)} 
\quad\text{and}\quad
Z^{(k)}= N^{(k,1)}N^{(k,2)}N^{(k,3)}\,N^{(k,4)}.
\end{equation}
For gating variable type $i\in\{M, N, H\}$, the rates for the two-state continuous-time Markov processes are $\{\alpha_{i}(V^{(k)}),\beta_{i}(V^{(k)})\}$. 
In this case, each compartment will have three ion channels corresponding to Na, K and leak channels.

As we shall see in Section \ref{modelformulation}, it is possible to restrict each compartment to have at most one ion channel by 
introducing a multinomial random variable at each location $k$. In the Hodgkin-Huxley model above, the multinomial random variable may have 
four states indicating the presence of a Na channel, a K channel, a leak channel or the absence of a channel. 
The multinomial distribution from which this random variable is drawn can depend on the position $(k)$, making it possible to model 
a spatially dependent ion channel density. 
This allows us to interpret the compartmental length $h$ as the exclusion length between two ion channels.

The main objective of our paper is to consider the limit as $h\to 0$ in the above stochastic model. 
We show that the stochastic model converges to the deterministic model and we quantify the rate of convergence.
We also provide numerical verification for our error estimates.

It is useful to compare our model to previous stochastic cable models that have been studied in the literature. In \citep{Faisal2005, Faisal2007}, the authors consider a similar compartmental model except that there are $N$ ion channels in each of the compartments. Assuming $N$ is large enough, a further approximation is made so that the discretely valued stochastic process of the opening and closing of ion channels is replaced by brownian motion. This facilitates the numerical study of stochastic action potential propagation, which is the subject of \citep{Faisal2005, Faisal2007}. In terms of its spatial resolution, our model is finer than these models in the sense that our compartments can be made to contain at most one channel, and we do not make the brownian motion approximation.

In \citep{austin2008emergence,riedler2012limit}, ion channels are treated as discrete points $x_k$ with spacing $h$ on the one dimensional line, while retaining the second derivative in the cable model. 
The model corresponding to \eqref{simple_cable} will be:
\begin{equation}\label{PDEdelta}
C_{\rm cap}\PD{v}{t}=D\PDD{2}{v}{x}-\sum_{k}gZ^{(k)}\delta(x-x_k)(v-E),
\end{equation}
where $Z^{(k)}$ is a continuous-time Markov process taking values $0$ or $1$ with the transition rate from $0$ to $1$ given by $\alpha(v(x_k))$ and $1$ to $0$ given by $\beta(v(x_k))$.

In contrast to our model, model \eqref{PDEdelta} treats ion channels as a genuine point. In fact, each ion channel has a physical size $\delta$ in the tens of nanometers range and
each ion channel generates a three-dimensional current distribution. Furthermore, the axon is a three-dimensional object with a diameter $d$ which can be as small as $100$nm and possibly thinner for dendritic structures \citep{hudspeth2013principles}. 
Length scales below $d$ or $\delta$ cannot be resolved by a spatially one-dimensional model.
If we let the distance between ion channels $h$ to be of order to $d$ or $\delta$, the biophysical faithfulness
of \eqref{PDEdelta} and \eqref{comp_simple} are thus comparable. Typical values of $h$ vary greatly but can be on the order of $100$ nm \citep{sato2019stochastic} making $h$ comparable to $d$ or $\delta$. Equation \eqref{PDEdelta} can be seen as the limiting case in which the spacing between ion channels
is sufficiently large with respect to $d$ or $\delta$. 
One advantage of our model is that an effective computational method can be devised to simulate the model. 
This is the subject of Section \ref{S:simulation}. On the other hand, it seems difficult to devise an accurate numerical method for \eqref{PDEdelta}, especially near the 
delta function singularities where the ion channels are located. We also point out that our model naturally accommodates spatially random distribution 
of channels by having compartments with no channels. Our analytical results encompass such cases.

The contributions of this paper include the followings. First, we introduce a new stochastic model for voltage-gated ion channels that offers an alternative to existing approaches, such as those in \citep{austin2008emergence}, \citep{buckwar2011exact}, \citep{riedler2012limit}, and \citep{riedler2015spatio}. Our model is more natural from the perspective of numerical simulation and encompasses several interesting examples, which we discuss in Section \ref{SS:Examples}. We explain specific strategies for simulation in Section \ref{S:simulation}.

Second, our main result (Theorem \ref{T:Main} in Section \ref{S:Result})  establishes a strong convergence result of our stochastic model as $n \to \infty$ ($h\to 0$). It not only gives a stronger notion of convergence than the functional laws of large numbers in existing work, but also offers an explicit error bound; see our explanation after Theorem \ref{T:Main}.

Third, our proof method in Section \ref{S:Proof} employs a multi-scale ansatz to define a "corrector" term that can be bounded in terms of $h=L/n$
almost surely, with this term tending to zero as $n\to\infty$. This approach is similar to the method used in \citep{wright2022} and \citep{mcginnis2023}, introducing a homogenization technique to the study of piecewise-deterministic Markov processes (PDMPs) for the first time.

\subsection{A compartmental model for voltage-gated ion channels}\label{modelformulation}

In this section, we describe our general model and explain how it covers  a few important examples including \eqref{E:HH1}-\eqref{E:HH3}.

Suppose the ion channels are located on a circle $\mathbb{S}$ with total length $L$. We note here that our assumption of a circular axon is to avoid considerations of boundary effects. For sealed axons (no current boundary conditions at both ends), our circular results can be applied directly by simply considering a circle made of two identical copies of sealed axons.
Fix a positive integer $n\in\mathbb{N}$ and divide the circle  into $n$ compartments of equal length $h=h_n=L/n$. Under this discretization, there are $n$ points $\mathbb{S}_n:=\{kh\}_{k=0}^{n-1}\subset \mathbb{S}$. 
So, in the discrete-space model, locations in space are indexed by $k$ and correspond to $kh\in \mathbb{S}_n\subset \mathbb{S}$ in the macroscopic model.

\medskip
\noindent
{\bf Stochastic models. } 
For the $k$-th compartment at time $t$, the voltage and the states of the ion channels  will be represented  by
$V^{(k)}(t)$ and $\{Z^{(k)}_{i,j}(t)\}_{1\leq i\leq I,\,1\leq j\leq J}$ respectively, where $I,\,J\in\mathbb{N}$ are fixed throughout this paper. Roughly,
\[
Z^{(k)}_{i,j}(t)=
\begin{cases}
1,& \quad \text{if the channel of type }i \text{ is in configuration }j \text{ (e.g. open)}\\
0,& \quad \text{if the channel of type }i \text{ is not in configuration }j \text{ (e.g. closed)}
\end{cases}
\]
Precisely, we  consider a sequence of continuous-time Markov processes, indexed by $n$,  that is parameterized by two collections  $\{g_{i,j}:\;1\leq i\leq I,\,1\leq j\leq J \}$ and 
$\{A_{i,(a,b)}:\;1\leq i\leq I,\,1\leq a,b\leq J,\,a\neq b \}$ of deterministic functions on $\R$. 
For each $n\in \mathbb{N}$, such a  process $$(V,Z):=\left( V^{(k)},\,\{Z^{(k)}_{i,j}\}_{1\leq i\leq I,\,1\leq j\leq J}\right)_{k=0}^{n-1}$$
is described in \eqref{eq:general_model V} and \eqref{eq:general_model Z} below. 

\FloatBarrier
\begin{figure}
    \centering
    \includegraphics[scale=0.3]{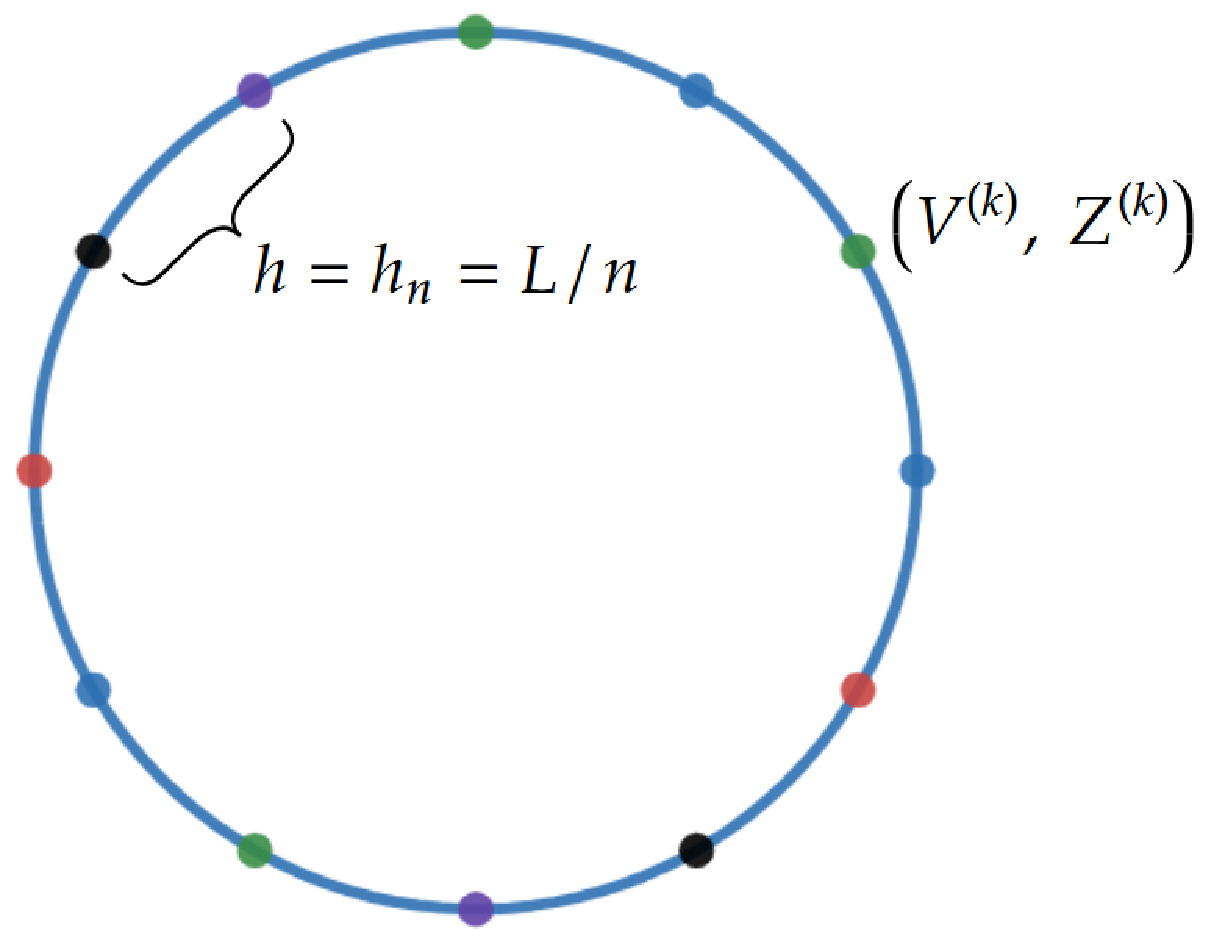}
    \caption{A discretized circle $\mathbb{S}_n:=\{kh\}_{k=0}^{n-1}$ with $n=12$, which consists of $n$ points that divide the circle $\mathbb{S}$ into $n$ pieces of length $h=h_n=L/n$. At each location $kh$, the voltage $V^{(k)}$ and the states of the channels $Z^{(k)}=\{Z^{(k)}_{i,j}\}_{1\leq i\leq I,\,1\leq j\leq J}$ are described by equations
    \eqref{eq:general_model V} and \eqref{eq:general_model Z} respectively
    }
    
\end{figure}

For each $k\in\{0,1,\ldots,n-1\}$, we consider the equation
\begin{align}
\dot{V}^{(k)}(t)= D\dfrac{V^{(k+1)}(t)-2V^{(k)}(t)+V^{(k-1)}(t)}{h^2}+\sum_{i=1}^{I}\sum_{j=1}^{J}Z^{(k)}_{i,j}(t)\,g_{i,j}(V^{(k)}(t)),
\label{eq:general_model V}
\end{align}
and, for  $i\in\{1,2,\ldots, I\}$ and $j\in\{1,2,\ldots, J\}$, the equation
\begin{align}\label{eq:general_model Z}
Z^{(k)}_{i,j}(t) = Z^{(k)}_{i,j}(0)\,&+\sum_{a\in \{1,\ldots,J\}:\,a \neq j}Y^{(k)}_{i,(a,j)}\left(\int_0^tA_{i,(a,j)}(V^{(k)}(s))\,Z^{(k)}_{i,a}(s)\,ds\right) \notag\\
&-\sum_{a\in \{1,\ldots,J\}:\,a \neq j}Y^{(k)}_{i,(j,a)}\left(\int_0^tA_{i,(j,a)}(V^{(k)}(s))\,Z^{(k)}_{i,j}(s)\,ds\right)
\end{align}
for $t\in (0,\infty)$,
where
$\{Y^{(k)}_{i,(a,b)}\}$ are independent unit rate Poisson processes. 

The index $k$ specifies the spatial location, as before, and the index $i$ specifies different types of ion channels. For fixed $i$ and $k$, one may think of index $j$ as specifying the configuration of the ion channel of type $i$ at location $k$. A specific way to assign configurations to an ion channel is by assigning a configuration to each combination of its gates' states i.e. being open or closed. Naturally then our model makes sense only if an ion channel can only be in one configuration at once. In \eqref{eq:general_model Z}, ones sees that for any fixed $i\in\{1,...,I\}$ and $j\in\{1,...,J\}$ and at any time $t$, if $Z^{(k)}_{i,j}(t)=1$ and if for all $a \in \{1,...,J\} $ s.t $a \neq j$, $Z^{(k)}_{i,a}(t)=0$, then $Z^{(k)}_{i,j}(t)$ must almost surely decrease to $0$ before it increases. Likewise, if $Z^{(k)}_{i,a}(t)=0$ except for just one $a_0 \neq j$ so that $Z^{(k)}_{i,a_0}(t)=1$, then $Z^{(k)}_{i,j}(t)=0$, and it must almost surely increase to $1$, before it decreases.

\begin{assumption}\label{A:functions g and A}
Suppose  $\{g_{i,j}:\;1\leq i\leq I,\,1\leq j\leq J \}$ and 
$\{A_{i,(a,b)}:\;1\leq i\leq I,\,1\leq a,b\leq J,\,a\neq b \}$ are deterministic functions on $\R$ such that 
\begin{itemize}
    \item $g_{i,j}:\,\R\to\R$ is globally Lipschitz continuous for all $1\leq i\leq I,\,1\leq j\leq J$. 
    
    \item $A_{i,(a,b)}:\,\R\to\R_+$ is a non-negative Borel-measurable function that is bounded on any compact interval, whenever $a\neq b$ and $1\leq i\leq I,\,1\leq a,b\leq J$. 
   
\end{itemize}
\end{assumption}

Given $(V^{(k)}(t))_{t\in\R_+}$ (i.e. suppose we know the voltage at location $k$ for all time), the vector
$Z^{(k)}_{i,\cdot}:=(Z^{(k)}_{i,j})_{j=1}^{J}$ would then be a time in-homogeneous continuous-time Markov process for each  $i\in \{1,2,\ldots,I\}$, with transition rates given by the functions $\{s\mapsto A_{i,(a,b)}(V^{(k)}(s))\}_{a\neq b}$. Under our initial conditions discussed further below in Assumption \ref{A:initial condtions}, $Z^{(k)}_{i,j}(t)$ is either 0 or 1 (i.e. takes value in $\{0,1\}$) for all time $t\in\R_+$ and all $(i,j)$.

Let  $\textbf{A}_{i}(\cdot)$ be the corresponding 
$J\times J$ rate matrix. That is, the off-diagonal entries are $A_{i,(a,b)}(\cdot)$ for $a\neq b$, and the diagonal entries are defined so that the row sums are zero. The forward Kolmogorov's equation for $Z^{(k)}_{i,\cdot}$, given $(V^{(k)}(t))_{t\in\R_+}$, can be written as 
$$\dfrac{d}{dt}\textbf{P}^{(k)}_i(t) =\textbf{P}^{(k)}_i(t)\, \textbf{A}_{i}(V^{(k)}(t)), $$
where $\textbf{P}^{(k)}_i(t)=(p^{(k)}_i(t,z,x))_{z,x}$ is the transition matrix for the continuous-time Markov chain $Z_{i,\cdot}^{(k)}$, given $(V^{(k)}(t))_{t\in\R_+}$.

\begin{remark}[Related models]\rm
The model  \eqref{eq:general_model V}-\eqref{eq:general_model Z} is general enough to incorporate  products of independent gates, such as those in the stochastic Hodgkin-Huxley model \eqref{E:HH1}-\eqref{E:HH3}. 
When $g_{i,j}$ are all linear functions for all $(i,j)$ and $A_{i,(a,b)}$ are all smooth and take values between two fixed positive constants, a similar model was considered in \citep{austin2008emergence} in which
the potential $V$ is a continuous function on an interval and is subject to Dirichlet boundary condition. A major difference is that in \citep{austin2008emergence} the voltage is a continuous function in space but here in  \eqref{eq:general_model V}-\eqref{eq:general_model Z} we consider the voltages $\{V^{(k)}\}$ in discrete compartments.
Both our model here and that in  \citep{austin2008emergence} lie in the framework of piecewise-deterministic Markov processes; see \citep{buckwar2011exact,riedler2012limit,riedler2015spatio}.    
\end{remark}

\medskip
\noindent
{\bf PDE models. } 
Later, we shall let $n\to\infty$ (equivalently $h=h_n\to 0$) while fixing $I$ and $J$.
The result of this paper, loosely stated, is that as $h \to 0$, the dynamics $(V,Z)$ are well approximated by those of $(v,z)$ for finite times where $(v,z)$ solves the following PDE: 
\begin{align}
\dfrac{\partial v}{\partial t}(t,x)=&\, D\Delta v+\sum_{i=1}^{I}\sum_{j=1}^{J}z_{i,j}(t,x)g_{i,j}(v(t,x)) \label{eq:general_PDE1}
\\
\dfrac{\partial z_{i,\cdot}}{\partial t}(t,x)=&\,A_i(v(t,x))^Tz_{\ell,\cdot}(t,x),\quad 1\leq i\leq I.
\label{eq:general_PDE2}
\end{align}

We represent a point in $\mathbb{S}$ by a principle angular variable $\theta\in [0,\,L)$, and endow $\mathbb{S}$ with the Lebesgue measure of the interval $[0,\,L)$. 
The transition
probability density of a standard Brownian motion $B$ on $\mathbb{S}$ with respect to such  measure is explicitly given by 
\begin{equation}\label{densityBM_S}
 p^{\mathbb{S}}(t,x,y)=\frac{1}{\sqrt{2\pi t}} \sum_{k\in\Z} e^{\frac{-(y-x+L k)^2}{2t}}.
\end{equation}
The transition density for a Brownian motion with variance $2D$ is $p^{\mathbb{S}}(2D t,x,y)$. 
We shall interpret \eqref{eq:general_PDE1} via Duhammel's formula as the integral equation
\begin{align}
    v(t,x)=e^{D \Delta t}v(0,\cdot)(x)+\int_0^{t}e^{D \Delta (t-s)}\left(\sum_{i=1}^{I}\sum_{j=1}^{J}z_{i,j}(s,\cdot)g_{i,j}(v(s,\cdot)) \right)(x)\,ds,
    \label{eq:Duhmlv}
\end{align}
where we define the fundamental operator acting on functions $\phi \in L^{\infty}(\mathbb{S})$ as
\begin{align}
e^{D \Delta t}\phi(x) := \int_{\mathbb{S}}p^{\mathbb{S}}(2Dt,x,y)\phi(y)\,dy,\qquad (t,x)\in \R_+\times \mathbb{S}.
\end{align}
Equation \eqref{eq:general_PDE2} can be more explicitly written as
\begin{equation}
\begin{aligned}
\label{eq:general_PDE2z}
z_{i,j}(t,x)= z_{i,j}(0,x)\,+\sum_{a \in\{1,\ldots,J\}: a \neq j} \int_0^tA_{i,(a,j)}(v(s,x))z_{i,a}(s,x)ds
\\
-\sum_{a \in\{1,\ldots,J\}: a \neq j}\int_0^tA_{i,(j,a)}(v(s,x))z_{i,j}(s,x)ds,  
\end{aligned}
\end{equation}
where we used the definition  $A_{i,(a,a)}:=-\sum_{j \in \{1,\ldots,J\}:j\neq a}A_{i,(a,j)}$. Equation \eqref{eq:general_PDE2z} should be compared with \eqref{eq:general_model Z}.

Before describing some examples, we address existence and uniqueness of solutions.
\begin{assumption}[Initial condition for PDE \eqref{eq:general_PDE1}-\eqref{eq:general_PDE2}] \label{A:initial condtions_PDE}
Suppose  $v_0\in \mathcal{C}(\mathbb{S})$ and
 $\{(z_ {i,j})_0\}_{1 \leq i \leq I,\,1 \leq j \leq J} \subset \mathcal{C}(\mathbb{S};\,[0,1])$ are such that 
\begin{equation}\label{E:initial condtions_PDE}
 \sum_{j=1}^{J}(z_{ i,j})_0(x)=1 \quad\text{for all } i\in\{1,2,\cdots, I\} \text{ and } x\in\mathbb{S}.
\end{equation}
\end{assumption}

\begin{lemma}[Existence and uniqueness for PDE]\label{L:existencePDE}
Suppose Assumptions \ref{A:functions g and A} and \ref{A:initial condtions_PDE} hold.  Then \eqref{eq:general_PDE1}-\eqref{eq:general_PDE2} has a unique solution in the following sense:
there exists a unique element $(v,z)$  in $\mathcal{C}(\R_+\times \mathbb{S})\times \mathcal{C}(\R_+\times \mathbb{S};\,[0,1])^{IJ}$ 
that satisfies equation \eqref{eq:Duhmlv} and equations \eqref{eq:general_PDE2z} for all $(t,x)\in \R_+\times \mathbb{S}$.
Furthermore, $v$ and $z:=\left(z_{i,j}:\;1\leq i\leq I,\,1\leq j\leq J\right)$ satisfy the followings:
\begin{itemize}
\item[(i)]  there exists a constant $C\in(0,\infty)$ such that for all $T\in(0,\infty)$, 
\begin{equation}
    \|v \|_{L^\infty([0,T]\times \mathbb{S})} \leq C e^{C T}.
    \label{inq_v}
\end{equation}
\item[(ii)] for all time $t\in\R_+$,
\begin{equation}
 \sum_{j=1}^{J}z_{ i,j}(t,x)=1 \quad\text{for all } i\in\{1,2,\cdots, I\} \text{ and } x\in\mathbb{S}.
\end{equation}
\end{itemize}
Suppose, furthermore, for all indexes $1 \leq i \leq I,\,1\leq j,a,b \leq J$,
\begin{equation}\label{A:regularity}
g_{i,j}\in\mathcal{C}^{1,1}(\mathbb{R}),\;A_{i,(a,b)}\in \mathcal{C}^{1}(\mathbb{R}),\;(z_ {i,j})_0\in  \mathcal{C}^1(\mathbb{S};\,[0,1]),\;v_0\in \mathcal{C}^2(\mathbb{S}).  
\end{equation}
Then  $(v,z)\in \mathcal{C}^2(\R_+\times \mathbb{S})\times \mathcal{C}^1(\R_+\times \mathbb{S};\,[0,1])^{IJ}$.
\end{lemma}

The proof of Lemma \ref{L:existencePDE} follows from a standard fixed-point argument and Gronwall's lemma; see \citep[Section 3.3]{austin2008emergence}. The regularity assumption \eqref{A:regularity} are not needed for the existence and uniqueness of continuous solution in Lemmas \ref{L:existencePDE}  and \ref{L:existenceStoc}, but it will be assumed in Theorem \ref{T:Main}.

\begin{assumption}[Initial condition for \eqref{eq:general_model V} and \eqref{eq:general_model Z}] \label{A:initial condtions_Discrete} \label{A:initial condtions}
Suppose that, for each $n\in\mathbb{N}$ (hence for each $h=h_n=L/n$), 
$V^{(k)}(0)=v_0(hk)$ for $k\in\{0,1,\cdots,n-1\}$ and that
$\{Z_{i,\cdot}^{(k)}(0):\;1\leq i\leq I,\,0\leq k\leq n-1\}$ are independent and identically distributed (i.i.d.) random variables such that 
\begin{equation}\label{E:initial condtions_Discrete}
\mathbb{P}(Z_{i,\cdot}^{(k)}(0)=\,e_j)=\ (z_{i,j})_0(hk) \quad \text{for }j\in\{1,2,\cdots, J\},
\end{equation}
where $\{e_j\}_{j=1}^J$ is the standard basis of $\R^J$.
\end{assumption}

\begin{remark}
    Independence in $i$ is not needed for the proof of Theorem \ref{T:Main}, but we keep it here to fix ideas. See Example \ref{Ex:Switching Channels} for an example of dependent initial conditions.  
\end{remark}

\begin{lemma}[Existence and uniqueness for stochastic model]\label{L:existenceStoc}
Suppose Assumptions \ref{A:functions g and A}, \ref{A:initial condtions_PDE} and \ref{A:initial condtions} hold. Fix any $n\in \mathbb{N}$.
Suppose $\{Y^{(k)}_{i,(a,b)}:\;:\;1\leq i\leq I,\,1\leq a,b\leq J,\,a\neq b,\;0\leq k \leq n-1 \}$ are independent unit rate Poisson processes and
$\{Z_{i,\cdot}^{(k)}(0):\;1\leq i\leq I,\,0\leq k\leq n-1\}$ are independent random variables satisfying Assumption \ref{A:initial condtions}, and these two independent families are defined on the same probability space $(\Omega, \mathcal{F},\P)$. Then there exists a unique continuous-time Markov process $(V,Z):=\left( V^{(k)},\,\{Z^{(k)}_{i,j}\}_{1\leq i\leq I,\,1\leq j\leq J}\right)_{k=0}^{n-1}$ that satisfies
\eqref{eq:general_model V} and \eqref{eq:general_model Z}. 
Furthermore, $V$ and $Z$ satisfy the following:
\begin{itemize}
\item[(i)] there exists a deterministic constant $C\in(0,\infty)$ such that for all $T\in(0,\infty)$,
\begin{equation}\label{E:aprioriBoundV}
\sup_{n\in\mathbb{N}}\sup_{t\in[0,T]} \max_{0\leq k\leq n-1}|V^{(k)}(t)| \leq C\,e^{CT} \qquad\P-a.s.
\end{equation}
\item[(ii)]
for all $1\leq i\leq I$ and $0\leq k\leq n-1$,
the process
$Z_{i,\cdot}^{(k)}=\left(Z_{i,j}^{(k)}\right)_{j=1}^J$ takes values in the standard basis of 
$\R^J$ and satisfies
 \[
 \sum_{j=1}^{J}Z^{(k)}_{ i,j}(t)=1 \qquad\text{for all }t\in\R_+.
 \]
 \end{itemize}
\end{lemma}
The proof of Lemma \ref{L:existenceStoc} follows from solving from one jump of $Z$ to the next jump, as a piecewise-deterministic Markov process; see \citep[Section 1.2]{anderson2015stochastic} and  \citep[Section 3.4]{austin2008emergence}. The initial conditions would be fully specified if we also demand independence in $i$; however doing so is unnecessary for the results of this paper. Nevertheless, for the Hodgkin-Huxley model described by \eqref{Def:Hodgkin_Huxley}, it is probably most natural to take independence in $i$ as it indexes different types of ion channels.

\subsection{Examples}\label{SS:Examples}

\begin{example}[Wave propagation]\label{Eg:AllenCahnType}\rm

Consider the following special case of the general model \eqref{eq:general_model V}-\eqref{eq:general_model Z}:
\begin{equation}\label{E:V}
\frac{dV^{(k)}_t}{dt}=\, \frac{D\,(V^{(k-1)}-2V^{(k)}+V^{(k+1)}) }{h^2} \,+\,Z^{(k)}_t\,f(V^{(k)}_t) -g(V^{(k)}_t)
\end{equation}
where 
$\{Z^{(k)} \}_{k=0}^{n-1}$ satisfy
\begin{equation}\label{E:Zkl}
Z^{(k)}_t=\,Z^{(k)}_0+ \mathcal{N}^{(k)}_+\left(\int_0^t \alpha(V^{(k)}_s)\, (1-Z^{(k)}_s)\,ds\right) - \mathcal{N}^{(k)}_-\left(\int_0^t \beta(V^{(k)}_s)\,Z^{(k)}_s \,ds\right), 
\end{equation}
where $\{ \mathcal{N}^{(k)}_+,\, \mathcal{N}^{(k)}_-\}_{k\in \mathbb{S}_n}$ are independent unit-rate Poisson processes. 

Here, there is one channel with one gate at each location $k$ and the channel state  $Z^{(k)}$ is either open (1) or closed (0). Furthermore, given the $V^{(k)}$,  the transition rate of $Z^{(k)}$ from closed (0) to open (1)  is $\alpha(V^{(k)}_t)$, and that from open to closed is $\beta(V^{(k)}_t)$  at time $t$. 

Note that \eqref{E:V}-\eqref{E:Zkl} is indeed a special case of the general model \eqref{eq:general_model V}-\eqref{eq:general_model Z} by the following choice of parameters and initial conditions. 
Set $J=I=2$, $g_{1,1}(v)=f(v)$ and $g_{2,1}=-g(v)$. Further set $g_{1,2}\equiv g_{2,2}\equiv 0$. The rate functions are given by $A_{1,(1,2)}=\alpha$ and $A_{1,(2,1)}=\beta$. For all $(a,j) \in \{1,2\}\times\{1,2\}$ we have that $A_{2,(a,j)}=0$. Because there is no stochastic term associated with $g_{2,1}$, we would need to initialize $Z^{(k)}_{2,1}=1$ for all $k$. Since the rate matrix is $0$ for $i=2$, $Z^{(k)}_{2,1}$ remains $1$ for all time. 



In this example, the corresponding PDE
\eqref{eq:general_PDE1}-\eqref{eq:general_PDE2}
becomes
\begin{align}
\frac{\partial v(t,x)}{\partial t}=&\,D\Delta v\,+\,z\,f(v) -g(v) \label{eqn_V_pde}\\
\frac{dz(t,x)}{dt}=&\,\alpha(v) (1-z) -\beta(v)z. \label{eqn_Z_pde}
\end{align}
PDE \eqref{E:V}-\eqref{E:Zkl} can exhibit travelling wave behavior (see Section \ref{S:simulation}) and is related to reaction-diffusion equations such as the Allen-Cahn equation. Suppose $\alpha$ and $\beta$ tends to infinity at the same rate as a parameter $\epsilon\to 0$. For example, suppose 
\[
\alpha(v)=\frac{1}{\epsilon} \hat{\alpha}(v) \quad\text{and}\quad \beta(v)=\frac{1}{\epsilon}\hat{\beta}(v)
\]
for some fixed positive functions $\hat{\alpha}$ and $\hat{\beta}$. Heuristically, as $\epsilon\to 0$, equation \eqref{eqn_Z_pde} suggests that $z(t,x)\to \frac{\alpha(v)}{\alpha(v)+\beta(v)}$ and that equation \eqref{eq:general_PDE1} converges to 
\[
\frac{\partial v(t,x)}{\partial t}=\,D\Delta v\,+\,\frac{\alpha(v)}{\alpha(v)+\beta(v)}\,f(v) -g(v).
\]
which is a reaction-diffusion equation. For suitable choices of $\alpha$ and $\beta$, the reaction term can be bi-stable; see for instance  \citep[Chapter 6]{keener2009mathematical} for wave propagation in excitable systems. 

As a remark, a related slow-fast analysis for the stochastic model  in \citep{austin2008emergence} 
was considered in
\citep{genadot2012averaging} and  \citep{genadot2014multiscale}, where the spacing between adjacent channel ($h$ in this paper and $1/N$ in their papers) is fixed.

\end{example}

\begin{example}[Multiple Gates] \label{Ex:product_process}\rm
We give a simple example  to demonstrate how products of gates can be formulated as Markov chains.  Fixing $i\in \{1,2,\ldots, I\}$ and $k\in\mathbb{S}_n$, suppose we are given two independent gates $\xi_{i}^{(k)} \in \{0,1\}$ and $\widetilde{\xi}_{i}^{(k)}\in \{0,1\}$. Suppose that $\xi_{i}^{(k)}$ transitions from $0$ to $1$ with rate $\alpha$ and from $1$ to $0$ with rate $\beta$ . Analogously the rates for $\widetilde{\xi}_{i}^{(k)}$ are given by $\tilde{\alpha}$ and $\tilde{\beta}$.  Then the process $(\xi_{i}^{(k)},\widetilde{\xi}_{i}^{(k)})\in \{0,1\}^2$ can be modeled by a process $Z_{i,\cdot}^{(k)} \in \{e_1,e_2,e_3,e_4\} $ meaning
\begin{equation}
\begin{aligned}
Z_{i,\cdot}^{(k)} \eqd \left[(1-\xi_{i}^{(k)})(1-\widetilde{\xi}_{i}^{(k)}), \ \xi_{i}^{(k)}(1-\widetilde{\xi}_{i}^{(k)}), \ (1-\xi_{i}^{(k)})\widetilde{\xi}_{i}^{(k)}, \ \xi_{i}^{(k)}\widetilde{\xi}_{i}^{(k)}\right]^T.
\end{aligned}
\label{Eq:example_1}
\end{equation}
 Thus $J=4$ in this example is determined by the fact that we need to represent a product of two binary gates. The rates for $Z_{i,\cdot}^{(k)}$ are then determined by the rates $ \xi_{\ell}^{(k)}$ and $\widetilde{\xi}_{\ell}^{(k)}.$  The rate matrix is given by
 \begin{align*}
     \textbf{A}_i=
     \begin{bmatrix}
      -\alpha -\widetilde{\alpha} & \alpha & \widetilde{\alpha} & 0 \\
       \beta  & -\widetilde{\alpha}-\beta & 0 & \widetilde{\alpha} 
       \\
      \widetilde{\beta}  & 0 & -\alpha -\widetilde{\beta} & \alpha
       \\
       0  & \widetilde{\beta} & \beta & -\beta-\widetilde{\beta} 
     \end{bmatrix}.
 \end{align*}
 For example, the rate that $Z_{i,\cdot}^{(k)}$ transitions from $e_1$ to $e_2$ is given by $\alpha$ and the rate from $e_4$ to $e_2$ is given by $\widetilde{\beta}.$ Thus the processes $Z_{i,4}^{(k)}$ exactly models $\xi_{i}^{(k)}\widetilde{\xi}_{i}^{(k)}$, the product of two gates.  Note, however, that defining $Z_{i,\cdot}^{(k)}$ as the product process of two gates is only a subset of the ways in which we may define it, and it is thus a more general object. 
\end{example}

\begin{example}[Hodgkin-Huxley model] \label{Ex:HH}\rm
To obtain the Hodgkin-Huxley model \eqref{E:HH1}-\eqref{E:HH3} specifically, we take $I=3$ and $J=16$. $I=3$ corresponds to the three channels: Sodium, Potassium, and leak. $J=16$ accommodates the four gates for Sodium channels and Potassium channels. For $a,b \in \{1,2,\ldots, 16\}$, we write $a \sim b$, if $(a,b)$ is an edge between the labeled vertices of the hypercube pictured in Figure \ref{fig:Labeled_Cube}. Then one possible way to define the rate matrices and functions $\{g_{i,j}\}$ is
\begin{align}
A^{(k)}_{i,(a,b)}=\begin{cases}
\alpha_{M} & i=1, \ a \sim b, \ \text{and} \ 0<b-a<8 \\
\beta_{M} & i=1, \ a \sim b, \ \text{and} \ 0<a-b<8 \\
\alpha_{H} & i=1, \ a \sim b, \ \text{and} \  b-a=8\\
\beta_{H} & i=1, \ a \sim b, \ \text{and} \  b-a=-8\\
\alpha_{N} & i=2, \ a \sim b, \ \text{and} \ a<b \\
\beta_{N} & i=2, \ a \sim b, \ \text{and} \ a>b\\
0 & \text{Otherwise}
\end{cases}
\quad 
g_{i,j}(\cdot)=
\begin{cases}
-g_{Na}(\cdot-E_{Na+}) & (i,j)=(1,16) \\
-g_{K}(\cdot-E_{K+}) & (i,j)=(2,16) \\ 
-g_{L}( \cdot -E_L) & (i,j)=(3,1) \\
0 & \text{Otherwise}
\end{cases}.
\label{Def:Hodgkin_Huxley}
\end{align}
\begin{figure}
    \caption*{Transition Network of Hodgkin-Huxley Ion Channel}
    \centering
    \includegraphics[width=\textwidth]{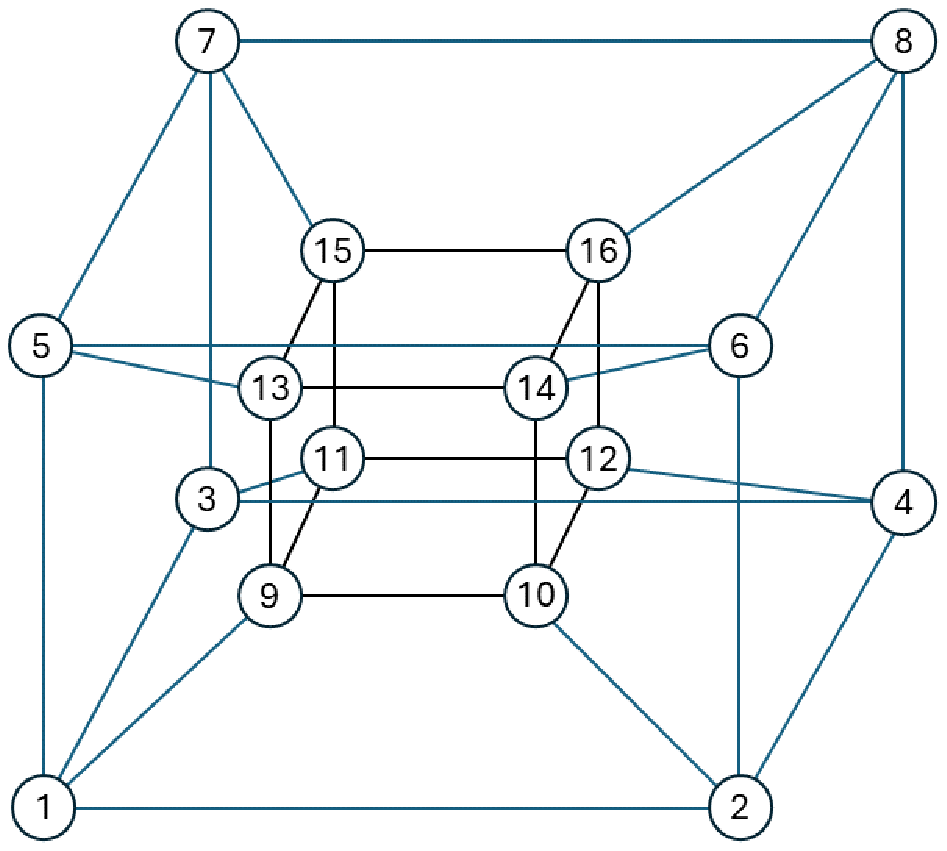}
    \caption{Possible transition between configurations of an ion channel in the Hodgkin-Huxley Model.}
    \label{fig:Labeled_Cube}
\end{figure}
Note that the leak channel is not stochastic, and therefore we choose the rate matrices for the leak channels to be $0$ in all entries. Then according to our choice of the functions $\{g_{i,j}\}$, at $t=0$ we have to set $Z_{3,1}^{(k)}=1$.

One way to pick initial conditions in this example would be the following. Since, we have three gate types $\{M,N,H\}$, we suppose we are given three sufficiently smooth functions $z_H , z_M,$ and $ z_N$,  and suppose $z_H(x),z_M(x),z_N(x) \in [0,1]$ for all $x$. Then we let 
    \begin{align}
    (z_{i,j})_0 \equiv \begin{cases}
    (2-i)(1-z_M)^3(1-z_H)+(i-1)(1-z_N)^4 & i \in \{1,2\}, j=1 \\
   (2-i)z_M(1-z_M)^2(1-z_H) +(i-1)z_N(1-z_N)^3& i \in \{1,2\}, j \in \{2,3,5\}
   \\
   (2-i)z_M^2(1-z_M)(1-z_H)+(i-1)z_N^2(1-z_N)^2 & i \in \{1,2\}, j \in \{4,6,7\}
   \\
   (2-i)z_M^3(1-z_H)+(i-1)z_N^3(1-z_N) & i \in \{1,2\}, j =8 \\
   (2-i)(1-z_M)^3z_H+(i-1)(1-z_N)^3z_N & i \in \{1,2\},j=9\\
   (2-i)z_M(1-z_M)^2z_H+(i-1)z_N^2(1-z_N)^2& i \in \{1,2\}, j \in \{10,11,13\}
   \\
   (2-i)z_M^2(1-z_M)z_H+ (i-1)z_N^3(1-z_N)& i\in\{1,2\}, j \in \{12,14,15\}
   \\
   (2-i)z_M^3z_H+(i-1)z_N^4 & i\in \{1,2\}, j =16 
   \\
   1 & i=3, j=1 \\
   0 & \text{Otherwise}
    \end{cases}.
    \end{align}
    Further we assume independence in $i$ so the that the distribution on the initial conditions is fully specified.
\end{example}

In defining the Hodgkin-Huxley model in \eqref{Def:Hodgkin_Huxley}, we assumed that each compartment contained one of each kind of channel. We can introduce missing channels into the model by including a degenerate random process that is constant in time, but whose initial conditions are still chosen randomly according to Assumption \ref{A:initial condtions}. We demonstrate this with an example. 
\begin{example}
[Mutually exclusive channel]
\label{Ex:Switching Channels}
\rm
Suppose that in each compartment $k$ there is either the ion channel associated with $f$ or the one associated with $g$, but never both. Moreover, the ion channel associated with $f$ can open and close, while the one associated with $g$ is always open. Such a model can be obtained by modifying the set-up of Example \ref{Ex:product_process} with $i \in\{1,2\}$. Suppose $\tilde{\alpha}_i\equiv \tilde{\beta}_i\equiv 0$. Further suppose that  $\widetilde{\xi}^{(k)}_1(0) \sim \xi^{(k)}_1(0) \sim\text{Ber}(p)$, $\widetilde{\xi}^{(k)}_1(0)$ and  $\xi^{(k)}_1(0)$ are independent. Then assume  $\xi^{(k)}_2(0)=1-\xi^{(k)}_1(0)$ and $\widetilde{\xi}^{(k)}_2(0)=1-\widetilde{\xi}^{(k)}_1(0)$. Suppose $V$ satisfies 
\begin{equation}\label{E:V_xik2}
\frac{dV^{(k)}_t}{dt}=\, \frac{D\,(V^{(k-1)}-2V^{(k)}+V^{(k+1)}) }{h^2} \,+\,\xi_1^{(k)}\,\widetilde{\xi}_1^{(k)}\,f(V^{(k)}_t) \,-\,\widetilde{\xi}_2^{(k)}\,g(V^{(k)}_t).
\end{equation}

To see we can rewrite the equation \eqref{E:V_xik2} in the framework of our general model \eqref{eq:general_model V}-\eqref{eq:general_model Z}, first
note \eqref{Eq:example_1} still holds for all time if we set 
\begin{align}
\mathbb{P}(Z^{(k)}_{1,\cdot}(0)=e_j)=
\begin{cases}
(1-p)^2& j=1\\
(1-p)p&j=2\\
(1-p)p&j=3\\
p^2 &j=4\\
\end{cases},
\quad \text{and} \quad Z^{(k)}_{2,\cdot}(0)=
\begin{cases}
e_1 & Z_{1,\cdot}^{(k)}=e_3 \\
e_2 & Z_{1,\cdot}^{(k)}=e_4  \\  
e_3 & Z_{1,\cdot}^{(k)}=e_1  \\
e_4 & Z_{1,\cdot}^{(k)}=e_2
\end{cases}.
\label{E:example2}
\end{align} 
Moreover, since $\widetilde{\alpha}_i\equiv \widetilde{\beta}_i\equiv 0$, the second part of \eqref{E:example2} holds for all time and either $Z_{1,\cdot}^{(k)}(t) \in \{e_1,e_2\}$ or $Z_{1,\cdot}^{(k)}(t) \in \{e_3,e_4\}$. Thus, just as only one of the two products $\{\xi_i^{(k)}\widetilde{\xi}_{i}^{(k)}\}_{i=1,2}$  can ever be equal to $1$, so too can only one of $\{Z^{(k)}_{i,\cdot}\}_{i=1,2}$ be equal to $e_4.$ This is reflected in the block structure of the rate matrices
\begin{align*}
     \textbf{A}_i=
     \begin{bmatrix}
      -\alpha_i & \alpha_i & 0 & 0 \\
       \beta_i  & -\beta_i & 0 & 0 
       \\
      0 & 0 & -\alpha_i  & \alpha_i
       \\
       0  & 0& \beta_i & -\beta_i 
     \end{bmatrix}.
 \end{align*}
Note there are other (simpler) possible rate matrices which would work for this example. Then the $V$ satisfying
\begin{equation}\label{E:V_xik3}
\frac{dV^{(k)}_t}{dt}=\, \frac{D\,(V^{(k-1)}-2V^{(k)}+V^{(k+1)}) }{h^2} \,+\,Z^{(k)}_{1,4}\,f(V^{(k)}_t) \,-\,(Z^{(k)}_{2,3}+Z^{(k)}_{2,4})\,g(V^{(k)}_t)
\end{equation}
is equal in distribution to the one satisfying \eqref{E:V_xik2} and clearly fits into the framework of \eqref{eq:general_model V}-\eqref{eq:general_model Z}. 

There is another way to write \eqref{E:V_xik3} in our framework. This is because equations \eqref{eq:general_model V}-\eqref{eq:general_model Z} are over parameterized to accommodate the interpretation that $i$ represents channel type and $j$ represents channel state. If we do away with this interpretation, we can take $I =1$. Then, with the same initial conditions and rate matrix for $i =1$ as before, we let $V$ satisfy
\[\frac{dV^{(k)}_t}{dt}=\, \frac{D\,(V^{(k-1)}-2V^{(k)}+V^{(k+1)}) }{h^2} \,+\,Z^{(k)}_{1,4}\,f(V^{(k)}_t) \,-\,(Z^{(k)}_{1,1}+Z^{(k)}_{1,2})\,g(V^{(k)}_t) .\]

\end{example}

The next example shows how the we may include macroscopic variation the ion channel density by varying, in $k$, in $k$ the probability of it being present or absent in the $k^{\text{th}}$ channel. The example is only for one kind of ion channel, but a multinational random variable at each site $k$ could be used to include more kinds of ion channels.

\begin{example}[Macroscopic channel density]\rm
If we take $g \equiv 0$ in Example \ref{Ex:Switching Channels}, then we have an example with one type of single gated ion channel that has the probability of $1-p$ of being missing at any site $k$, but let us be a bit more general. Suppose that at each site $k$, the probability an ion channel exists is $p^{(k)}$, and the probability it is open, given it exists, is $q^{(k)}$. That is, in reference to Example \ref{Ex:Switching Channels}, we take $\widetilde{\xi}^{(k)}\sim \text{Ber}(p^{(k)})$ and $\xi^{(k)} \sim \text{Ber}(q^{(k)}).$ Then the first part of \eqref{E:example2} becomes 
\begin{align}
\mathbb{P}(Z^{(k)}_{1,\cdot}(0)=e_j)=
\begin{cases}
(1-p^{(k)})(1-q^{(k)})& j=1\\
(1-p^{(k)})q^{(k)}&j=2\\
p^{(k)}(1-q^{(k)})&j=3\\
p^{(k)}q^{(k)} &j=4\\
\end{cases}.\label{E:product_struct}
\end{align}
Our main theorem can be applied to such a set-up as long as for all $k$, $p^{(k)}:=p(hk) \in [0,1]$ and $q^{(k)}:=q(hk)\in[0,1]$ for some $p,q\in \mathcal{C}^{1}(\mathbb{S})$. Thus the assumption in \eqref{E:initial condtions_Discrete} is met. Note that this means the limiting PDE for $v$ may be written in the form 
\begin{align}
\begin{aligned}
\frac{\partial v_t}{dt}=\, D\Delta v_t \,+\,p(x)\,z_t(x)\,f(v_t)
\\
\frac{\partial z_t}{dt}=\alpha_1(v_t)(1-z_t)-\beta_1(v_t)z_t
\end{aligned}
\end{align}
after undoing the product structure introduced in \eqref{E:product_struct} used to pass to the deterministic limit. The initial condition for $z$ is $q$. The function $p$ reflects the channel density of the ion channels.
\label{Ex:macro_density}
\end{example}

\section{Strong convergence result}
\label{S:Result}

For each $t\in\R_{>0}$ and $(i,j)$, the function $z_{i,j}(t,\cdot )$ is continuous on the circle $\mathbb{S}$, but the random variables $\{Z^{k}_{i,j}(t)\}_{k=0}^{n-1}$ take value in the set $\{0,1\}$ which is "too discontinuous" as a function over space (the circle): even after spatial interpolation,  $\{Z^{k}_{i,j}(t)\}_{k=0}^{n-1}$ does not converge in distribution to any element in $\mathcal{C}(\mathbb{S})$ as $n\to\infty$. 
In order to overcome this, a usual approach  is to
consider the empirical measure  which is a weighted sum of Dirac-delta measure at points in the lattice $\{kh\}_{k=0}^{n-1}$ at time $t$, and then establish a convergence as $n\to\infty$. This approach was carried out to establish functional law of large numbers in 
\citep{austin2008emergence,riedler2012limit,riedler2015spatio} and functional central limit theorems in 
\citep{riedler2012limit,riedler2015spatio} in the general context of piecewise deterministic Markov processes, but these convergence results are either in probability or in distribution. 

Here, however, we instead establish strong convergence of $\{Z^{(k)}\}_{k=0}^{n-1} \to z$ as $n\to\infty$ by considering spatial local averages of $\{Z^{(k)}\}_{k=0}^{n-1}$ in space, to be precise see \eqref{Def:localspatial_Z} below. 
An advantage of using spatial local averages rather than empirical measures is that the former gives a stronger notion of convergence for each $p\in[0,1)$, to be explained below Theorem \ref{T:Main}.


\medskip
\noindent
{\bf Spatial local averages. } 
Precisely, for $h\in(0,1)$ and $p\in[0,1)$, we let 
 $\Phi_{h,p} \in \mathcal{C}(\mathbb{S};\, [0,1])$ be a smooth "bump" function that approximates the indicator function $1_{B(0,\,h^p/2)}$ in the sense that 
\begin{equation}
\Phi_{h,p}(x):=\begin{cases}
1 & \text{ if } x\in B(0,\,[h^{p-1}/2]h) \\
0 & \text{ if } x\in \mathbb{S} \setminus B(0,\,[h^{p-1}/2]h+h),
\end{cases}
\label{Def:Phi}
\end{equation} 
where $B(x,r):=\{y\in \mathbb{S}:\,d(x,y)\leq r\}$ is the closed ball centered at $x$ with radius $r$, i.e. the set of points $y$ on the circle with geodesic distance less than $r$ from the center. A precise construction for this smooth function can be found in \eqref{Def:Phi2}. 

Now, for each $t\in\R_+$ we define a smooth function $\bar{z}_t\in \mathcal{C}(\mathbb{S};\, [0,1])$ which is given by local spatial averages of $Z$ as
 \begin{equation}\label{Def:localspatial_Z}
    \bar{z}_{i,j}(t,x):=\frac{1}{N_{h,p}}\sum_{k =0}^{n-1}\Phi_{h,p}(x-hk)Z_{i,j}^{(k)}(t) \quad x\in \mathbb{S},
\end{equation} 
where
\begin{equation}
    N_{h,p}:=2[h^{p-1}/2]+1
    \label{E: Size of N}
\end{equation}
is  the number of ion channels we averaged over. We further define $\overline{Z}^{(k)}$ to be the restriction of $\bar{z}_t$ on the discrete circle. That is,
\begin{align}
\overline{Z}^{(k)}_{i,j}(t):= \bar{z}_{i,j}(t,hk) ,\quad k \in \{0,1,2,\ldots, n-1\}.
\label{Def:discretizedLA}
\end{align}
In other words, $\overline{Z}^{(k)}$ is the average of the  $2[h^{p-1}/2]+1$ nearest ion channels to compartment $k$ including $Z^{(k)}$ itself. Thus $\bar{z}$ is a smooth interpolation of these averages. The larger the $p$, the smaller the number of channels being averaged over. As $p \uparrow 1$, $\overline{Z}^{(k)}\to Z^{(k)}$ for all $k\in \{0,1,2,\ldots, n-1\}$.

\medskip

Our main result is the following almost sure error bound. Recall that $h=h_n=L/n$.
    \begin{theorem}
    \label{T:Main}
     Suppose Assumptions \ref{A:functions g and A}, \ref{A:initial condtions_PDE}, and \ref{A:initial condtions} hold. Suppose the regularity assumption \eqref{A:regularity} holds.
     Let $(V,Z)$ solve \eqref{eq:general_model V} and \eqref{eq:general_model Z}. Let $(v,z)$ solve \eqref{eq:general_PDE2} and \eqref{eq:general_PDE2z}. Then for each $p\in[0,1)$, $\rho \in (p,1)$ and $T \in (0,\infty)$, there exists a real-valued random variable $C_{\omega,T,\rho}$   
     such that we have almost surely that
     \begin{align}
        &\sup_{t \in [0,T]}\left(\max_{k\in \{0,1,2,\ldots, n-1\}}|V^{(k)}_t -v(t,hk)|+\max_{k\in \{0,1,2,\ldots, n-1\}}|\overline{Z}_{i,j}^{(k)}(t)-z_{i,j}(t,hk)|\right) \notag \\
        \leq & C_{\omega,T,\rho} (h^p \vee h^{1/2-p/2} \vee h^{\rho-p})\left|\log(h^{-1})\right|  \label{E:Main}
    \end{align}
for all  $n\in\mathbb{N}$ and $(i,j) \in I\times J$.
\end{theorem}

\medskip

Theorem \ref{T:Main} gives a strong convergence result  as $n\to\infty$ for the voltage and the spatial averages of $Z$ for every $p\in[0,1)$. It not only  provides a richer type of convergence compared to the convergence in probability as seen in previous works, but also gives an error bound between the stochastic system and the PDE. 

Note that $p=1/3$ gives the optimal bound in terms of $h$ in \eqref{E:Main}. Below we offer a heuristic explanation of why $p=1/3$ appears and why the term $\left|\log(h^{-1})\right|$ appears, by considering spatial local averages at time $t=0$. 

Fix $(i,j)$ and consider the error $\max_{k\in \{0,1,2,\ldots, n-1\}}|\overline{Z}_{i,j}^{(k)}(0)-z_{i,j}(0,hk)|$ at $t=0$. The term $\left|\log(h^{-1})\right|$ appears naturally as a result of maximizing over $h^{-1}=n$ i.i.d. random variables, according to Assumption \ref{A:initial condtions}. Each of them has the same distribution as $|\overline{Z}_{i,j}^{(0)}(0)-z_{i,j}(0,0)|$ (when $k=0$).

For each $n\in \mathbb{N}$ we fix a positive integer $N_n \in \{1,2,\ldots, n-1\}$ and define the local average
\begin{align*}
S_{N_n}:=\dfrac{1}{N_n}\sum_{k=0}^{N_n}Z_{i,j}^{(k)}(0).
\end{align*}
Then, since $z_{i,j}$ is $ \mathcal{C}^1(\mathbb{S})$,

\begin{align}
    &|S_{N_n} -(z_{i,j})(0,0)|\\ &=\dfrac{1}{N_n}\left|\sum_{k=0}^{N_n}\left(Z_{i,j}^{(k)}(0)-(z_{i,j})(0,hk)\right)+\sum_{k=0}^{N_n}\left((z_{i,j})(0,hk)-(z_{i,j})(0,0)\right)\right|
    \label{E:Avg_Heuristic}\\
    &\leq  \left|\dfrac{1}{N_n}\sum_{k=0}^{N_n}\left(Z_{i,j}^{(k)}(0)-(z_{i,j})(0,hk)\right)\right|+O(N_nn^{-1}) 
    \label{E:Avg_Heuristic2}
\end{align}

The first term goes to $0$ by the law of large numbers and is roughly  of order
$O(N_n^{-1/2})$ by the central limit theorem. This holds
for any choice of the sequence $\{N_n\}$ such that  $N_n \in \{1,2,\ldots, n-1\}$ and $\lim_{n\to\infty}N_n=+\infty$.
The best upper bound of \eqref{E:Avg_Heuristic2} among all such choices of $\{N_n\}$ is of order $O(h^{1/3})=O(n^{-1/3})$, which is obtained when $N_n=O(n^{2/3})$. 

Furthermore, an advantage of using spatial local averages over empirical measures is that the former provides a stronger notion of convergence. To explain this, note that convergence of the empirical measure under the vague topology or under the $H^{-1}$ norm in \citep{austin2008emergence} is achieved through integration with a fixed function that does not depend on $h=L/n$. 
In contrast, our spatial local averages are obtained by integrating the empirical measure with respect to a mollified indicator function over a ball of radius $O(h^p)$, which tends to zero as the number of channels $n=L/h$ tends to infinity, providing information about the ion channels at a higher spatial resolution.

We have a model which lends itself easily to numerical simulation. Thus in Section \ref{S:simulation}, we verify Theorem \ref{T:Main} numerically. Specifically, we can see numerically that \[\sup_{t \in [0,T]}\max_{k \in \{0,1,2,...,n-1\}} |V_{t}^{(k)} - v(t,hk)| \to 0\]  as $h \to 0$. However, it remains unclear that the upper bound we achieve in our theorem is sharp, specifically the exponent for the rate of convergence of $V \to v$. This is in stark contrast to other similar homogenization results where the random medium is not changing in time and the sharpness of the exponent is resolved clearly in numerical experiments; see for example \citep{mcginnis2023}. On the other hand, the sharpness of the rate convergence for the local averages $\overline{Z}$ to the deterministic $z$ is explained by the heuristic argument above.

\section{Proof of strong convergence}
\label{S:Proof}

In the sections below, constants, generically denoted $C$ never depend on $h, k, n$ or $t$ and are not random. Constants denoted $C_\omega$ also do not depend on $h, k, n$ or $t$, but are random variables. Certain constants may also depend on $T$. In these cases, the constants are denoted by $C_T$ or $C_{\omega,T}$.

We shall give the details of our proof first for the simpler model  \eqref{E:V}-\eqref{E:Zkl} in Example \ref{Eg:AllenCahnType}. This is a nontrivial task because this system for each $n\in\mathbb{N}$ has an infinite dimensional state space. The argument can be extended to the more general model in \eqref{eq:general_model V}-\eqref{eq:general_model Z}, since the double sum there consists of analogous terms, but we only carry out the argument for the toy model so the reader does not get lost in the notation.

Without loss of generality we take $L=1$. 
By Assumptions \ref{A:functions g and A} and \ref{A:initial condtions_PDE} and the regularity assumption \eqref{A:regularity},
\begin{equation}\label{A:regularity_toy}
f,g\in\mathcal{C}^{1,1}(\mathbb{R}),\;\alpha,\beta\in \mathcal{C}^{1}(\mathbb{R}),\;z_0\in  \mathcal{C}^1(\mathbb{S};\,[0,1]),\;v_0\in \mathcal{C}^2(\mathbb{S}).  
\end{equation}
We also enforce Assumption \ref{A:initial condtions_Discrete} for the stochastic model. That is,  we suppose 
that, for each $n\in\mathbb{N}$ (hence for each $h=h_n=1/n$), 
$V^{(k)}(0)=v_0(hk)$ for $k\in\{0,1,\cdots,n-1\}$ 
and
$\{Z^{(k)}_0\}_{k=0}^{n-1}$ are independent such that   $Z^{(k)}_0\sim {\rm Bernoulli}(z_0(kh))$, where
${\rm Bernoulli}(q)$ denotes the Bernoulli distribution (being 1 with probability $q$ and being 0 with probability $1-q$). 

\bigskip

We shall show that $\left(V,\,Z\right)$ described by equations \eqref{E:V}-\eqref{E:Zkl} is well approximated by a pair $(v,z)$ that solves the PDE \eqref{eqn_V_pde}-\eqref{eqn_Z_pde} as $n\to\infty$.

Recall from \eqref{Def:localspatial_Z} that
 $\bar{z}_t\in \mathcal{C}(\mathbb{S};\, [0,1])$ is defined by
 \begin{equation*}
    \bar{z}_t(x):=N_{h,p}^{-1}\sum_{k =0}^{n-1}\Phi_{h,p}(x-hk)Z_t^{(k)} \quad x\in \mathbb{S},
\end{equation*} 
and from \eqref{Def:discretizedLA} that
\begin{align*}
\overline{Z}^{(k)}_t:= \bar{z}_t(hk) ,\quad k \in \{0,1,2,\ldots, n-1\}.
\end{align*}

Because $\bar{z}$ is an average, it has better regularity than $\Phi_{h,p}$. Precisely, by \eqref{eq regularity of phi}, there exists a constant $C\in(0,\infty)$ such that for $j \in\{1 ,2\}$ and any  $(v_0,z_0)\in \mathcal{C}(\mathbb{S})\times \mathcal{C}(\mathbb{S};[0,1])$,
\begin{align}
\sup_{t\in\R_+}\|\partial_x^{j}\bar{z} \|_{L^\infty(\mathbb{S})} \leq Ch^{1-p-j} \quad \text{ for all }h,p\in(0,1).
\label{Inq:Regularit z}
\end{align}
Of course for $j=0$, we have the bound $\|\bar{z}\|_{L^\infty([0,T]\times \mathbb{S})} \leq 1$.

Note that the local averages consist of sums of independent Poisson processes. We want to compare these to their means as we did with the random variables in \eqref{E:Avg_Heuristic}.  In order to establish a  uniform rate for the law of large 
numbers on a growing number of sums of Poisson processes, we shall supply the following Lemma.
 
 \begin{lemma}
\label{Lemma_spatial_law}{(Error bound for Poisson Processes)}
Let $\gamma\in(0,\infty)$ be a constant and  $\{n_i\}_{i\geq 1}$ be a sequence of natural numbers such that $\sum_{i=1}^{\infty} n_{i}^{-\gamma}<\infty$. Let $\{\mathcal{N}_{k,i}\}_{1\leq k\leq n_i,\,i\geq 1}$ be a collection of unit rate Poisson processes such that $\{\mathcal{N}_{k,i}\}_{1\leq k\leq n_i}$ are independent for each $i\geq 1$,  and $\{\tau_{k,i}\}_{1\leq k\leq n_i,\,i\geq 1}$ be a collection of non-decreasing random functions $\tau_{k,i}: [0,\infty) \to [0,\infty)$, both defined on the same probability space. Suppose 
for any $T\in(0,\infty)$, there exists a constant $\tau_T\in (0,\infty)$ such that $\sup_{t\in[0,T]}\tau_{k,i}(t) \leq \tau_T$ for all $1\leq k\leq n_i$ and $i\geq 1$ almost surely.
Then
there exists a real-valued random variable $\Gamma_{\omega,T}$ such that 
\begin{equation}\label{E:spatial_law}
\sup_{t \in [0,T]}\left|\sum_{k=1}^{n_i}\Big(\mathcal{N}_{k,i}(\tau_{k,i}(t)) - \tau_{k,i}(t)\Big)\right| \leq \Gamma_{\omega,T} \vee 6\gamma n_i^{1/2}\log(n_i)\qquad \text{for }i\geq 1\quad \P-a.s.
\end{equation}

\end{lemma}

The proof of Lemma \ref{Lemma_spatial_law} will be given in the Appendix.
Here is how Lemma \ref{Lemma_spatial_law} will be used in this paper. Let $p \in [0,1)$ be fixed. Below, in \eqref{E:Uses_Lemma1_2}, we require that there exists an almost surely finite $C_{\omega,T}$ such that
for all $n \in \mathbb{N}$, 
\begin{align}
\sup_{k\in\{1,2,\cdots,n\}}\sup_{t \in [0,T]}\left|\overline{Z}^{(k)}-N_{h,p}^{-1}\sum_{j \in \mathbb{S}_n}\Phi_{h,p}(hk-hj)(Z_0^{(k)}+\tau^+_{k,n}-\tau^+_{k,n})\right| \leq C_{\omega,T} n^{p/2-1/2}\log(n)
\end{align}
with probability one.
Here $\tau_{k,n}^+(t):=\int _0^t\alpha(V^{(k)})(1-Z^{(k)})ds$ and $\tau_{k,n}^-(t):=\int_0^t\beta(V^{(k)})Z^{(k)}ds$. Thus  there exists a constant $\tau=\tau_T>0$ such that $\sup_{t \in[0,T]}\tau^+_{k,n}\vee \tau^-_{k,n} \leq \tau$ 
for all $n \in \mathbb{N}$ and $k \in \{1,2,3,\ldots, n\}$. We take as the sequence for the $n_i$ in Lemma \ref{Lemma_spatial_law} the following sequence:  $N_{1,p},N_{1/2,p},N_{1/2,p},N_{1/3,p},N_{1/3,p},N_{1/3,p},\ldots$. Recalling from \eqref{E: Size of N} that $N_{h,p}\sim O(n^{1-p})$, we are required to choose a $\gamma$ such that 
\[\sum_{n=1}^\infty \sum_{k=1}^{n}n^{-\gamma (1-p)} =\sum_{n=1}^\infty n^{1-\gamma (1-p)}<\infty.\]
Thus for our application, when $n \in \{1,2,3,\ldots,\}$ we can take any $\gamma $ such that $\gamma >2/(1-p)$. 

Let us now give an interpretation of the $C_{\omega,T,\rho}$ that appears Theorem \ref{T:Main}. Rather than considering $n\in \{1,2,3,\ldots\}$, we could also consider a more generic sequence $n \in\{n_1,n_2,n_3,\cdots\}$  with $\gamma$ such that 
\[\sum_{j=1}^\infty n_j^{1-\gamma (1-p)}<\infty.\] (Here $n_i$ is being used for the number of ion channels, not as the sequence in Lemma \eqref{Lemma_spatial_law}). Then the requirement for $\gamma$ depends on the generic sequence.  We can regard \eqref{E:V} and \eqref{E:Zkl} as a sequence of models indexed by $i$ with compartment number parameter $n \in \{n_1,n_2,n_3, \ldots \}$. Physically we may imagine this sequence as a given infinite sample coming from a population whose behavior is described by the model. For the $i$-th sample, we have $n_i$ spatial averages, that is $\{\overline{Z}^{(k)}\}_{k=1}^{n_i}$, each with roughly $n_i^{1-p}$ addends, that we hope to converge to $z$ as $i \to \infty.$ For this sequence of samples we obtain a finite random variable $\Gamma_{\omega,T}$ as long as $\lim_{i \to \infty}n_i\to \infty$ not too slowly. 

\medskip
\noindent
{\bf Intermediate PDE. } 
We derive strong convergence of $V \to v$ by way of an intermediate problem. We first show that $V^{(k)}$ is approximated by $\bar{v}(hk)$, where $\bar{v}$ solves the intermediate PDE
\begin{equation}\label{eqn_intermediate_avg}
\frac{\partial\bar{v}_t(x)}{\partial t}=\, D\Delta \bar{v}_t \,+\,\bar{z}_t(x)\,f(\bar{v}_t) -g(\bar{v}_t), 
\end{equation}
where $\bar{z}$ is defined in \eqref{Def:localspatial_Z}.

We rewrite the PDE using Duhammel's formula as 
\begin{align}
    \bar{v}_t(x)=e^{D \Delta t}v_0+\int_0^{t}e^{D \Delta (t-s)}(\bar{z}_sf(\bar{v}_s)-g(\bar{v}_s))ds,
    \label{eq:Duhml}
\end{align}
where we define the fundamental operator acting on functions $w \in L^{\infty}([0,T]\times\mathbb{S})$ as
\begin{align}\label{E:barv_Duhamel}
e^{D \Delta t}w(t,x) := \int_{\mathbb{S}}p^{\mathbb{S}}(2Dt,x,y)w(t,y)dy
\end{align}
where $p^{\mathbb{S}}$ is defined in \eqref{densityBM_S}. 

Precisely,
$\bar{v}$ is the unique element in $\mathcal{C}(\R_+\times \mathbb{S})$ that satisfies \eqref{eq:Duhml} for all $(t,x)\in \R_+\times \mathbb{S}$. This unique element exists by a fixed point argument, because $\bar{z}_t\in \mathcal{C}(\mathbb{S};\, [0,1])$ are bounded for all $t\in\R_+$ and $\{f,g\}$ satisfies Assumption \eqref{A:regularity_toy}. A Gronwall-type argument also gives that there exist constants $C_1,\,C_2\in(0,\infty)$ such that for all $T\in(0,\infty)$, 
\begin{equation}
    \sup_{h\in (0,1)} \sup_{p \in (0,1)}\|\bar{v} \|_{L^\infty([0,T]\times \mathbb{S})} \leq C_1e^{C_2T}.
    \label{inq_v_bar}
\end{equation}

First, we obtain a regularity estimate of $\bar{v}$ using  \eqref{Inq:Regularit z}.

\begin{lemma}\label{L:partialx_barv}
 Let $v_0 \in \mathcal{C}^2(\mathbb{S})$, and suppose $f,g \in \mathcal{C}^1(\mathbb{R})$ and are globally Lipschitz. Then for each $T \in (0,\infty)$ there exists constants $C_1,C_2$ independent of $T$ and a constant $C_T$ such that for all $p\in[0,1)$, and $h\in(0,1)$,  
\begin{align}
\|\partial_x^{j}\bar{v} \|_{L^\infty([0,T]\times\mathbb{S})}
    \leq
    \begin{cases}
    C_1e^{C_2T} &j \in\{0,1\}
    \\
    h^{-p}C_T & j = 2    \end{cases}
\end{align}
and
\begin{equation}
\|\partial_t \bar{v} \|_{L^\infty([0,T]\times\mathbb{S})} \leq h^{-p}C_T.
    \label{inq_time_derv}
\end{equation}
\end{lemma}

\begin{remark}
In many situations with more specific information about $f$ and $g$, the bound can be made smaller in terms of $T$. For example, in cases where $\bar{v}$ is bounded independent of $T$, $C_T$ depends algebraically on $T.$
\end{remark}

\begin{proof}
The case for $j=0$ follows from the discussion before \eqref{inq_v_bar}. The bound for $\partial_t\bar{v}$ then follows from \eqref{eqn_intermediate_avg} and the case when $j=2$ as well as the fact that $f$ and $g$ are Lipschitz. We now address the cases of $j=1,2.$ 

For any $T\in(0,\infty)$ and $w\in L^\infty([0,T]\times \mathbb{S})$ we let
\begin{align}
  W(t,x):= \int_{0}^te^{D \Delta (t-s)}w(s,x)ds,\quad (t,x)\in [0,T]\times \mathbb{S}.
\end{align}

Suppose, furthermore, $\partial_x^{j-1}w$ exists and lies in $L^\infty([0,T]\times \mathbb{S})$. Then by integration by parts and by using the symmetry $\partial_xp^{\mathbb{S}}(2Dt,x,y)=-\partial_yp^{\mathbb{S}}(2Dt,x,y)$, we obtain that
\begin{align}
  \partial_x^{j}W(t,x):= \int_{0}^t\int_{\mathbb{S}}\partial_xp^{\mathbb{S}}(2D(t-s),x,y) \partial ^{j-1}_yw(t,y)dyds.
  \label{Def:W}
\end{align}
Applying Young's convolution inequality to the right hand side and the heat kernel estimate 
$$\sup_{x,y\in\mathbb{S}}\int_{\mathbb{S}}\Big|\partial_xp^{\mathbb{S}}(2Dt,x,y) \Big|\,dy\leq \frac{C}{\sqrt{t}} \quad \text{for all }t\in(0,\infty),$$
where $C\in(0,\infty)$ is a constant, we obtain
\begin{align}
\|\partial_x^{j}W\|_{L^\infty([0,T]\times \mathbb{S})} \leq CT^{1/2}\|\partial_x^{j-1}w\|_{ L^\infty([0,T]\times \mathbb{S})}.
\end{align}

In view of \eqref{eq:Duhml} and \eqref{E:barv_Duhamel}, we shall take $w(s,x)=\bar{z}_s(x)f(\bar{v}_s(x))-g(\bar{v}_s(x))$. In order that 
$\partial_x^{j-1}w\in L^\infty([0,T]\times \mathbb{S})$ for this particular $w$, we need that the $(j-1)^{\text{th}}$ derivative of $f$ and $g$ are bounded in a compact interval. Specifically, letting $I_T:=[-\|\bar{v}\|_{L^\infty([0,T]\times \mathbb{S})},\|\bar{v}\|_{L^\infty([0,T]\times \mathbb{S})}]$, it is enough to take $f,g \in \mathcal{C}^{j-1}(\mathbb{R}) \subseteq \mathcal{C}^{j-1}(I_T)$. This gives us an $L^\infty([0,T]\times\mathbb{S})$ bound on the derivatives of $\bar{v}$ in terms of $\|\bar{v}\|_{L^\infty([0,T]\times \mathbb{S})}$ after differentiating \eqref{eq:Duhml} to the desired degree, using \eqref{Inq:Regularit z}, and using the assumption that $v_0 \in \mathcal{C}^2(\mathbb{S})$. Taking $\partial_x^{j}$  of \eqref{eq:Duhml} with $j=1$  first proves the case for $j=1$. Then the $j=2$ case is proved using the $j=1$ case.

The declared estimate for the time derivative $\|\partial_t \bar{v} \|_{L^\infty(\mathbb{S})}$ follows from the case $j=2$.
\end{proof}

It is in fact possible to establish higher regularity for $\bar{v}$ in the sense of Hölder continuity using interpolation theory for operator norms. Specifically, after introducing the interpolation parameter $\rho$ and interpolating between $\mathcal{C}^1$ and $\mathcal{C}^2$ norms of the heat kernel, we obtain the following Lemma.
\begin{lemma}
\label{L:dx^2_barv}
Let  $f,g \in \mathcal{C}^1(\mathbb{R})$ be globally Lipschitz and $v_0 \in \mathcal{C}^2(\mathbb{S})$. For each $\rho \in (0,1)$ and ecah $T\in(0,\infty)$, there exists constants $C_{\rho,T}$ such that 
\begin{align*}
    \sup_{t \in [0,T]}| \partial_x^2\bar{v}(x_1)-\partial_x^2\bar{v}(x_2)| \leq h^{-p}C_{\rho,T} |x_1-x_2|^\rho 
\end{align*}
for all $x_1,x_2 \in \mathbb{S}$, $p\in[0,1)$, and $n\in \mathbb{N}$.
\end{lemma}
\begin{proof}
The proof relies on \eqref{Inq:Regularit z} and is similar to the proof of Lemma \ref{L:partialx_barv}. However, it requires a more delicate heat kernel estimate. Namely we will use the fact that for all $\rho\in (0,1)$, there exists a constant $C_\rho\in(0,\infty)$ such that
\begin{equation}\int_{\mathbb{S}}\left|\partial_xp^{\mathbb{S}}(2Dt,x_1,y)-\partial_xp^{\mathbb{S}}(2Dt,x_2,y)\right|dy \leq \dfrac{C_\rho|x_1-x_2|^{\rho}}{t^{1/2+\rho/2}}
\end{equation}
for all $x_1,x_2\in \mathbb{S}$ and $t\in (0,\infty)$. Now we use the estimate with \eqref{Def:W}  in the following way:
\begin{align}
\begin{aligned}
  &\left|\partial_x^{j}W(t,x_1)- \partial_x^{j}W(t,x_2)\right| \\&\leq \int_{0}^t\int_{\mathbb{S}}\left|\partial_xp^{\mathbb{S}}(2D(t-s),x_1,y) 
  -\partial_xp^{\mathbb{S}}(2D(t-s),x_2,y)\right|\left|
  \partial^{j-1}_yw(s,y)\right|dyds
  \\
  &\leq C_{\rho}|x_1-x_2|^{\rho}\,t^{1/2-\rho/2}\|\partial_x^{j-1}w\|_{L^{\infty}([0,T]\times\mathbb{S})}.
 \end{aligned}
\end{align}
Now the proof follows from that of Lemma \ref{L:partialx_barv}. 

\end{proof}
Lemma \ref{L:partialx_barv} and Lemma \ref{L:dx^2_barv} give us the following gradient and error estimates. There exists a constants $C_1$, $C_2$, $C_T$ and $C_{T,\rho}$ such that for all $k \in \{0,1,2,\ldots,n-1\}$ we have for all $h,p\in(0,1)$, $\rho\in(p,1)$ and $T\in(0,\infty)$,
\begin{align}
    \sup_{t \in [0,T]}|\bar{v}(hk)-\bar{v}(hk-h)| \leq & C_1e^{c_2T}h
    \label{eq one difference}
    \\
     \sup_{t \in [0,T]}|\bar{v}(hk+h)-2\bar{v}(hk)+\bar{v}(hk-h)| \leq & C_Th^{2-p}
     \label{eq two difference}
     \\
     \sup_{t \in [0,T]}\left|\dfrac{\bar{v}(hk+h)-2\bar{v}(hk)+\bar{v}(hk-h)}{h^2}-\partial^2_x\bar{v}(hk)\right| \leq & C_{T,\rho} h^{\rho-p}.
     \label{eq 4 difference}
\end{align}

\medskip
\noindent
{\bf Approximate Ansatz and Bounds for $\chi$. } 
Next, we show that $V^{(k)}$ is well-approximated by $\bar{v}(hk)$ where $\bar{v}$ solves the intermediate PDE
\eqref{eqn_intermediate_avg}. 

To do this, we shall construct an approximate ansatz $\overline{V}$ for $V=(V^{(k)})_{k=0}^{n-1}$ that enables us to quantify the error between $V$ and $\bar{v}$.
The approximate ansatz $\overline{V}$ for $V$ is given by 
\begin{equation}
\overline{V}^{(k)}:=\bar{v}(hk)+h^2\chi^{(k)}f(\bar{v}(hk)) \quad \text{for } k\in \{0,1,\cdots, n-1\},
\end{equation}
where $f$ is the same function in our model \eqref{E:V}, and, for each $t\in \R_+$, $\{\chi^{(k)}_t\}_{k=0}^{n-1}$  solves 
\begin{equation}
D(\chi^{(k+1)}_t-2\chi^{(k)}_t+\chi^{(k-1)}_t)= \overline{Z}_t^{(k)}-Z_t^{(k)}.
\label{eqn_chi}
\end{equation}
Note the left hand side is a discrete Laplacian scaled by $D$ with periodic boundary conditions. We can view the left hand side of \eqref{eqn_chi} as an $n \times n$ matrix of rank $n-1$.  Since the column sums of this matrix are $0$, a solution $\{\chi^{(k)}_t\}_{k=0}^{n-1}$ exists because the sum of the right hand side of \eqref{eqn_chi} over $k\in\{0,1,\ldots, n-1\}$ is zero. 

The idea of an approximate ansatz is that it solves the equation for $\bar{v}$ but with a residual that is small in $h$. This allows us to use Gronwall's inequality to obtain the approximation error between $\overline{V}$ and $V$. If we can prove a generous enough bound for $\chi$, then we will have that $\overline{V}^{(k)}$ can be approximated by $\bar{v}(hk)$. Before delving into the details of the approximation error, we need to understand $\chi$ better.

Note that when averaged over all $k$, the right hand side of \eqref{eqn_chi} is $0$. Therefore, we can indeed solve for $\chi$. Furthermore, the solution $\chi$ depends on $p$, although not explicitly written. When $p=1$, $\chi\equiv0$ is a possible solution as the right hand side vanishes. For the other extreme, $p=0$, the right hand side of \eqref{eqn_chi} is a generic vector in the image of the discrete Laplacian, i.e. a vector with no constant component. Standard matrix inequalities then give that there exists a constant $C$ such that for all $n\in \mathbb{N}$, 
$$ \sup_{t \in \R_+} \max_{k\in\{0,1,\cdots,n-1\}}|\chi^{(k)}| \leq Ch^{-2}.$$ 

This estimate is too coarse for our purposes but can be improved.
In fact, there exists a solution $\chi$ to equation \eqref{eqn_chi} and a constant $C$ (when $D=1$, $C=1/4$ works, see \eqref{Bound_l1_nu} and \eqref{bound_l1_difnu}) such that for all $p \in [0,1]$ and $n\in\mathbb{N}$,
\begin{align}
    \sup_{t \in \R_+} \max_{k \in \{0, \ldots, n-1\}}|\chi^{(k)}| \leq Ch^{-2+2p} \quad \text{and} \quad \sup_{t \in \R_+}\max_{k \in \{0, \ldots, n-1\}}|\chi^{(k+1)}-\chi^{(k)}| \leq Ch^{-1+p}.
    \label{inq_bound_chi}
\end{align}
One may check this by directly computing a solution $\chi$. See Section \ref{Ap:chi_bounds} in the appendix. Note that the constant $C$ can be chosen uniformly for all time $t\in\R_+$ because the $Z^{(k)}$ are bounded uniformly for all time.

We did not need to use anything about the stochasticity of $\chi$ in the preceding discussion. However, since $\chi$ involves sums of Poisson processes, and therefore is a multivariate jump process, we need some preliminary bounds on how often it jumps.

Recall the solution $\chi_t$ to equation \eqref{eqn_chi} for $t\in \R_+$ is a  jump process in $\R^n$, which jumps whenever $Z^{(k)}$ jumps for any $k$.  Suppose that there are $J=J_n$ many jumps 
for the process  $\{Z_t^{(k)}\}_{k=0}^{n-1}$ in the interval $[0,T]$, and that the jump times are $t_1,t_2,...,t_{J_n}$. Let $t_0=0$ and $t_{J_n+1}=T$, so that $0<t_1<t_2<...<t_{J_n}<t_{J_n+1}=T$  almost surely.

Recall that the apriori bound \eqref{E:aprioriBoundV} for $V$ is uniform in $n\in\mathbb{N}$. This, together with our assumption that $\alpha$ and $\beta$ are bounded on compact intervals (Assumption \ref{A:functions g and A}), implies that $\alpha(V)$ and $ \beta(V)$ are uniformly bounded for $h\in(0,1)$.

Hence $J_n$ is bounded above by the total number of jumps of $2n$ independent unit-rate Poisson processes before time $\widetilde{T}$, where
$$\widetilde{T}:=T\,( \|\alpha(V)\|_{L^{\infty}([0,T]\times \mathbb{S})} \vee \|\beta(V)\|_{L^{\infty}([0,T]\times \mathbb{S})}).$$ 
Therefore, by the functional LLN for Poisson processes (c.f. \citep[Theorem 1.2 and Lemma 1.3]{anderson2015stochastic}),
with probability one there exists an $\R$-valued random variable $C_{\omega,T}$ such that 
\begin{align}\label{inq_nu}
|J_n| \leq C_{\omega,T} n=C_{\omega,T} {h}^{-1} \quad \text{ for all }n\in\mathbb{N}.
\end{align}
Note that the boundedness assumption of $\alpha$ and $\beta$ is used in \eqref{inq_nu}.

Furthermore, for all $q\in \mathbb{N}$, there exists a constant $C_{q}\in(0,\infty)$ such that
\begin{align}\label{inq_nu_E}
\E[J_n^q]\leq C_q\,\left[1+ (2n \,\widetilde{T})^q\right] \quad \text{ for all }n\in\mathbb{N} \text{ and }T\in(0,\infty). 
\end{align}
This inequality could be useful to obtain moment bounds for the error on the left hand side of \eqref{E:Main}, which would not be pursued in this paper.

Further note that for each $i \in\{1,2,\ldots,J_n\}$, we have almost surely that there is exactly one $k_i$ for $k_i \in\{0,1, \ldots,n-1\}$ such that $Z^{(k_i)}$ jumps at $t_i$. For all other $k \in \{0,\ldots,i-1,i+1, \ldots n-1\}$, $Z_t^{(k)}$ does not jump at $t_i$. The amount that $\chi_t^{(k_i)}$ jumps is given by $\lim_{t\to  t_i^-}|\chi_t^{(k_i)}-\chi_{t_i}^{(k_i)}| $. One can show by directly computing $\chi_t$ (see Section \ref{Ap:chi_bounds} in the appendix) that for any $i \in \{1,2,\ldots,J_n\}$, we have 
\begin{align}
\max_{k \in \{0,\ldots, n-1\}}\lim_{t\to t_i^-}|\chi_t^{(k)}-\chi_{t_i}^{(k)}| \leq Ch^{-1+p}.
\label{eqn_jump_amt}
\end{align}
The constant again does not depend on $p$ (when $D=1$,  $C=1/8$ works by \eqref{bound_max_nu}).

Now recall the approximate ansatz 
\begin{align}
\overline{V}^{(k)}:=\bar{v}(hk)+h^2\chi_t^{(k)}f(\bar{v}(hk)).
\label{def_of_ansatz}
\end{align}
It is approximate in the following sense. For each interval in $\{(t_j,t_{j+1})\}_{j=0}^{J_n}$, over which $\chi$ is constant, let
\begin{align}
\text{res}_{t}^{(k)}:=\,&\dfrac{d\overline{V}^{(k)}}{dt}-\left( \frac{D\,(\overline{V}^{(k-1)}-2\overline{V}^{(k)}+\overline{V}^{(k+1)}) }{h^2} \,+\,Z^{(k)}_s\,f(\overline{V}_s^{(k)}) -g(\overline{V}_s^{(k)})\right).
\label{eqn_res_def}
\end{align}
Although this is a definition, we can view it as a system which $\overline{V}$ solves over each of the time intervals $(t_j,t_{j+1})$. In the next lemma, we prove an estimate on the residual. 

\begin{lemma}{(Consistency)}\label{L:consistency} Let $\tau_T:=[0,T]\setminus \{t_j\}_{j=0}^{J_n+1} =\bigcup_{j=1}^{J_n+1}(t_{j-1},t_j)$ be the complement of the jump times of $Z$. Suppose $f,g\in \mathcal{C}^{1,1}(\mathbb{R})$ and are globally Lipschitz. Then for all 
$p \in[0,1)$, $\rho \in (p,1)$, and $T \in (0,\infty)$, there exists a constant  $C_{T,\rho}$ such that the inequality
\begin{equation}
\sup_{t \in \tau_T}\max_{k \in \{0,\ldots,n-1\}}|\text{\emph{res}}_{t}^{(k)}| \leq C_{T,\rho}(h^{p }\vee h^{\rho-p}).
\label{ode_for_res}
\end{equation}
holds for all  $n \in \mathbb{N}$ almost surely, where $h=1/n$.
\label{lemma_bound_res}
\end{lemma}
\begin{remark}
In certain applications where there is some constant $B$ such that $\sup_{t \in [0,T]}\max_{k \in {0, \ldots, n-1}}\left|V^{(k)}\right| \leq B, $ it will be the case that $\bar{v}$ inherits this bound, and so the $T$ dependency in \eqref{ode_for_res} will be algebraic.
\end{remark}
\begin{proof}
Write $\tau=\tau_T$ for simplicity.
We calculate the residual using the definition of the ansatz in \eqref{def_of_ansatz}. Using a discrete version of the product rule:
    \begin{equation}\begin{aligned} \chi^{(k+1)}f(\bar{v}(hk+h))-2\chi^{(k)}f(\bar{v}(hk))+\chi^{(k-1)}f(\bar{v}(hk-h)))= 
    \\\chi^{(k+1)}(f(\bar{v}(hk+h))-2f(\bar{v}(hk))+f(\bar{v}(hk-h)))\\
    +(\chi^{(k+1)}-\chi^{(k-1)})(f(\bar{v}(hk))-f(\bar{v}(hk-h)))
     \\+(\chi^{(k+1)}-2\chi^{(k)}+\chi^{(k-1)})f(\bar{v}(hk)).
    \end{aligned}
    \end{equation}
    Since $\bar{v}$ solves \eqref{eqn_intermediate_avg} and $\chi$ solves \eqref{eqn_chi}, we find that 
    \begin{equation}
    \begin{aligned}
    \text{res}^{(k)}=-\left(\frac{D\,(\bar{v}_t(hk+h)-2\bar{v}_t(hk)+\bar{v}_t(hk-h)}{h^2}-D\Delta \bar{v}_t(hk)\right)
    \\-\chi^{(k+1)}(f(\bar{v}(hk+h))-2f(\bar{v}(hk))+f(\bar{v}(hk-h)))\\
    -(\chi^{(k+1)}-\chi^{(k-1)})(f(\bar{v}(hk))-f(\bar{v}(hk-h)))
    \\-Z^{(k)}\left[f\left(\bar{v}(hk)+h^2\chi^{(k)}f(\bar{v}(hk))\right)-f(\bar{v}(hk))\right]\\
    +\left[g\left(\bar{v}(hk)+h^2\chi^{(k)}f(\bar{v}(hk))\right)-g(\bar{v}(hk))\right] \\
    + h^2 \chi^{(k)}f'(\bar{v}(hk))\partial_{t}\bar{v}(hk).
    \end{aligned}
    \end{equation}
    Using the triangle inequality, we have 
     \begin{equation}
    \begin{aligned}
    \sup_{t \in \tau}|\text{res}^{(k)}|\leq \sup_{t \in \tau}\left|\frac{D\,(\bar{v}_t(hk+h)-2\bar{v}_t(hk)+\bar{v}_t(hk-h)}{h^2}-D\Delta \bar{v}_t(hk)\right|
    \\+\sup_{t \in \tau}\left|\chi^{(k+1)}(f(\bar{v}(hk+h))-2f(\bar{v}(hk))+f(\bar{v}(hk-h)))\right|\\
    +\sup_{t \in \tau}\left|(\chi^{(k+1)}-\chi^{(k-1)})(f(\bar{v}(hk))-f(\bar{v}(hk-h)))\right|
    \\+\sup_{t \in \tau}\left|Z^{(k)}\left[f\left(\bar{v}(hk)+h^2\chi^{(k)}f(\bar{v}(hk))\right)-f(\bar{v}(hk))\right]\right|\\+\sup_{t \in \tau}\left|\left[g\left(\bar{v}(hk)+h^2\chi^{(k)}f(\bar{v}(hk))\right)-g(\bar{v}(hk))\right]\right|
    \\
    +\sup_{t \in \tau}\left| h^2 \chi^{(k)}f'(\bar{v}(hk))\partial_{t}\bar{v}(hk) \right|
    .
\end{aligned}
    \end{equation}
    The first term is bounded by \eqref{eq 4 difference}. The second term is bounded by \eqref{inq_bound_chi} and, using that $f'$ is locally Lipschitz continuous, \eqref{eq one difference}  and \eqref{eq two difference}. The third term is bounded using the second inequality in \eqref{inq_bound_chi}, and  using that $f'$ is continuous, \eqref{eq one difference}. For the fourth and fifth terms, we use the Lipschitz continuity of $f$ and $g$ as well as  the first inequality in \eqref{inq_bound_chi}. The sixth term is bounded using the first inequality in \eqref{inq_bound_chi} and \eqref{inq_time_derv}. 
\end{proof}
With the preceding estimate on the residual, we can now carry out the  argument based on Grownall's inequality. Recall that $h=h_n=1/n$.

\begin{lemma}{(Stability and Spatial Homogenization)}\label{lemma_Homogenization}
Suppose  $f,g\in \mathcal{C}^{1,1}(\mathbb{R})$ and are globally Lipschitz. Suppose that  $\bar{v}_0 \in \mathcal{C}^2(\mathbb{S})$ and $V_0^{(k)}=\bar{v}_0(hk)$ for all $k\in\{0,1,\cdots, n-1\}$. Then for each $p \in [0,1)$, $\rho \in (p,1)$ and each $T \in (0,\infty)$, there exists a constant $C_{T,\rho}$ and random variable $C_{T,\omega}$ that is finite almost surely such that for all $n \in \mathbb{N}$ we have almost surely
\begin{equation}\label{E:Homogenization}
    \sup_{t \in [0,T]}\max_{k \in \{0,\ldots,n-1\}}|\bar{v}_t(hk)-V^{(k)}_t| \leq C_{T,\omega} C_{T,\rho}(h^{p}\vee h^ {\rho -p}).
\end{equation}
\end{lemma}

\begin{remark} The lemma says that if we take local averages of $Z$, and supply this to the intermediate PDE for $\bar{v}$, we can get an approximations for the behavior of $V$ by that of $\bar{v}$, with the best error bound on the right of \eqref{E:Homogenization} when $p=1/2$. This approximation may alternatively be established by using a local limit theorem of the simple random walks, as in \citep[Appendix]{cox2002stepping} or \citep[Section 5]{chen2017hydrodynamic}. However,
here we demonstrate a method that requires only a coarser estimate on the heat kernel.
\end{remark}

\begin{proof}[Proof of Lemma \ref{lemma_Homogenization}]
Throughout the proof we use $\|\cdot\|$ to denote $\max_{k \in \{0, \ldots, n-1\}}|\cdot|$.
Note that the system of ODEs defined by
\begin{equation} 
\dfrac{dU_t^{(k)}}{dt}=\dfrac{D(U_t^{(k-
1)}-2U_t^{(k)}+U_t^{(k+
1)})}{h^2}
\end{equation}
has fundamental solution, denoted as $e^{D\Delta_h t}$, where $(D\Delta_hU)^{(k)}:=\frac{D(U^{(k-
1)}-2U^{(k)}+U^{(k+
1)})}{h^2}$. Let $\mathscr{V}^{(k)}_t:=V^{(k)}_t-\overline{V}_t^{(k)}$, $\mathscr{F}^{(k)}_t:=f(V^{(k)}_t)-f(\overline{V}^{(k)}_t)$ and $\mathscr{G}^{(k)}_t:=g(V^{(k)}_t)-g(\overline{V}^{(k)}_t)$.
Now we view \eqref{eqn_res_def} as an equation $\overline{V}$ solves for each time interval $(t_{j},t_{j+1})$ with the initial condition at $t_j$ given by $\overline{V}(t_{j})$ as defined in the ansatz \eqref{def_of_ansatz}. Thus, we rewrite \eqref{eqn_res_def} and \eqref{E:V} using the variation of constants formula for $t \in (t_j,t_{j+1})$ to obtain
\begin{equation}
\|\mathscr{V}_t\|\leq 
\|e^{D \Delta_h t}\mathscr{V}_{t-t_j}\|+
\int_{t_j}^{t}\|e^{D \Delta_h (t-t_j-s)}\left(Z_s
\mathscr{F}_s
-\mathscr{G}_s
+\text{res}_s\right)\|ds.
\end{equation}
Since $f$ and $g$ are Lipschitz we find that
\begin{equation}
\begin{aligned}
\|\mathscr{V}_t\| \leq &\,
\|\mathscr{V}_{t_j}\|+
\int_{t_j}^{t}\left(C\|\mathscr{V}_s\|+\|
\text{res}_s\|\right)ds
\\ 
\leq &\, \|\mathscr{V}_{t_j}\|+(t_{j+1}-t_j)\sup_{t \in (t_j,t_{j+1})}\|
\text{res}_t\|+\int_{t_j}^{t}C\|\mathscr{V}_s\|ds.
\end{aligned}
\end{equation}
Let $\mathscr{V}_{{t_j}^-}=\lim_{t \to t_j^-}\mathscr{V}_t.$ Since $\bar{v}$ and $V$ are continuous in time, 
$\mathscr{V}_{t_j}-\mathscr{V}_{{t_j}^-}= h^2(\chi_{{t_j}^{-}}-\chi_{t_j})f(\bar{v}_{t_j}).$  Thus we obtain
\begin{equation}
\|\mathscr{V}_t\| \leq \, h^2\|\chi_{{t_j}^-} -\chi_{t_j}\|\|f(\bar{v}_{t_j})\| +\|\mathscr{V}_{t_{j-1}}\|+(t_{j+1}-t_{j-1})\sup_{t \in (t_{j-1},t_{j+1})}\|
\text{res}_t\|+\int_{t_{j-1}}^{t}C\|\mathscr{V}_s\|ds.
\end{equation}
Proceeding this way over each interval gives
\begin{equation}
\|\mathscr{V}_t\| \leq \, h^2\sum_{j=1}^{J}\|\chi_{{t_j}^-} -\chi_{t_j}\|\|f(\bar{v}_{t_j})\| +\|\mathscr{V}_{0}\|+T\esssup_{t \in (0,T)}\|
\text{res}_t\|+\int_{0}^{t}C\|\mathscr{V}_s\|ds.
\end{equation}
According to \eqref{eqn_jump_amt} and the discussion preceding \eqref{eqn_jump_amt}, there exists a constant $C$ such that for all jump times $t_j \in \{t_1,t_2,\ldots,t_J\}$ we have that $\|\chi_{t_j}-\chi_{{t_j}^{-}}\|\|f(\bar{v}_{t_j})\|\leq Ch^{-1+p}\max_{j \in \{1,\ldots, J\}}\|f(\bar{v}_{t_j})\|$. From \eqref{inq_bound_chi} we have that $|\mathscr{V}_{0}|\leq Ch^{2p}$. Combing these observations with \eqref{inq_nu}, we find that
\begin{align}
    \|\mathscr{V}_t\| \leq C_\omega(1+T) \max_{j \in \{1,\ldots, J\}}\|f(\bar{v}_{t_j})\| h^{p} +T\esssup_{t\in [0,T]}\| \text{res}_t\| + \int_{0}^{t}C\|\mathscr{V}_s\|ds.
\end{align}

By Gronwall's inequality 
\begin{equation}
\begin{aligned}
&\sup_{t \in [0,T]}\max_{k}\left|V^{(k)}_t-\overline{V}^{(k)}_t\right| 
  \\ &\leq  (1+T)\left(C_\omega \max_{j \in \{1,\ldots, J\}}\|f(\bar{v}_{t_j})\|h^p+\esssup_{t \in [0,T]}\|\text{res}_t^{(k)}\|\right)e^{CT}.
\label{eqn_with_j}
\end{aligned}
\end{equation}
Now we use Lemma \ref{lemma_bound_res} and \eqref{inq_bound_chi} to complete the proof. 
\end{proof}

So far, we have that $V^{(k)}$ can be approximated by $\bar{v}(hk)$ from Lemma \ref{lemma_Homogenization} above. Next we argue that the pair $(\bar{v},\bar{z})$ is approximated by $(v,z)$ in $L^\infty([0,T] \times \mathbb{S})$, using another Gronwall style argument.
Lemma \ref{L:2} is the culmination of this comparison. Note that $v$ has equal or better regularity than that of $\bar{v}$. 
\begin{lemma}\label{L:2}
    Suppose  $v_0 \in \mathcal{C}^{2}(\mathbb{S})$ and $z_0\in \mathcal{C}^{1}(\mathbb{S})$. Suppose $f,g \in  \mathcal{C}^{1,1}(\mathbb{R})$ and are globally Lipschitz. Finally suppose $\alpha,\beta \in  \mathcal{C}^1(\mathbb{R})$. Then for each $\rho \in (p,1)$ and $T \in (0,\infty)$, there exists finite random constant $C_{\omega,T,\rho}$ such that for all $p \in[0,1)$ we have almost surely that
    \begin{equation}
        \sup_{t \in [0,T]}\left(\|v -\bar{v}\|+\|z-\bar{z}\|\right) \leq  C_{\omega,T,\rho} (h^p \vee h^{1/2-p/2} \vee h^{\rho-p})\left|\log(h^{-1})\right|
        \label{Inq:Intermediate_2}
    \end{equation}
    for all $n\in\mathbb{N}$, where $h=h_n=1/n$.
    \end{lemma}
\begin{proof}
Throughout the proof, we use $\|\cdot\|$ to denote the $L^\infty(\mathbb{S})$ norm. We reuse $C, C_T, C_\omega$ and $C_{T,\omega}$ as constants which, as usual, never depend on $h, n, t,k,$ or $p$ in any way. The argument is elementary, yet technical. The technical details of the argument lie in comparing $z(x)$ to $\bar{z}(x)$, wherein we use eight successive triangle inequalities. The end result is that for all $x$ and $t \in [0,T]$,
\begin{align}
\begin{aligned}
\|z_t-\bar{z}_t\| \leq C_\omega (h^p\vee h^{1/2-p/2} \vee h^{\rho -p})\log(h^{-1}) \\ +\int_0^tC\left(\|z_s-\bar{z}_s\|+\|v_s-\bar{v}_s\|\right)ds.
\end{aligned}
\label{Pre-Gronwall}
\end{align}
Since we also have the inequality

\begin{align}
\|v_t-\bar{v}_t\| \leq \int_0^tC\left(\|z_s-\bar{z}_s\|+\|v_s-\bar{v}_s\|\right)ds,
\end{align}
which is straight forward to establish, we may then apply Gronwall's inequality to the quantity $\|v_t-\bar{v}_t\| +\|z_t-\bar{z}_t\| $.

Now we carry our the comparison of $\bar{z}$ to $z$. We use that $z$ and $v$ are locally Lipschitz continuous (and thus globally over a compact interval) which holds since we have assumed $z_0 \in \mathcal{C}^1(\mathbb{S})$ and $v_0 \in \mathcal{C}^2(\mathbb{S})$. 

For each $x\in\mathbb{S}$, there exists a $k=k_x \in \{0,...,n-1\}$ such that $|kh-x|\leq Lh.$ Using $z \in \mathcal{C}^1$, there exists a constant $C_T$  such that
\begin{align}
\begin{aligned}
\max_{x \in \mathbb{S}}|z(x)-z(h k_x)| \leq C_Th \quad \text{for all }n\in\mathbb{N}.
\end{aligned}
\end{align} 
Moreover, from the regularity established in \eqref{Inq:Regularit z}, we also have that for pairs $\{k_x,x\}$ that $|\bar{z}(x)-\bar{z}(hk_x)| \leq Ch^{1-p}$. Since these are small in $h$, our main concern then is in estimating $|\bar{z}(hk)-z(hk)|$. Using the definition of $\bar{z}$ we have
\begin{equation}
\bar{z}(hk)-z(hk) =
N_{h,p}^{-1}\sum_{j=0}^{n-1}\Phi_{h,p}(hk-hj)[Z^{(j)}-z(hj)+z(hj)-z(hk)].
\label{E:zbar_minus_z}
\end{equation}
Note that in \eqref{E:zbar_minus_z} we can replace $\Phi_{h,p}$ by the indicator function $1_{h,p}$ without changing the equation, because $\Phi_{h,p}$ is being evaluated at $h$ times an integer in each term in \eqref{E:zbar_minus_z}. Here the indicator function $1_{h,p}$ is
\begin{equation}
    1_{h,p}(x):=
    \begin{cases}
    1 & \quad \text{if } 1-[h^{p-1}/2]h \leq x \leq [h^{p-1}/2]h \\
    0 & \quad  \text{Otherwise}
    \end{cases}.
\end{equation}
That is, we have for all $k \in \{0,1,2,\ldots,n-1\}$, $\Phi_{h,p}(hk)=1_{h,p}(hk)$.
Again using the fact that $z$ is locally Lipschitz along with the definition of $N_{h,p}$ in \eqref{E: Size of N} implies the existence of a constant $C_T$ such that\begin{equation}\left|\sum_{j=0}^{n-1}\dfrac{1_{h,p}(hk-hj)}{N_{h,p}}[z(hj)-z(hk)]\right| \leq C_Th^p.\end{equation} Therefore we concern ourselves with the terms of the sum that look like  $Z^{(j)}-z(hj)$ in \eqref{E:zbar_minus_z}. Here we use the definitions of these terms (We only write down the $\beta$ term as the $\alpha$ term is handled analogously. For the $\alpha$ term we put $\cdots$).  
\begin{equation}
\begin{aligned}
    Z^{(j)}-z(hj)&=Z_0^{(j)}-z_0(hj)+\cdots \\ -&\left(Y^-_j\left(\int_{0}^{t}\beta(V^{(j)}_s)Z^{(j)}_s ds\right)-\int_{0}^{t}\beta(v_s(hj))z_s(hj) ds\right).
    \end{aligned}
    \label{E:with_Pois}
\end{equation}
We have assumed that $Z_0^{(j)}\sim \text{Ber}(z_0(hj))$ according to our Assumption \ref{A:initial condtions_Discrete}. Thus for all $k\in\{0,1,\ldots, n-1\}$, there exists a $C_\omega$ which is finite almost surely such that \begin{equation} \left|\sum_{j=0}^{n-1}\dfrac{1_{h,p}(hk-hj)}{N_{h,p}}[Z_0^{(j)}-z_0(hj)]\right| \leq C_{\omega} h^{1/2-p/2} \log(h^{-1}).
\end{equation} 
To bound the second part of the right hand side of \eqref{E:with_Pois}, we can write 
\begin{equation}
\begin{aligned}
&Y^-_k\left(\int_{0}^{t}\beta(V^{(j)}_s)Z^{(j)}_s ds\right)-\int_{0}^{t}\beta(v_s(hj))z_s(hj) ds\\
 &=   Y^-_j\left(\int_{0}^{t}\beta(V^{(j)}_s)Z^{(j)}_s ds\right)-\int_{0}^{t}\beta(V^{(j)}_s)Z^{(j)}_sds  \\+&\int_{0}^{t}\beta(V^{(j)}_s)Z^{(j)}_sds -\int_{0}^{t}\beta(v_s(hj))z_s(hj) ds.
    \end{aligned}
    \label{E:make_into_comp}
\end{equation}
Lemma \ref{Lemma_spatial_law} 
uniformly bounds the sum of compensated Poisson process defined by the first two terms in the right hand side of  \eqref{E:make_into_comp}. Specifically, we have the existence of a finite $C_{\omega,T}$ such that for all $n\in\mathbb{N}$, $k\in\{0,1,\cdots, n-1\}$ and $p\in(0,1)$,
\begin{equation}
\begin{aligned}
 \sup_{t \in[0,T]}\left|\sum_{j=0}^{n-1}\dfrac{1_{h,p}(hk-hj)}{N_{h,p}}\left[Y^-_j\left(\int_{0}^{t}\beta(V^{(j)}_s)Z^{(j)}_s ds\right)-\int_{0}^{t}\beta(V^{(j)}_s)Z^{(j)}_sds \right]\right|
\\ 
\leq C_{\omega,T} h^{1/2-p/2}\log(h^{-1}).
 \end{aligned}
 \label{E:Uses_Lemma1_2}
\end{equation}
To bound the second to two terms on the right hand side of  \eqref{E:make_into_comp}, we write 
\begin{equation}
\begin{aligned}
&\int_{0}^{t}\beta(V^{(j)}_s)Z^{(j)}_sds -\int_{0}^{t}\beta(v_s(hj))z_s(hj) ds\\ &=\int_{0}^{t}\beta(V^{(j)}_s)Z^{(j)}_sds-\int_{0}^{t}\beta(\bar{v}_s(hj))Z^{(j)}_sds\\+&\int_{0}^{t}\beta(\bar{v}_s(hj))Z^{(j)}_sds-\int_{0}^{t}\beta(v_s(hj))z_s(hj)ds.
\end{aligned}
\label{E:use_intermediate}
\end{equation}
For the first two terms on the right hand side of \eqref{E:use_intermediate}, our intermediate result in Lemma \ref{lemma_Homogenization}, as well as our assumption that $\alpha,\beta \in  \mathcal{C}^{1}(\mathbb{R})$ give us that for all $p\in [0,1)$ and $\rho \in (p,1)$, there exists $C_{\omega,T,\rho}$ such that
\begin{equation}
\begin{aligned}
   \sup_{t \in [0,T]}\left|\sum_{j=0}^{n-1}\dfrac{1_{h,p}(hk-hj)}{N_{h,p}} \left[\int_{0}^{t}\beta(V^{(j)}_s)Z^{(j)}_sds-\int_{0}^{t}\beta(\bar{v}_s(hj))Z^{(j)}_sds\right]\right| \\ \leq  C_{\omega,T,\rho} h^{p\wedge (\rho-p)}\left|\log(h^{-1})\right|
   \end{aligned}
\end{equation}
for all  $n\in\mathbb{N}$ and $k\in\{0,1,\cdots, n-1\}$.

Now to bound the second two terms on the right hand side of \eqref{E:use_intermediate}, we write
\begin{equation}
\begin{aligned}
&\int_{0}^{t}\beta(\bar{v}_s(hj))Z^{(j)}_sds-\int_{0}^{t}\beta(v_s(hj))z_s(hj)ds \\
&=\int_{0}^{t}\beta(\bar{v}_s(hj))Z^{(j)}_sds-\int_{0}^{t}\beta(v(hj))Z^{(j)}_sds \\+&\int_{0}^{t}\beta(v(hj))Z^{(j)}_sds-\int_{0}^{t}\beta(v_s(hj))z_s(hj)ds
    \end{aligned}
    \label{E:break_up_product}
\end{equation}
We bound the first two terms in the right hand side of \eqref{E:break_up_product} using the Lipschitz constant for $\beta$ like so,
\begin{equation}
\begin{aligned}
\left|\sum_{j=0}^{n-1}\dfrac{1_{h,p}(hk-hj)}{N_{h,p}} \left[\int_{0}^{t}\beta(\bar{v}_s(hj))Z^{(j)}_sds-\int_{0}^{t}\beta(v_s(hj)Z^{(j)}_sds \right]\right| 
\\
\leq C\int_0^t\|\bar{v}_s-v_s\|ds.
\end{aligned}
\end{equation}
The term  now matches the form we are trying to obtain in \eqref{Pre-Gronwall}. Now to bound the second two terms in \eqref{E:break_up_product},  we write
\begin{equation}
\begin{aligned}
    &\int_{0}^{t}\beta(v(hj))Z^{(j)}_sds-\int_{0}^{t}\beta(v_s(hj))z_s(hj)ds\\ &=\int_{0}^{t}\beta(v(hj))Z^{(j)}_sds-\int_{0}^{t}\beta(v(hk))Z^{(j)}_sds\\
    &+\int_{0}^{t}\beta(v(hk))Z^{(j)}_sds-\int_{0}^{t}\beta(v_s(hj))z_s(hj)ds.
    \end{aligned}
    \label{E:Remove_averaging}
\end{equation}
The first two terms  in \eqref{E:Remove_averaging} are bounded by the assumption that $\alpha, \beta \in \mathcal{C}^1(\mathbb{R})$ and that $v \in \mathcal{C}^{1}(\mathbb{S})$:
\begin{equation}
    \sup_{t \in[0,T]}\left| \sum_{j=0}^{n-1}\dfrac{1_{h,p}(hk-hj)}{N_{h,p}}\left[\int_{0}^{t}\beta(v(hj))Z^{(j)}_sds-\int_{0}^{t}\beta(v_s(hk))Z^{(j)}_sds\right]\right| \leq C_Th^p
\end{equation}
Now to bound the second two terms in \eqref{E:Remove_averaging} we write
\begin{equation}
\begin{aligned}
&    \int_{0}^{t}\beta(v(hk))Z^{(j)}_sds-\int_{0}^{t}\beta(v_s(hj))z_s(hj))ds 
\\&=    \int_{0}^{t}\beta(v(hk))Z^{(j)}_sds-\int_{0}^{t}\beta(v_s(hk))z_s(hk)ds
\\
+&\int_{0}^{t}\beta(v_s(hk))z_s(hk)ds-\int_{0}^{t}\beta(v_s(hj))z_s(hj)ds
    \end{aligned}
\end{equation}
Using the definition of local averages in \eqref{Def:localspatial_Z}, the first two terms become 
\begin{equation}
\begin{aligned}
\left|\sum_{j=0}^{n-1}\dfrac{1_{h,p}(hk-hj)}{N_{h,p}} \left[\int_{0}^{t}\beta(v(hk))Z^{(j)}_sds-\int_{0}^{t}\beta(v_s(hk))z_s(hk)ds\right] \right| \\ \leq C_T\int_0^{t}\|\bar{z}_s-z_s\|ds, 
\end{aligned}
\end{equation}
The last two terms are bounded using the fact that $z,v\in \mathcal{C}^1(\mathbb{S})$ 
\begin{equation}
\begin{aligned}
\left|\sum_{j=0}^{n-1}\dfrac{1_{h,p}(hk-hj)}{N_{h,p}} \left[\int_{0}^{t}\beta(v(hk))z_s(hk)ds-\int_{0}^{t}\beta(v_s(hj))z_s(hj)ds\right] \right| \\ \leq C_T h^{p}, 
\end{aligned}
\end{equation}
establishing \eqref{Pre-Gronwall} and completing the proof.
\end{proof}

\begin{proof}[Proof of Theorem \ref{T:Main}] The proof for the toy model follows directly from Lemmas \ref{lemma_Homogenization} and \ref{L:2}. The proof for the general model  \eqref{eq:general_model V}-\eqref{eq:general_model Z} follows from the same argument without much change.

We use the general definition \eqref{Def:new localspatial_Z} for the local averages. The local averages would be used to define an intermediate PDE for $\bar{v}$. 
Using the linearity of the Laplacian, the ansatz would then be
\begin{equation}
\label{Def:new localspatial_Z}
    \overline{V}^{(k)}:=\bar{v}(hk)+h^2\sum_{i=1}^I\sum_{j=1}^J\chi_{i,j}^{(k)}g_{i,j}(\bar{v}(hk))
\end{equation} 
where 
\begin{equation}
D(\chi_{i,j}^{(k+1)}-2\chi_{i,j}^{(k)}+\chi_{i,j}^{(k-1)})=\overline{Z}_{i,j}^{(k)}-Z_{i,j}^{(k)}.
\end{equation}
Lemma \ref{L:consistency} holds for the general model with this ansatz.
We then compare $\bar{v}$ with $V$ as in Lemma \ref{lemma_Homogenization} and  compare $(\bar{v},\bar{z}_{i,j})$ with $(v,z_{i,j})$ as in Lemma \ref{L:2}. The constants would then depend on $I,J$ which we assumed are fixed.
\end{proof}

\section{Numerical Simulations}\label{S:simulation}
Next we perform numerical simulations for the stochastic system. The specific system we simulate is given by 
\begin{equation}
\begin{aligned}
\dot{V}&=h^{-2}(V^{(k+1)}-2V^{(k)}+V^{(k-1)})+Z^{(k)}(1-V^{(k)})-\frac{1}{10}V^{(k)}
\\
Z^{(k)}&=Z^{(k)}_0+Y^{(k)}_{+}\left(\int_0^{t}\alpha(V^{(k)})(1-Z^{(k)})ds\right) -Y^{(k)}_{-}\left(\int_0^{t}\beta(V^{(k)})Z^{(k)}ds\right),
\end{aligned}
\label{E:Numerical Object}
\end{equation}
where $\alpha(v)=e^{10(v-0.5)}$ and $\beta(v)=e^{-10(v-0.5)}$. This specific system is a special case of the general model in \eqref{eq:general_model V} and \eqref{eq:general_model Z} by taking the following. With $J=I=2$, set $g_{1,1}(v)=(1-v)$ and $g_{2,1}=-av$ where $a=1/10$. Further set $g_{1,2}\equiv g_{2,2}\equiv 0$. The rate functions are given by $A_{1,(1,2)}=\alpha$ and $A_{1,(2,1)}=\beta$. For all $(a,j) \in \{1,2\}\times\{1,2\}$ we have that $A_{2,(a,j)}=0$. Because there is no stochastic term associated with $g_{2,1}$, we would need to initialize $Z_{2,1}=1$. Since $A_{2,(a,j)}=0$, we have $Z_{2,1}=1$ for all time. In any case, it is notionally simpler to work with \eqref{E:Numerical Object}. 

We set $L=16$ and
the initial conditions to be the following: 
\begin{equation}
\begin{aligned}
    V^{(k)}_0=&\,e^{-((k-(Ln-1)/2)/n)^2}\\
    Z^{(k)}_0 \sim& \,\text{Bernouli}\left(\dfrac{\alpha(V_0^{(k)})}{\alpha(V^{(k)}_0)+\beta(V^{(k)}_0)}\right).
\end{aligned}
\end{equation}
Note that there are $Ln=16n$ lattice sites in total, with periodic boundary conditions.

We use a modified version of the SSA with Poisson thinning to simulate \eqref{E:Numerical Object}. The use of thinning to simulate PDMPs is a natural idea and has been studied in \citep{Thomas_thinning}. Poisson thinning to simulate non-homogeneous Poisson processes originates in \citep{Poisson_Thinning}. We find it useful to incorporate in our simulation method. 

Our algorithm is equivalent to algorithms that require solving a hitting time problem. See algorithms described in \citep{reidler2013numerics} and \citep{Ding} for examples. In \citep{reidler2013numerics}, the author proves strong convergence of the numerical solution as the time step goes to $0$ for a PDMP with a single rate. In \citep{Ding}, the authors show how a PDMP with multiple rates may be simulated by solving the hitting time problem for one global rate. Our algorithm is similar in that we have a global rate; however, we avoid needing to solve a hitting time problem by using Poisson thinning. We adopt the terminology of \citep{reidler2013numerics} in calling our algorithm pseudo-exact, but since we also use thinning, we refer to our first algorithm as the \textit{Psuedo-Exact Thinning (PET) Algorithm}. We then introduce a second algorithm, because of our PET Algorithm turns out to be slow as would any exact algorithm be for PDMP with many reaction rates. 

The second algorithm is not exact, and has a step less justified than but analogous to $\tau$ leaping. We thus refer to it as the \textit{Inexact Leaping (IL) Algorithm}. We note that $\tau$ leaping has been justified in the case of chemical reaction networks \citep{Gillespie-2005} so we suspect that IL can be justified as well. Having a fast algorithm would make it possible to use Monte-Carlo methods to explore better the behavior of stochastic ion channel models. We note that there appears to be no work studying the quality of $\tau$ leaping methods for PDMPs in the literature.

\medskip

{\bf Description of PET and IL Algorithms. } 
We first give some explanation of our PET algorithm, whose steps are explicitly written in in Algorithm I below. Since solutions $V$ of \eqref{E:Numerical Object}  stay bounded below $1$ and above $0$, we can define a local rate $\lambda :=\alpha(1)+\beta(0)$. The rate function gives an upper bound for the maximum rate of any of the Poisson processes in \eqref{E:Numerical Object}. We also set a global rate $\Lambda:=2nL\lambda$, which is an upper bound on the number of Poisson processes among all $Ln$ lattice sites. (The local and global rates may be adjusted at each time-step and one does not actually need totally boundedness of $V$, so long as one may estimate $V$ one time-step into the future.) Now suppose at some time $t_0 \geq 0$, we have an approximate solution of $(V_{t_0},Z_{t_0})$ which we call $(\hat{V}_{t_0},\hat{Z}_{t_0})$ . We generate an exponential random variables $\tau$ with rate $\Lambda$, which gives a lower bound on how long we must wait for any one of the $Z^{(k)}$ to switch (either from $0$ to $1$ or $1$ to $0$). Thus for $t \in [t_0,t_0+\tau)$, $Z^{(k)}$ remains constant for all $k.$ We use a numerically stable second order numerical ODE method such as the trapezoidal method to find $\hat{V}^{(k)}_{t_0+\tau}$ which we can do since $V^{(k)}_{t_0+\tau}=\lim_{t\to \tau}V^{(k)}_{t_0+t}$. We then pick an integer $k_0$ uniformly from $\{0,1,2,\ldots,n-1,\ldots,Ln-1\}$. If $\hat{Z}_{t_0}^{(k_0)}=0$ we generate a Bernoulli random variable with success probability $\frac{\alpha(\hat{V}^{(k_0)}_{t+\tau})}{2\lambda}$. If $Z_{t_0}^{(k_0)}=1$, we use $\frac{\beta(\hat{V}^{(k_0)}_{t+\tau})}{2\lambda}$ as the success probability. In either case, if successful, we set $\hat{Z}_{t_0+\tau}^{(k_0)}=1-\hat{Z}_{t_0}^{(k_0)}$, otherwise we change nothing. We repeat the procedure now starting from $t_0+\tau$ and iterate until we have simulated to a desired time. 

\medskip

{\bf Algorithm I: Psuedo-Exact Thinning (PET). } 
Fix a satisfactorily small time step $\Delta t$. Given we have the approximate solution at $(\hat{V}_{t_0}, \hat{Z}_{t_0})$ for some $t_0$, we generate the approximate solution $(\hat{V}_{t_0+\tau}, \hat{Z}_{t_0+\tau})$ at the potential next jump time:
\begin{itemize}
    \item {\bf Step 1}: Choose $\lambda$ such that $\lambda \geq \max_{k}\{\sup_{t \in [t_0,t_0+\Delta t]}\alpha(V^{(k)}_{t})(1-Z^{(k)}_{t_0}),\sup_{t \in [t_0,t_0+\Delta t]}\beta(V^{(k)}_{t})Z^{(k)}_{t_0}\}$, and set $\Lambda=2nL\lambda$.
    \item {\bf Step 2}: Generate random wait time $\tau \sim \text{Exp}\left(\Lambda\right)$ (exponential random variable with mean $1/\Lambda$). 
    \item {\bf Step 3}: Compute the numerical solution of $\hat{V}_{t_0+\tau}$ to second order using $\hat{Z}_{t_0}$ using increments equal or smaller than $\min\{\Delta t,\tau\}$.
    \item{\bf Step 4:} Choose integer $k$ uniformly from $\{0,1,2,\ldots, nL-1\}$ at random. Set $\alpha^+_k:=\alpha(\hat{V}_{t_0+\tau}^{(k)})(1-\hat{Z}_{t_0}^{(k)})$ and $\alpha^-_k:=\beta(\hat{V}_{t_0+\tau}^{(k)})\hat{Z}^{(k)}_{t_0}$,  and then choose  $y \sim \text{Bernouli}\left(\frac{\alpha^++\alpha^-}{2\lambda}\right)$.
    \item{\bf Step 5}: If $y=1$, set $\hat{Z}_{t_0+\tau}^{(k)}=1-\hat{Z}_{t_0}^{(k)}.$ If $y=0$, set $\hat{Z}_{t_0+\tau}\equiv\hat{Z}_{t_0+\tau}$. 
    \label{A:I}
\end{itemize}

\begin{remark}
    There may be a number of tweaks to the above algorithm one may use to make it more efficient practically speaking. For example, rather than picking $y \sim \text{Bernouli}\left(\frac{\alpha^++\alpha^-}{2\lambda}\right)$ in step 4, it may be advantageous to compute a uniform random variable in $[0,1]$, call it $y$, and compare $y$ to $\frac{\alpha^++\alpha^-}{2\lambda}$. This $y$ might then be pre-computed (before solving the PDE) and one might use knowledge of $V_{t_0}$, $Z_{t_0}$, and $\tau$ to estimate $V_{t_0+\tau}$. In certain cases $y$ may be obviously large enough or small enough that the outcome of $\hat{Z}^{(k)}_{t_0+\tau}$ can be decided without needing to perform Step 3. 
\end{remark}
There are obviously many tweaks of the algorithm above that are statistically equivalent, and one can choose the one most suitable for their numerical scheme. For example it is possible to pre-compute all random variables before iterating.
We have plans to explore improvements to and error analysis of the above algorithm in a future work. For now, we only note two things at a heuristic level. 

First we consider the computational complexity. Usually, when solving systems of ODEs like the one in \eqref{E:Numerical_Determinstic}, the computational complexity of the solver depends both on $\Delta t$, the time increment, and the dimension of the system. For solving a discretized heat equation (not as an approximation to the heat equation), for example, we expect the number of floating point operations to be roughly $O(n/\Delta t)$, and we may increase the $n$ without decreasing $\Delta t$ as long as we are using a stable numerical method. However, in simulating \eqref{E:Numerical Object} the global rate $\Lambda$ scales with $n$, $\tau\sim 1/n$. Therefore, the computational complexity will be $O(n^2)$. Thus when $n$ is increased, the time it takes to solve up to a fixed time grows quadratically in the number of ion channels, which can make it challenging to perform numerical experiments especially those involving a realistic number of ion channels. 

Second we consider the error. If a second order ODE solver with time step $\Delta t$ is being used, then the error generated in one time-step of the success probability of the Bernoulli trial is $O(\Delta t^2)$ i.e. $\alpha(\hat{V}_{\tau})\approx \alpha(V_\tau)+O(\Delta t^2).$  We thus expect a Bernouli trial to be "incorrect" about once every $O(1/\Delta t^2)$ steps. Fortunately, we only need to take $O(1/\Delta t)$ steps, so we expect the algorithm to simulate the statistical behavior of $Z$ exactly. This argument is far from a proof; however, we note that strong convergence as the time step goes to $0$ in a similar numerical method for PDMP with a single rate was proven in \citep{reidler2013numerics}.

 Since our PET Algorithm is slow relatively to a standard ODE solver where a fixed time-step may be used even for a larger and larger system size, we introduce a naive variation of PET in order to speed it up. Unfortunately this comes at a cost to its accuracy but we conjecture that is may be useful. Algorithm II, or the IL Algorithm fixes the time-step using an idea akin to that of tau leaping, but which is less justified in our setting. We note that their is essentially no formulation of $\tau$-leaping for piecewise deterministic Markov processes in the literature, and so we make an attempt at such a formulation for the express purpose of simulating our model for larger $n$. 
 
 Suppose we have a numerical solution at $t_0$ given by $(\hat{V}_{t_0}, \hat{Z}_{t_0})$. If we wish to fix a time step $\tau$, and we desire to compute $(\hat{V}_{t_0+\tau},\hat{Z}_{t_0+\tau})$, we run into the problem that multiple components of the process $\hat{Z}$ may have jumped over the interval of time of length $\tau.$ To compensate for this, we generate a Poisson distributed random variable with mean $\tau\Lambda$, which gives an upper bound on the number of firings that may have occurred over the time interval $\tau.$ After simulating to $\hat{V}_{t_0+\tau}$, we then perform thinning to update $\hat{Z}_{t_0+\tau}$. 

 \medskip

 {\bf Algorithm II: Inexact Leaping  (IL). } 
Given we have the approximate solution at $(\hat{V}_{t_0}, \hat{Z}_{t_0})$ for some $t_0$, we generate the approximate solution  $(\hat{V}_{t_0+\tau}, \hat{Z}_{t_0+\tau})$ for some fixed interval $\tau$:
\begin{itemize}
    \item {\bf Step 1}: Choose $\lambda$ such that $\lambda \geq \max_{k}\{\sup_{t \in [t_0,t+\tau]}\alpha(V^{(k)}_{t}),\sup_{t \in [t_0,t+\tau]}\beta(V^{(k)}_{t})\}$, and set $\Lambda=2nL\lambda$.
    \item {\bf Step 2}: Generate max number of possible jumps $m \sim \text{Poisson}\left(\tau{\Lambda}\right)$ (a Poisson random variable with mean $\tau\Lambda$.)
    \item {\bf Step 3}: Compute the numerical solution of $\hat{V}_{t_0+\tau}$ to second order using $\hat{Z}_{t_0}$.
    \item{\bf Step 4:} Choose i.i.d. integers $\{k_i\}_{i=1}^{m}$ uniformly from $\{0,1,2,\ldots, nL-1\}$. Set $\alpha^+_{i}:=\alpha(\hat{V}_{t_0+\tau}^{(k_i)})$ and $\alpha^-_i:=\beta(\hat{V}_{t_0+\tau}^{(k_i)})$ and $\tilde{Z}\equiv\hat{Z}_{t_0}$.
    
    \item{\bf Step 5}: Looping through $i \in \{1,\ldots, m\}$,  choose independently random variables $\{y_i\}_{i=1}^{m}$ such that $y_i \sim \text{Bernoulli}\left(\frac{\alpha_{i}^+(1-\tilde{Z}^{(k_i)})+\alpha_{i}^-\tilde{Z}^{(k_i)}}{2\lambda}\right)$. Set $\tilde{Z}^{(k_i)}=1-\tilde{Z}^{(k_i)}$ whenever $y_i=1$. Otherwise, do nothing. Then set $\hat{Z}_{t_0+\tau}\equiv\tilde{Z}.$
    \label{A:II}
\end{itemize}
We do not have a proof of convergence of the above algorithm, but we conjecture it is weak order one, since this is the case for $\tau$ leaping for chemical reaction networks \citep{Gillespie-2005}. By weak order one, we mean that, after taking the linear interpolant of $\hat{V}^{(k)}$ with respect to time, we have for each fixed $T$, $\sup_{t \in[0,T]}\left|\E[f(\hat{V}^{(k)})]-\E[f(V^{(k)})]\right| \sim O(\tau)$ as $\tau \to 0$ for $f$ in a suitable class of functions. Further, we conjecture the error converges uniformly over $h\in (0,h_0)$  for some positive constant $h_0<1$. In such a case, for sufficiently large $n$, it is better to use IL. The computational complexity of IL is much lower than that of PET for large systems, because we may fix the step size $\tau$, so that the number of flops is $O(n/\tau)$ as compared to $O(n^2)$. We save further exploration of the algorithm for a later work, as it is outside the scope of our primary objective of this paper. For now, we only show some qualitative comparisons between the algorithms below and use the algorithm to IL simulate our model for larger $n$ than we can with PET.

 \medskip

\noindent{\bf Numerical Results I.} First we display some numerical results simply using our PET Algorithm, Algorithm I. We compare simulations of \eqref{E:Numerical Object} with the following system of ODEs
\begin{equation}
\begin{aligned}
    \dot{U}&=h^{-2}(U^{(k+1)}-2U^{(k)}+U^{(k-1)})+S^{(k)}(1-U^{(k)})-\frac{1}{10}U^{(k)}
\\
S^{(k)}&=S^{(k)}_0+\int_0^{t}\alpha(U^{(k)})(1-S^{(k)})ds -\int_0^{t}\beta(U^{(k)})S^{(k)}ds.
\end{aligned}
\label{E:Numerical_Determinstic}
\end{equation}
Theorem \ref{T:Main} says that \eqref{E:Numerical Object}  should approach solutions of the PDE given by
\begin{equation}
\begin{aligned}
    \partial_t v&=\Delta v+z(1-v)-\frac{1}{10}v
\\
z&=z_0+\int_0^{t}\alpha(z)(1-z)ds -\int_0^{t}\beta(v)zds
\end{aligned}
\end{equation}
as $n\to \infty$ at a rate no slower than $n^{-1/3}|\log(n)|$. It is also clear that solutions of \eqref{E:Numerical_Determinstic} approach solutions of the PDE at least as fast. Thus we simply compare numerical solutions of \eqref{E:Numerical Object} with numerical solutions to \eqref{E:Numerical_Determinstic}.

\begin{figure}
    \captionsetup{width=.9\textwidth}
    \caption*{
        Quantifying the error using PET}
    \ 
    \includegraphics[ width=0.5\linewidth]{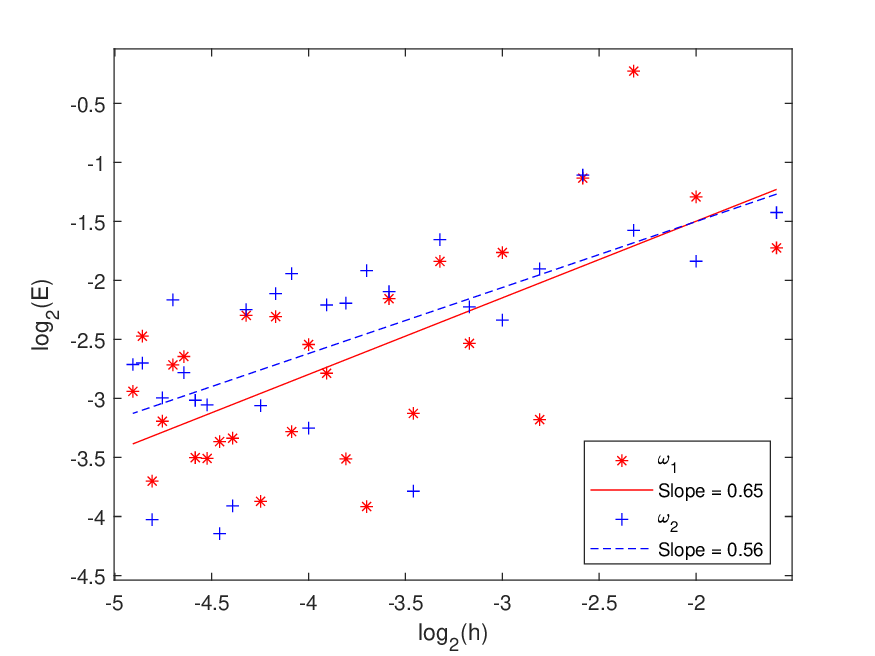}
    \includegraphics[, width=0.5\linewidth]{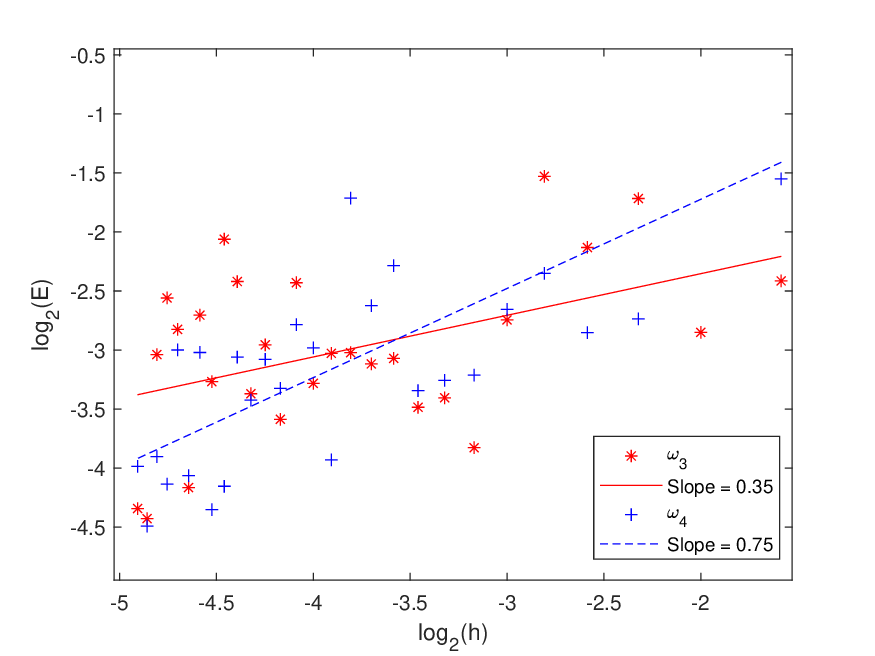}
    
    \includegraphics[ width=0.5\linewidth]{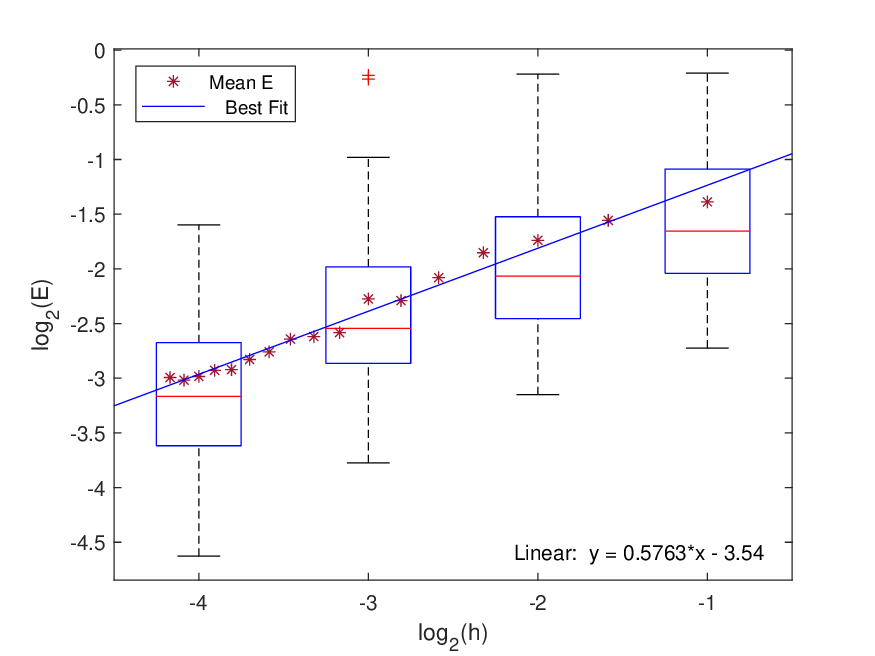}
    \includegraphics[height=0.25\textheight, width=0.5\linewidth]{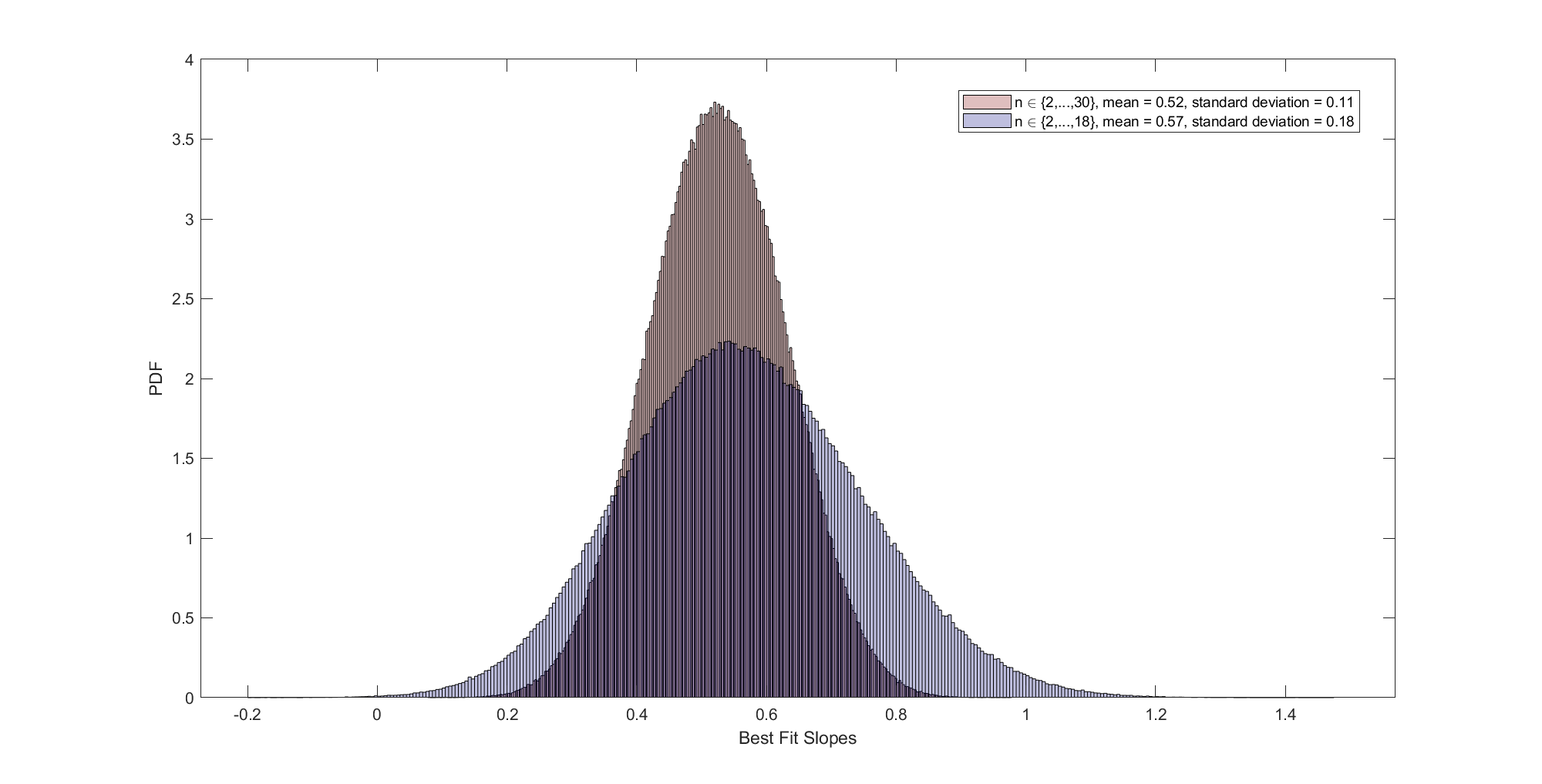}

    \caption{The slope of lines of best fit estimates the exponents of the rate of convergence of the error to $0$ for different experiments. In our experiments, that rate is consistent with the upper bound in Theorem \ref{T:Main}, i.e., almost sure convergence is observed for any given $\omega$. If we take a sample mean of $E(h)$ from $\Omega$, we see that a rate of $1/2$ captures the convergence of the sample mean to $0$. The two overlaid histograms show that the distribution of slopes depends upon the range over which $h$ is taken. For a range that includes more and smaller values of $h$, the distribution is tighter with a mean around $1/2$}
    \label{fig:experiments}
\end{figure}
            
\begin{figure}
    \caption*{Comparison of stochastic and deterministic for $h=1/16$} 
    \centering
    
    \includegraphics[width=\linewidth]{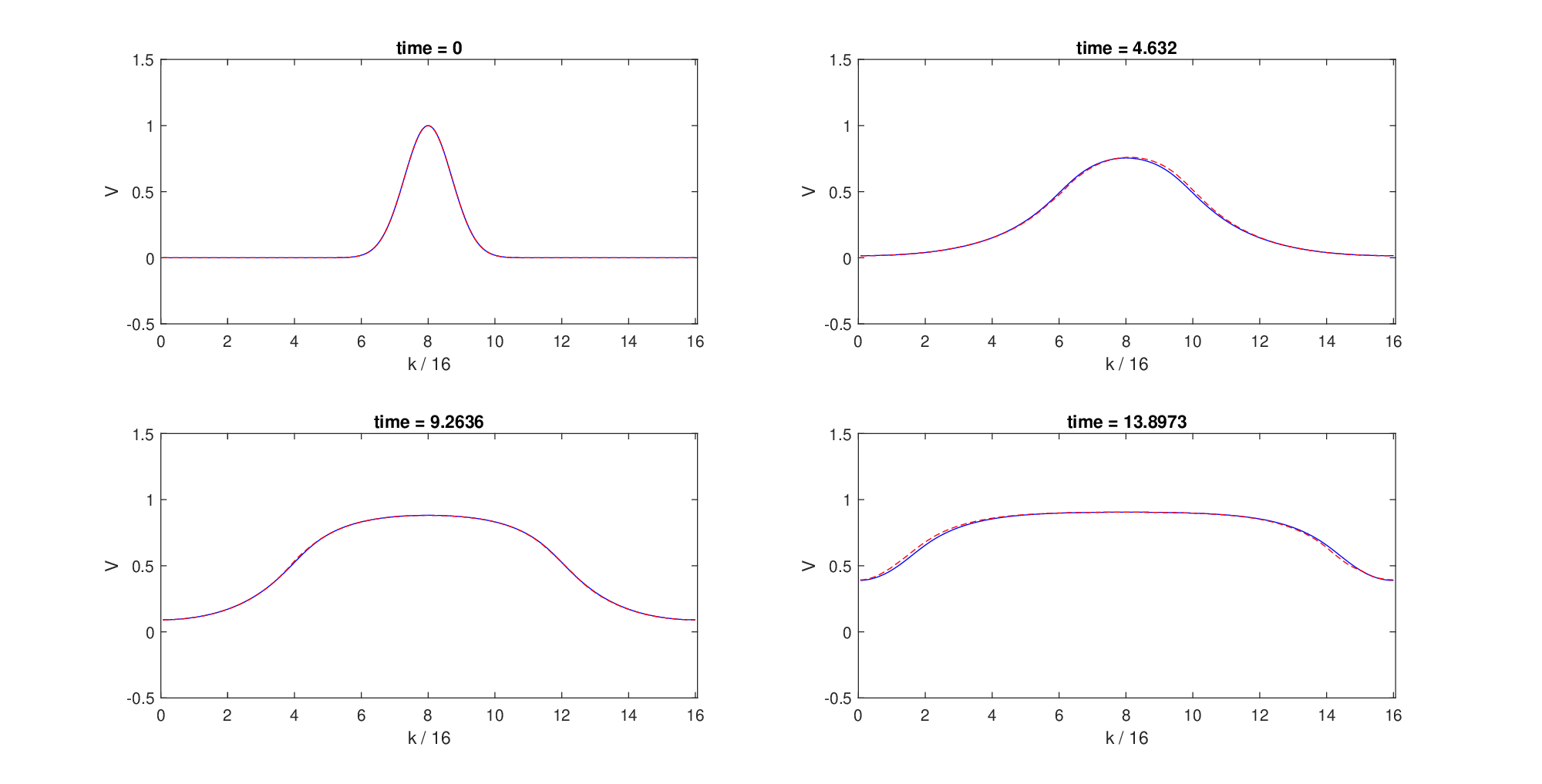}
    \caption{We compare the deterministic solution in solid blue with the stochastic solution in dashed red when $h=1/16$ at four time points. The solutions look virtually identical}
    
    \label{fig:Comparison}
\end{figure}
\begin{figure}
    \caption*{Comparison of stochastic and deterministic for $h=1/4$} 
    \centering
    
    \includegraphics[width=\linewidth]{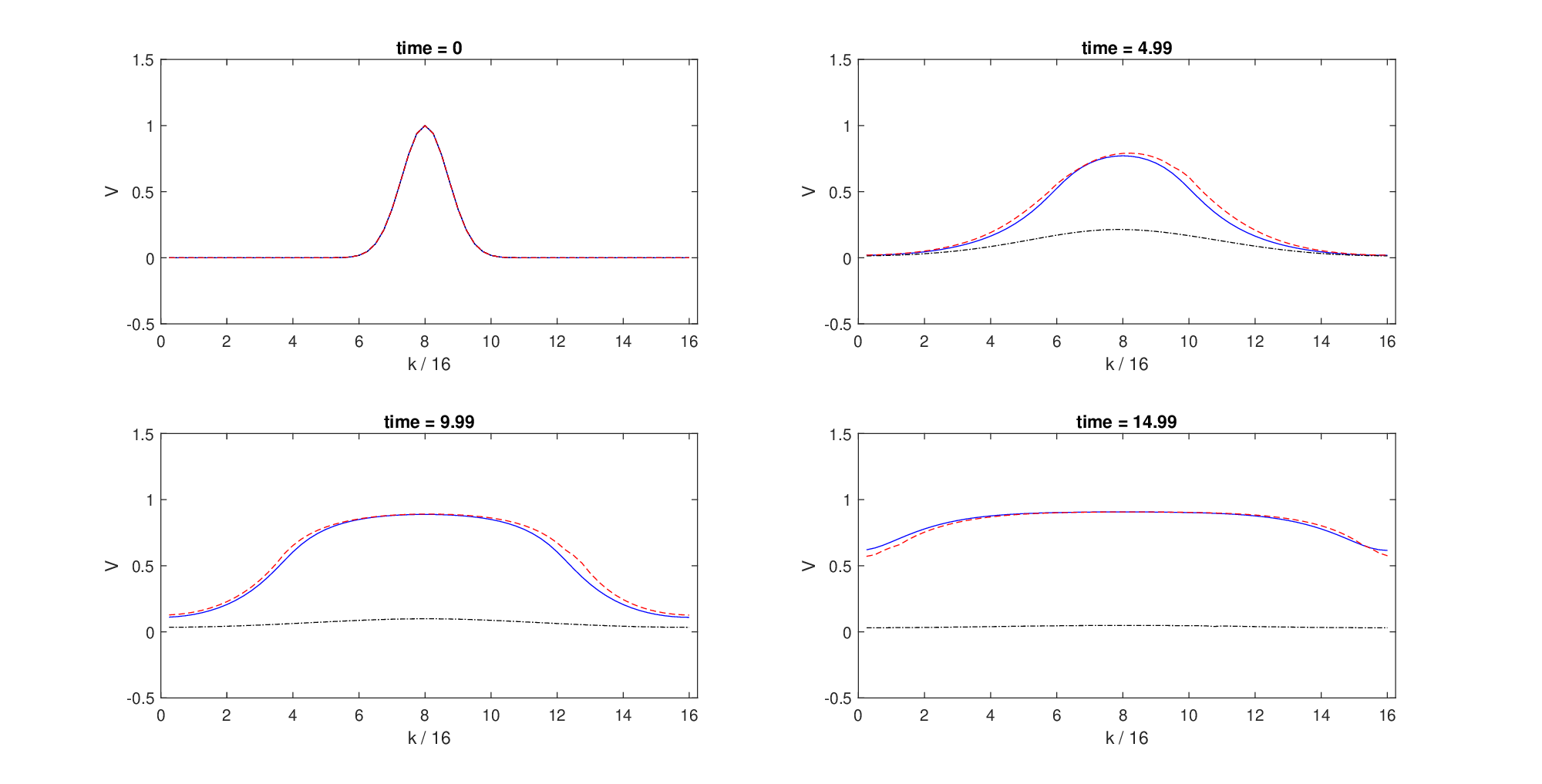}
    \caption{We compare the deterministic solution in solid blue with two stochastic realizations in dashed red and dashed black for $h=1/4$. One of the stochastic solutions matches closely the the deterministic solution. The other does not. We conclude that the stochastic solution has some probability of decaying away}
    \label{fig:Comparison_for_large_h}
\end{figure}

We define the \textit{error} as $E_\omega(h):=\sup_{t \in[0,15]}\max_{k}\left|\hat{V}^{(k)}-\hat{U}^{(k)}\right| $. The error $E_\omega(h)$ is thus a random function of $h$. In top two panels of Figure \ref{fig:experiments}, we perform four experiments $\omega_1,\omega_2,\omega_3,$ and $\omega_4$ (samples from space $\Omega$ in which the outcomes of the infinite array of Poison processes in Lemma \ref{Lemma_spatial_law} live ) measuring $E_{\omega_i}(h)$ for each $h$ ranging over $\{1/n\}_{n=2}^{30}$. We observe that $E$ does in fact generally decrease with $h$ in each experiment as demonstrated by the trend lines. This is evidence of almost sure convergence; however, the rate of decrease varies dramatically across experiments, and thus it is very challenging to draw conclusions regarding the sharpness of our result in Theorem \ref{T:Main}. This should be contrasted with the rate of weak convergence. For instance details in \citep{austin2008emergence} such as Lemma 4  suggest $h^{1/2}$ as the correct rate of weak convergence. In the bottom left panel of Figure \ref{fig:experiments}, we perform $100$ such experiments for each $h\in\{1/n\}_{n=2}^{18}$ and plot the sample mean of $E(h)$ vs. $h$ on a log-log plot. The slope of the trend-line gives strong evidence $h^{1/2}$ is correct for the weak convergence rate that is the rate at which the expectation of $E(h) \to 0$. Finally note that for fixed $h$ we may swap $E_{\omega_j}(h)$ with $E_{\omega_i}(h)$ for any two experiments to obtain another experiment. This gives us another albeit biased way to sample from the space $\Omega$ more efficiently by performing such swapping randomly. We generate a million experiments and plot the sample PDF for the slopes of best fit for $h$ over two ranges, $\{1/n\}_{n=2}^{18}$ and $\{1/n\}_{n=2}^{30}$ in the bottom right panel of Figure \ref{fig:experiments}.

In Figures \ref{fig:Comparison} and \ref{fig:Comparison_for_large_h} we compare the stochastic realizations of the solution profile with those of the deterministic system (which are close to the PDE). When $h=1/16,$ very little discrepancy is noticeable. When $h=1/4,$ we show two realization of the stochastic system, one in which the stochastic solution matches fairly closely to the deterministic solution and one in which it decays away. This may be explained by the fact that with so few ion channels, there is a moderate chance not enough ion channels  open when they should to sustain a large $V$.

\medskip
\noindent{\bf Numerical Results II.}
Next we focus on using the IL Algorithm, Algorithm II. We briefly compare the IL Algorithm to the PET Algorithm I and investigate the convergence of IL to the PDE for larger values of $n$. The set up in this section is identical to the previous. In Figure \ref{fig:Weak_error}, we compare the 
 numerical approximations $V^{(k)}$ for the index $k$ closest the $L/2$ over $15$ units of time using Algorithm I and II. We compute an \emph{algorithmic error} with the following procedure. We compute the sample mean from 100 samples of the numerical approximation for $V^{(k)}$ using Algorithm I and II. Then we take the maximum overtime. That is, if we define $V_{I}$ and $V_{II}$ to be the numerical approximation using Algorithm I and II respectively, we use 100 samples to estimate $\sup_{t \in[0,15]}\left|\mathbb{E}\left[V_{I}^{(k)}\right]-\mathbb{E}\left[V_{II}^{(k)}\right]\right|$. We do this for the four combinations of  different time steps $\tau=dt \in \{1/4,1/8\}$ and values of $h \in \{1/4,1/8\}$. The error as quantified in the figure appears to shrink when smaller $dt$ is taken for both values of $h$. However, we caution against drawing conclusions regarding convergence from the error quantity we calculate. To fully verify weak convergence as $\tau \to 0$ numerically, would require taking enough samples to estimate and resolve the error between the distributions of the numerical solutions of the two algorithm. However, such a task is large enough that we do not pursue it here. Instead we note that in Figure \ref{fig:Weak_error}, we see similar qualitative behavior. 
 
Finally we use IL with $\tau=1/8$ to study convergence of the stochastic system to the PDE in Figure \ref{fig:alg_II_strong} as we did with PET in Figure \ref{fig:experiments}. The only difference is that in the top panels, $h$ ranges over $ \{1/n\}_{n=2}^{50}$. In the bottom left panel, we take the mean of the error from 100 samples again but for  $h \in \{1/n\}_{n=2}^{30}$. The results we see are comparable to those we obtained with PET. Finally, in the bottom right panel we create another histogram of the slopes of a million experiments for $h \in \{1/n\}_{n=2}^{30}$ via swapping as we did before, but now we compare this to equivalent histogram in Figure \ref{fig:experiments}.

\begin{figure}

    \caption*{Comparing PET and IL}
    \includegraphics[height=0.25\textheight, width=0.5\linewidth]{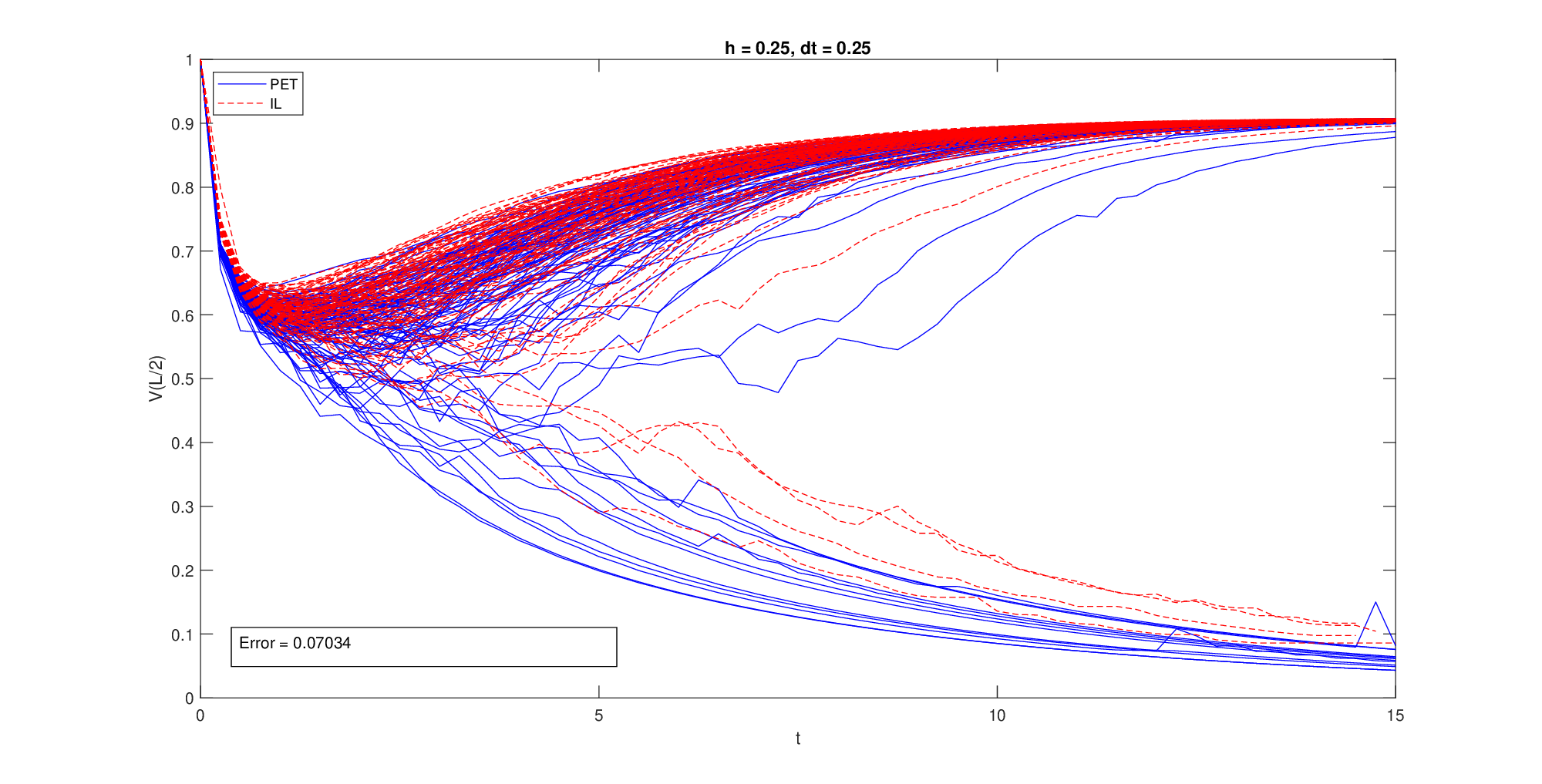}
    \includegraphics[height=0.25\textheight, width=0.5\linewidth]{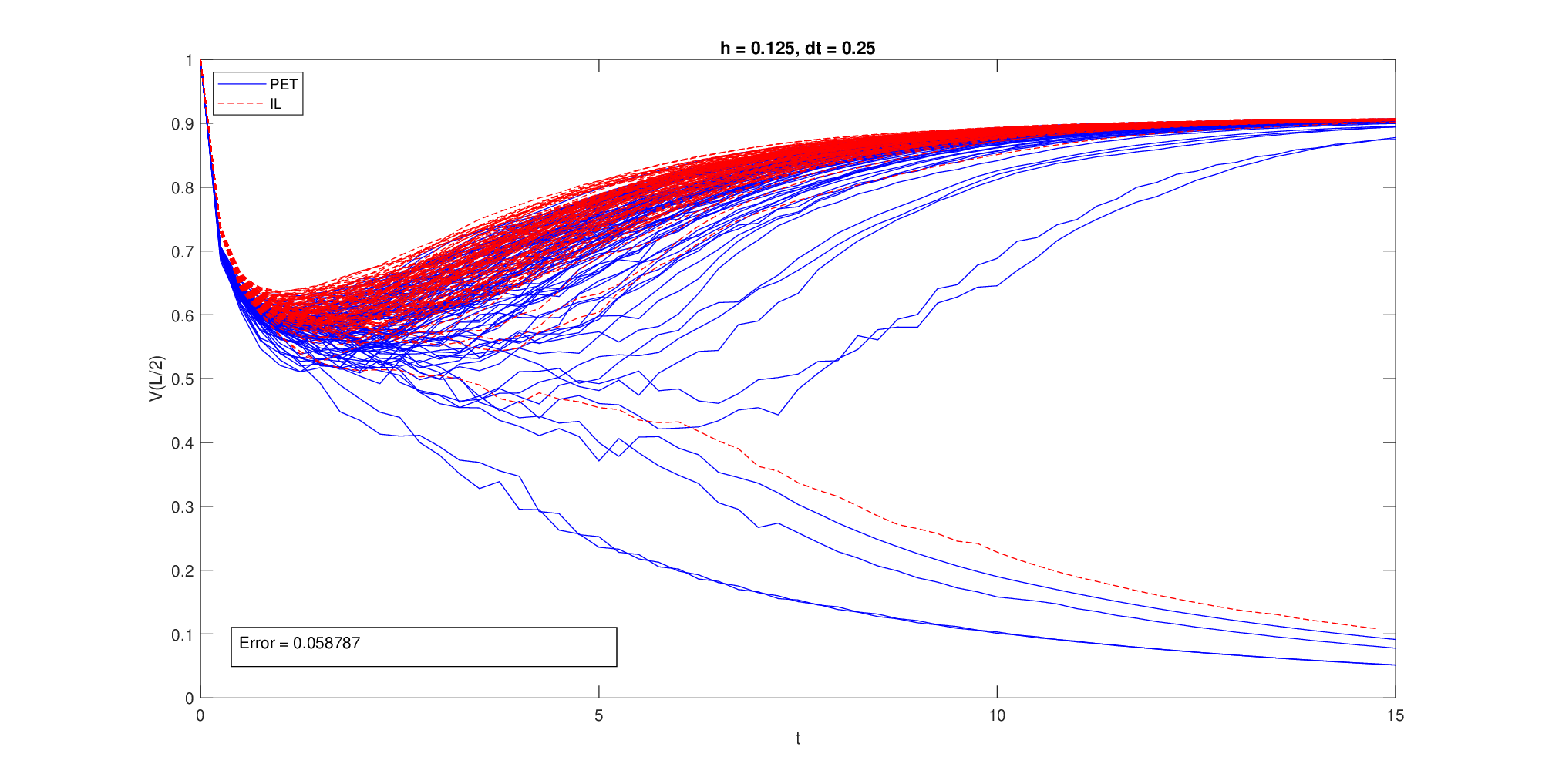}

    \includegraphics[height=0.25\textheight, width=0.5\linewidth]{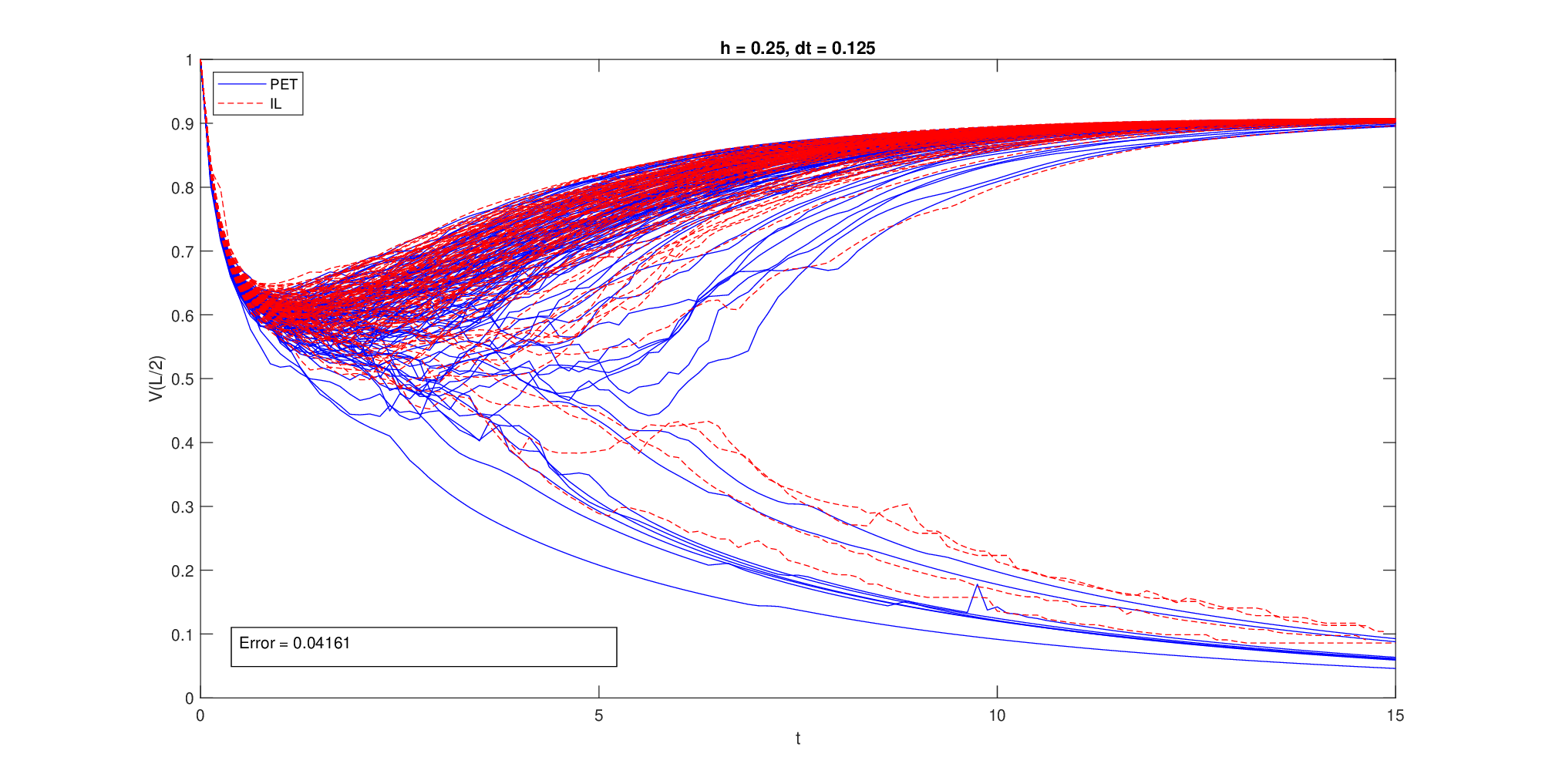}
    \includegraphics[height=0.25\textheight, width=0.5\linewidth]{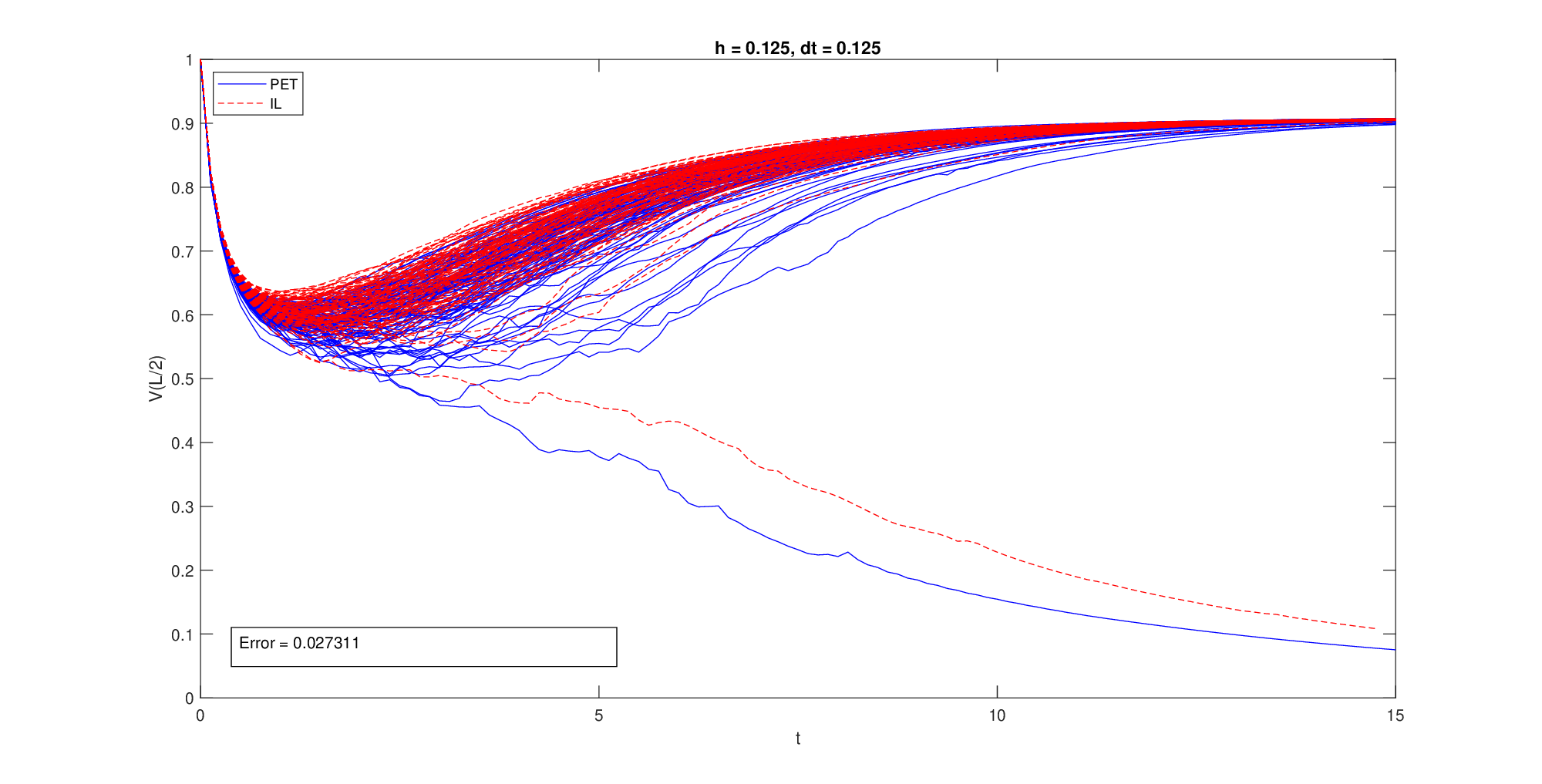}

    \caption{For smaller $h$, $V^{(k)}$ decays less frequently to near $0$. Both algorithms capture this trend,  but IL agrees more with PET when the smaller time step $1/8$ is taken}
    \label{fig:Weak_error}
\end{figure}

\begin{figure}
 \
       \caption*{Quantifying the error using IL}

    \includegraphics[ width=0.5\linewidth]{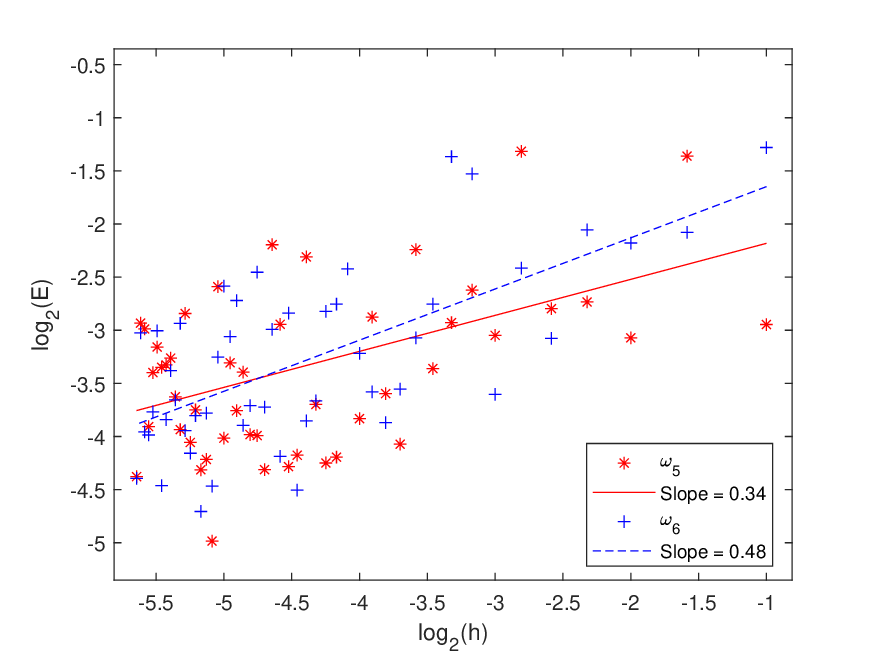}
    \includegraphics[ width=0.5\linewidth]{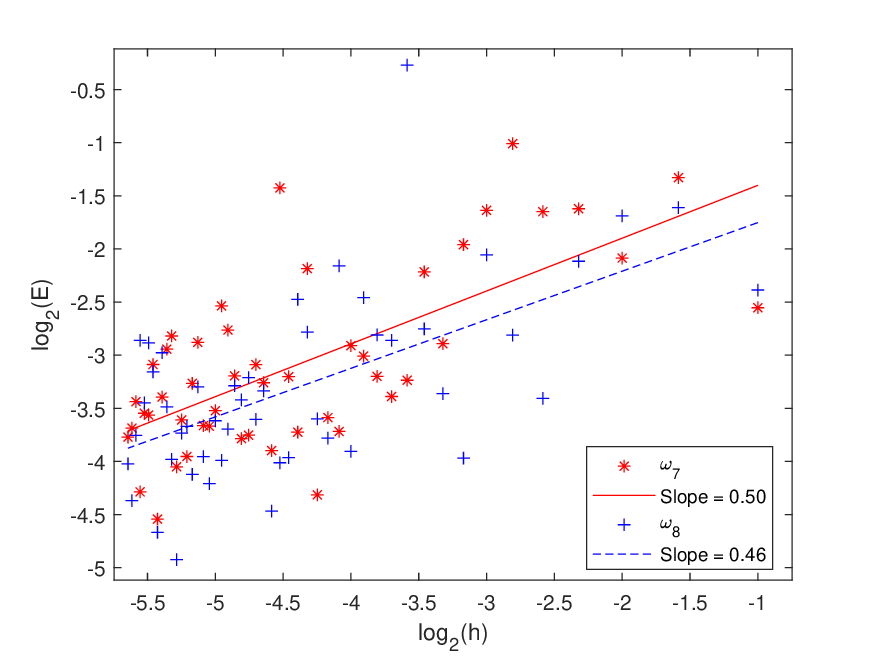}
    
    \includegraphics[width=0.5\linewidth]{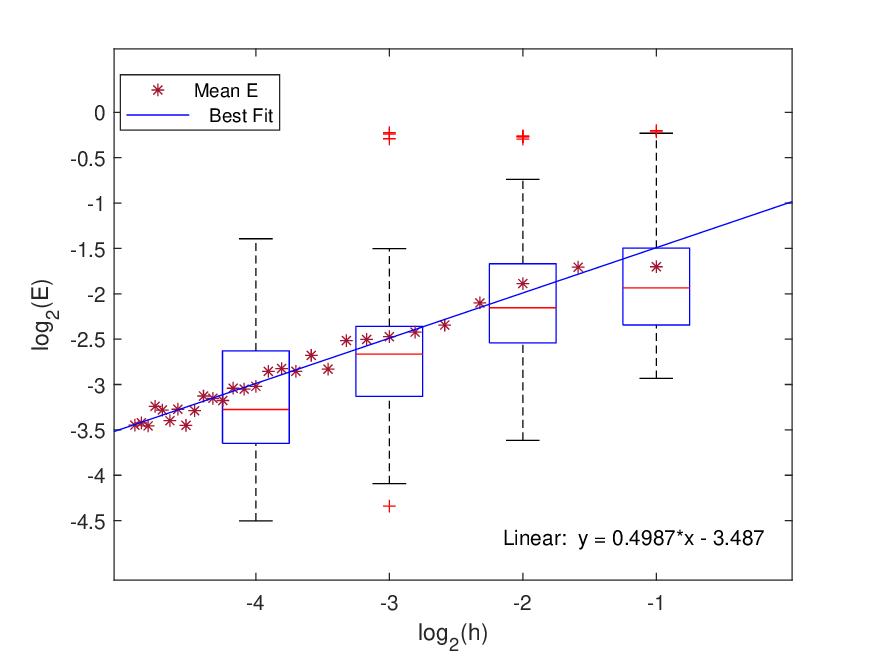}
    \includegraphics[height=0.25\textheight, width=0.5\linewidth]{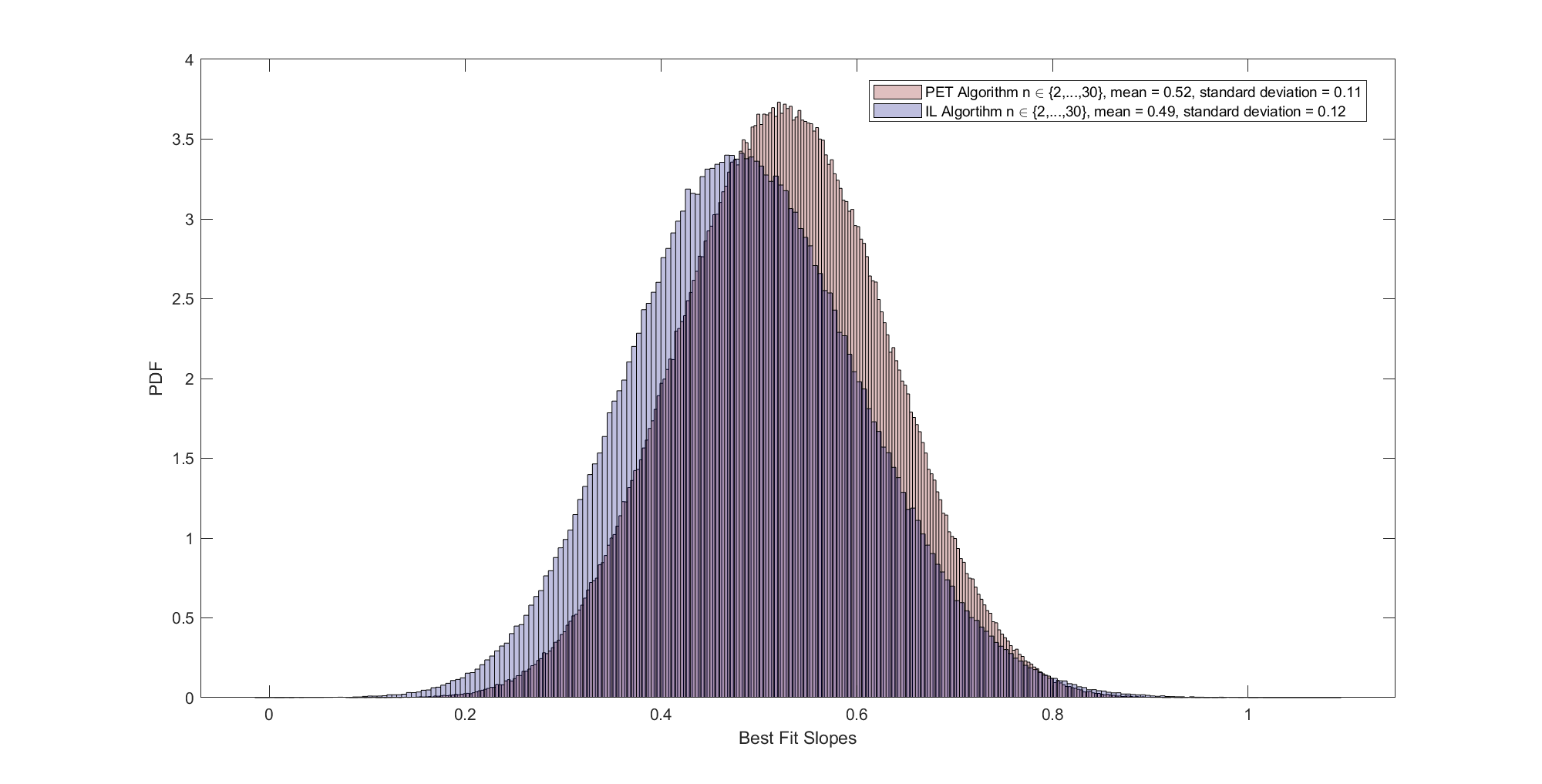}

    \caption{We that more less the trends observed in Figure \ref{fig:experiments} continue. The bottom right panel offers a comparison between the PET and IL algorithm. They produce roughly the same histogram of slopes for when $h \in \{1/n\}_{n=2}^{30}$ with both a mean near $1/2$ and roughly the same variability, although they are certainly not exact}
    \label{fig:alg_II_strong}
\end{figure}

\section{Conclusion}
We have introduced a general model that is both physically realistic, tunable, and easily computable. For example, we allow for rates which may be $0$. As shown in Example \ref{Ex:Switching Channels} and \ref{Ex:macro_density} in Section \ref{SS:Examples} this means we can incorporate randomly spaced or missing ion channels or more generally channels which are distributed according to some macroscopic density. We have elaborated on numerical methods for our model in Section \ref{S:simulation}. Our model may be use then to study the effects channel noise in combination with the effects of spatial heterogeneity in the ion channel distribution.

We have shown that the solutions of stochastic hybrid systems of the form given \eqref{eq:general_model V} and \eqref{eq:general_model Z} are well approximated in a strong sense by solutions of PDEs coupled to ODEs in the form of  \eqref{eqn_V_pde} and \eqref{eqn_Z_pde}. As far as we know, we are the first to prove a strong law of large numbers for such systems and the first to prove any law of large numbers result regarding point-wise 
 uniform convergence for $V$. Moreover, we are able to prove an upper bound on the rate of convergence and run numerical experiments to validate to some extent that our upper bound is correct at least as an upper bound.

Some downsides of our technique are that we require some smoothness on the functions $g_{i,j}$ but in reality the result should be true with only the Lipschitz assumption. We also require smoothness in the initial conditions where probably only continuity should suffice. Further, we have coupled convergence of $V$ with the convergence local averages $\overline{Z}$. By doing so, it may be that the rate of strong convergence we have obtained is not sharp for $V$. However, it is not impossible that the rate of strong convergence we have obtained for $V$ is sharp and that it simply differs from that of the weak convergence. This sharpness is a critical point, because if the rates do differ, it may suggest that path-wise behavior of the general model cannot be seen in an SPDE approximation. Pathwise information becomes more important in models which study random spacing of stochastic ion channels since the spacing is then seen as physically fixed. Our conjecture however, is that the rate of convergence of $V$ to $v$ is wrong, and that the numerical experiments in the left panel of Figure \ref{fig:experiments} are simply too noisy to disprove the rate of convergence is not $\sim h^{1/2}$ modulo some logarithmic factors.

\section*{Declarations}

\textbf{Conflicts of Interests:} All three authors declare there is no conflict of interest in this study. 

\noindent\textbf{Data Availability:} The numerical data used in this study was generated synthetically using the methods described in the paper. Due to its size and complexity, the raw data has not been made publicly available. However, all necessary details for reproducing the data, including the numerical methods and parameters, are provided in the manuscript. Researchers wishing to replicate the data may follow the outlined methods or contact the authors for further guidance.

\section*{Acknowledgements}
W.-T. Fan gratefully acknowledges the hospitality of the Center for Mathematical Biology at UPenn during his visit in Fall 2023, as well as the support of the National Science Foundation through grants DMS-2152103 and DMS-2348164.
Y.M. was supported by the NSF Materials Research Science and Engineering Center (DMR-2309034). Y.M. and J.M. were supported by the
Math+X Award (Proposal Number 234606) from the Simons Foundation. 

\bibliography{ion}

\section{Appendix}

\subsection{Appendix: Proof of Lemma \ref{Lemma_spatial_law}}

\begin{proof}[Proof of Lemma \ref{Lemma_spatial_law}]
For simplicity we write $\tau_{i}(t)=\sum_{k=1}^{n_i}\tau_{k,i}(t)$ and $Q_{i}(t)=\sum_{k=1}^{n_i}\mathcal{N}_{k,i}(\tau_{k,i}(t))$.
We will show that $\Gamma_{\omega,T}$ can be chosen to be $\Gamma_{\omega,T}=\sup_{i \leq i_{\omega,T}}\sup_{t \in [0,T]}|Q_{i}(t)-\tau_{i}(t)|$,  where $i_{\omega,T}$ is a random positive integer.

Let  $\widetilde{\mathcal{N}}_{k,i}$ be the compensated Poisson process corresponding to $\mathcal{N}_{k,i}$, and 
$\widetilde{Q}_{i}(t):=Q_{i}(t)-\tau_{i}(t)= \sum_{k=1}^{n_i} \widetilde{\mathcal{N}}_{k,i}(\tau_{k,i}(t))$
which is exactly the expression inside the absolute sign in \eqref{E:spatial_law}. By our assumption on $\{\tau_{k,i}\}$,  for each $i\geq 1$ and $T\in(0,\infty)$ we have 
\begin{equation}\label{E:BoundQ}
\sup_{t \in [0,T]}|\widetilde{Q}_{i}(t)|\leq \sup_{s \in [0,\tau_T]} \left|\sum_{k=1}^{n_i} \widetilde{\mathcal{N}}_{k,i}(s)\right| 
\end{equation}
and hence for any positive constants $\{\sigma_i\}$,
\begin{align*}
p_{i,T}:=\P\left(\sup_{t \in [0,T]}|\widetilde{Q}_{i}(t)|>3n_i\sigma_i\right)\leq &\,  \P\left(\sup_{s \in [0,\tau_T]} \left|\sum_{k=1}^{n_i} \widetilde{\mathcal{N}}_{k,i}(s) \right|>3n_i\sigma_i\right) \\
=&\,\P\left(\sup_{s \in [0,\tau_T]} \left| \widetilde{\mathcal{N}}(n_i s) \right|>3n_i\sigma_i\right),
\end{align*}
where $\widetilde{\mathcal{N}}$ is a unit rate compensated Poisson process. 

By Entemadi's inequality, 
\begin{equation}
\P\left(\sup_{s \in [0,\tau_T]}\left|\widetilde{\mathcal{N}}(n_i s)\right| \geq 3n_i\sigma_i\right) \leq 3\sup_{s \in [0,\tau_T]} \P\left(\left|\widetilde{\mathcal{N}}(n_i s)\right| \geq n_i\sigma_i \right).
\end{equation}
Hence, by the simple fact $ \P\left(\left|\widetilde{\mathcal{N}}(n_i s)\right| \geq n_i\sigma_i \right)\leq \P\left(\widetilde{\mathcal{N}}(n_i s)\geq n_i\sigma_i \right)+\P\left(-\widetilde{\mathcal{N}}(n_i s)\geq n_i\sigma_i \right)$ and  the Markov inequality, we obtain that for any positive constants and $\{\sigma_i\}$ and $\{y_i\}$,
\begin{align}
p_{i,T} \leq &\, 3\sup_{s \in [0,\tau_T]} \P\left(\left|\widetilde{\mathcal{N}}(n_i s)\right| \geq n_i\sigma_i \right)\\
\leq &\, 3 \sup_{s\in [0,\tau_T]}\left\{ e^{-y_in_i\sigma_i}\mathbb{E}\left[\exp\left(y_i \widetilde{\mathcal{N}}(n_i s) \right)+\exp\left(-y_i\widetilde{\mathcal{N}}(n_i s)\right)\right]\right\}
\end{align}

The generating function of the Poisson distribution gives $\mathbb{E}\left[\exp\left(y\widetilde{\mathcal{N}}(t\right)\right]=\exp\left(t(e^y-1-y)\right)$ for $y\in\R$ and $t\in\R_+$, and hence
$\mathbb{E}\left[\exp\left(y\widetilde{\mathcal{N}}(n_is\right)\right]\leq e^{ey^2 n_is/2}$ for $|y|\leq 1$, $s\in\R_+$ and $i\geq 1$. 
Now we take $y_i=n_i^{-1/2}$ and $\sigma_i=n_i^{-1/2}\log(n_i^{\gamma})$. The previous inequality and the last display give
\begin{equation} 
\sum_{i=1}^{\infty} p_{i,T} \leq 6e^{e\tau_T/2} \sum_{i=1}^{\infty }n_i^{-\gamma}<\infty.
\end{equation}
By the Borel-Cantelli Lemma, there exists an almost surely finite random index $i_{\omega,T}$ such that 
\begin{equation}
\sup_{t\in[0,T]}|\widetilde{Q}_{i}(t) |\leq 3\gamma n_i^{1/2}\log(n_i) \quad \text{for }i\geq i_{\omega,T}.
\end{equation}

Our claim in the first paragraph is established and the lemma is proved.
\end{proof}

\subsection{Appendix: Construction of the smooth bump function $\Phi_{h,p}\in \mathcal{C}(\mathbb{S};[0,1])$}

Although the following definition is technically complicated, all we desire is a smooth function which is one over an interval which contains approximately $h^{p-1}$ ion channels, zero over any interval containing the remaining ion channels, and transitions  between $0$ and $1$ over an interval which contains no ion channels.

For $x \in \mathbb{S}$, let 
\begin{equation}
\Phi_{h,p}(x):=\begin{cases}
1 & 1-[h^{p-1}/2]h \leq x \leq [h^{p-1}/2]h \\
\phi_h([h^{p-1}/2]h) -x)&  [h^{p-1}/2]h \leq x \leq ([h^{p-1}/2]+1)h 
\\
\phi_h(x+[h^{p-1}/2]h-1) & 1-([h^{p-1}/2]+1)h \leq x \leq 1-[h^{p-1}/2]h 
\\
0 & \text{Otherwise},
\end{cases}
\label{Def:Phi2}
\end{equation}
where  $[\cdot]$ denotes the integer part of a real number, and $\phi_h(x)$ is a function which transitions smoothly from $0$ to $1$ over the interval $[0,h]$. One way to define such function $\phi$ for $x \in [0,h] $ is to take 
\begin{align}
\phi_h(x)= \dfrac{e^{-h/x}}{e^{-1/(1-xh^{-1})}+e^{-h/x}}.
\end{align}
Then for $j \in \{0,1,2\}$, there exists a constant $C\in(0,\infty)$ such that  
\begin{align}
\left \| \dfrac{d^j}{dx^j}\phi_h \right \|_{L^\infty(\mathbb{S})} \leq Ch^{-j} \quad \text{for all  }h\in (0,1).
\label{eq regularity of phi}
\end{align}

\begin{figure}
    
    \caption*{Smooth Indicator Function}

        \hspace*{-.5cm}\includegraphics[width=\textwidth]{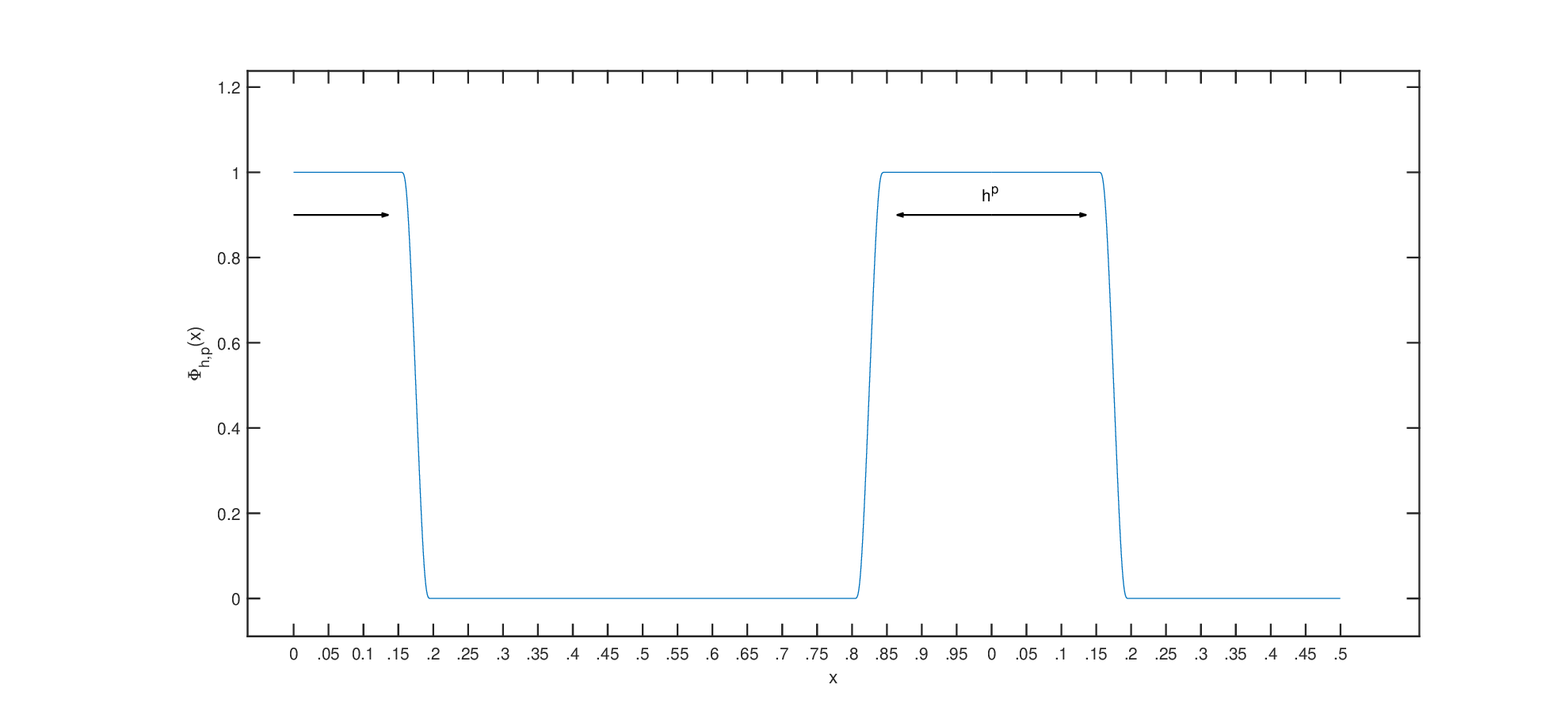}

    \caption{Here we depict $\Phi_{h,p}$ with $h=.05$ and $p=1/3$. For $1-[h^{1/3}/2]\leq x\leq [h^{1/3}/2], \, \, \Phi_{h,p}(x)=1$. Ion channels are located at tick-marks. Notice there are exactly $2[h^{p-1}/2]+1=7$ ion channels contained within the interval where $\Phi_{h,p}(x)=1$}

\end{figure}

\subsection{Appendix: Bounds for $\chi$}
\label{Ap:chi_bounds}
Recall the equation for which $\chi$ solves: \eqref{eqn_chi}. For simplicity we take $D=1.$ Here we justify both inequalities in \eqref{inq_bound_chi} as well as the one in \eqref{eqn_jump_amt}. Recall the definitions of $\overline{Z}$ and $N=N_{h,p}$ given by \eqref{Def:localspatial_Z}, \eqref{E: Size of N}, and \eqref{Def:discretizedLA}. Also recall we identify $n$ with $0$ because periodic boundary conditions imposed. \textit{As such, arithmetic of indices should be regarded as modular arithmetic with modulus $n$.}
For $m,k \in \{0,1,2,\ldots,n-1\},$ let us define $\hat{\delta}_m \in \mathbb{R}^{n}$ as 
\begin{align}
    \hat{\delta}_m^{(k)}:=
    \begin{cases}
    1-1/N & k=m \\
    -1/N &  1<|k-m|_n \leq (N-1)/2 \\
    0 & |k-m|_n >(N-1)/2
    \end {cases},
    \label{Def:hat_delta}
\end{align}
where 
$$|k-m|_n=|k-m| \mod{n}.$$
Note that $\hat{\delta}_m^{(k)}$ depends on $N$ and $n$, but these are not explicitly written in order to conserve subscripts and superscripts for indices. Then we can write $\overline{Z}$ as a convolution with $\hat{\delta}_m$ which is
\begin{align}
    \overline{Z}^{(k)}=\sum_{m=0}^{n-1}Z^{(m)}\hat{\delta}_m^{(k)}.
\end{align}
Thus we should try to solve for $\nu_m$ in the equation 
\begin{align}
\nu_m^{(k+1)}-2\nu_m^{(k)}+\nu_m^{(k-1)}=\hat{\delta}_m^{(k)}. 
\end{align}
Once we have done so, we see that 
\begin{align}
    \chi^{(k)}=\sum_{m=0}^{n-1}Z^{(m)}\nu^{(k)}_m 
    \label{E:chi_formula}
\end{align}
is a solution to \eqref{eqn_chi}. Therefore, using that for all $m$ and $t$, $Z^{(m)} \in \{0,1\}$
\begin{align}
    \sup_{t \in \mathbb{R}^+}\max_{k \in \{0,1,2,\ldots,n-1\}}|\chi_t^{(k)}|\leq\max_{k \in \{0,1,2,\ldots,n-1\}}\sum_{m=0}^{n-1}|\nu^{(k)}_m|.
    \label{Inq:chiyoung}
\end{align}
To solve for $\nu_m$, we split $\hat{\delta}_m$ into $\hat{\delta}_m=\delta_m-\bar{\delta}_m$. Here $\delta_m$ is the usual Kronecker symbol, and $\bar{\delta}_m$ is  the local average version $\delta_m$. Specifically
\begin{align}
    \bar{\delta}_m^{(k)}:=
    \begin{cases}
    1/N &  |k-m|_n \leq (N-1)/2 \\
    0 & |k-m|_n >(N-1)/2
    \end {cases}.
\end{align}
For each $k$, we have by definition of $N$ in \eqref{E: Size of N} and $\hat{\delta}_m$ in \eqref{Def:hat_delta} that
\begin{align}
\begin{aligned}
    \hat{\delta}_m^{(k)}&=\delta_m^{(k)}-\bar{\delta}^{(k)}_m
    \\
    &=\dfrac{1}{N}\sum_{j=-(N-1)/2}^{(N-1)/2}\left(\delta_m^{(k)}-\delta_m^{(k+j)}\right)
    \\
    &=\dfrac{1}{N}\sum_{j=1}^{(N-1)/2}\left(\delta_m^{(k)}-\delta_m^{(k-j)}\right)-\dfrac{1}{N}\sum_{j=1}^{(N-1)/2}\left(\delta_m^{(k+j)}-\delta_m^{(k)}\right)\\
    &=\dfrac{1}{N}\sum_{j=1}^{(N-1)/2}\sum_{l=1}^{j}\left(\delta_m^{(k-l+1)}-\delta_m^{(k-l)}\right)-\dfrac{1}{N}\sum_{j=1}^{(N-1)/2}\sum_{l=1}^{j}\left(\delta_m^{(k+l)}-\delta_m^{(k+l-1)}\right).
    \end{aligned}
    \label{E:delta_first}
\end{align}
To prevent equations from becoming too long, define 
\begin{align}
\mu^{(k)}_{m,j}:=\sum_{l=1}^{j}\delta_m^{(k+l)}.
\label{Def:nu_mu}
\end{align}
From \eqref{E:delta_first} we have that 
\begin{align}
\begin{aligned}
    \hat{\delta}_m^{(k)}&=\dfrac{1}{N}\sum_{j=1}^{(N-1)/2}\left(\mu_{m,j}^{(k-j)}-\mu_{m,j}^{(k-1-j)}\right)-\dfrac{1}{N}\sum_{j=1}^{(N-1)/2}\left(\mu_{m,j}^{(k)}-\mu^{(k-1)}_{m,j}\right)
    \\
    &=\dfrac{1}{N}\sum_{j=1}^{(N-1)/2}\left[\left(\mu_{m,j}^{(k-j)}-\mu_{m,j}^{(k-1-j)}\right)-\left(\mu_{m,j}^{(k)}-\mu^{(k-1)}_{m,j}\right)\right]
    \\
    &=\dfrac{1}{N}\sum_{j=1}^{(N-1)/2}\sum_{l=1}^{j}\left[\left(\mu_{m,j}^{(k-l)}-\mu_{m,j}^{(k-1-l)}\right)-\left(\mu_{m,j}^{(k-l+1)}-\mu_{m,j}^{(k-l)}\right)\right]
    \\&=-\dfrac{1}{N}\sum_{j=1}^{(N-1)/2}\sum_{l=1}^{j}\left(\mu_{m,j}^{(k+1-l)}-2\mu_{m,j}^{(k-l)}+\mu_{m,j}^{(k-1-l)}\right).
    \end{aligned} 
\end{align}
It follows from \eqref{Formula_for_nu} and the definition of $\mu$ in \eqref{Def:nu_mu} that
\begin{align}
    \nu_m^{(k)}=-\dfrac{1}{N}\sum_{j=1}^{(N-1)/2}\sum_{l=1}^{j}\mu_{m,j}^{(k-l)}=-\dfrac{1}{N}\sum_{j=1}^{(N-1)/2}\sum_{l=1}^{j}\sum_{i=1}^{j}\delta_m^{(k-l+i)},
    \label{Formula_for_nu}
\end{align}
 and thus 
 \begin{align}
 \begin{aligned}
 \max_{k \in \{0,1,2,\ldots, n-1\}}\sum_{m=0}^{n-1}\left|\nu_m^{(k)}\right|&=\sum_{m=0^{n-1}}\dfrac{1}{N}\sum_{j=1}^{(N-1)/2}\sum_{l=1}^{j}\sum_{i=1}^{j}\delta_m^{(k-l+i)}
\\
&=\dfrac{1}{N}\sum_{j=1}^{(N-1)/2}\sum_{l=1}^{j}\sum_{i=1}^{j}\delta_{k-l-j}^{(k-l+i)}
\\
&=\dfrac{1}{N}\sum_{j=1}^{(N-1)/2}j^2
\\
&=\dfrac{N^2-1}{24}.
\end{aligned}
\label{Bound_l1_nu}
\end{align}
Using \eqref{Inq:chiyoung} and \eqref{E: Size of N}, we establish the first inequality in \eqref{inq_bound_chi}. 

To establish the second, we take the discrete derivative of $\nu$. Using its formula in \eqref{Formula_for_nu}, we calculate 
\begin{align}
\begin{aligned}
    \nu_m^{(k+1)}-\nu_m^{(k)}&=-\dfrac{1}{N}\sum_{j=1}^{(N-1)/2}\sum_{l=1}^{j}\sum_{i=1}^{j}\left(\delta_m^{(k-l+i+1)}-\delta_m^{(k-l+i)}\right)
    \\ 
    =&-\dfrac{1}{N}\sum_{j=1}^{(N-1)/2}\sum_{l=1}^{j}\left(\delta_m^{(k-l+j+1)}-\delta_m^{(k-l+1)}\right).
    \end{aligned}
\end{align}
In this case we have that 
\begin{align}
\begin{aligned}
   \max_{k \in \{0,1,2,\ldots, n-1\}} \sum_{m=0}^{n-1}\left|\nu_m^{(k+1)}-\nu_m^{(k)}\right|&\leq \dfrac{1}{N}\sum_{j=1}^{(N-1)/2}\sum_{l=1}^{j}\left(\delta_{k-l+j+1}^{(k-l+j+1)}+\delta_{k-l+1}^{(k-l+1)}\right) 
    \\
    &= \dfrac{1}{N}\sum_{j=1}^{(N-1)/2}2j
    \\
    &=\dfrac{N^2-1}{4N},
    \end{aligned}
    \label{bound_l1_difnu}
\end{align}
and so the second inequality in \eqref{inq_bound_chi} is readily established in an analogous manner to the first.

Finally, we establish \eqref{eqn_jump_amt}. Using \eqref{E:chi_formula} and that almost surely there is exactly one $l \in \{0,1,2,\ldots,n-1\}$ for which $Z_t^{(l)}$ changes at time $t_i$ from $1$ to $0$ or $0$ to $1$, we obtain
\begin{equation}
    \left(\chi_t^{(k)}-\lim_{t\to t_i^-}\chi_t^{(k)}\right) =\sum_{m=0}^{n-1}\left(Z^{(m)}_t-\lim_{t\to t_i^-}Z^{(m)}_t\right)\hat{\delta}^{(k)}_m=\pm\sum_{m=0}^{n-1}\delta^{(m)}_{l}\nu^{(k)}_m=\pm\nu_l^{(k)}.
\end{equation}
We find from \eqref{Formula_for_nu} that
\begin{equation}
    \max_{k \in \{0,1,2\ldots,n-1\}}\left|\nu^{(k)}_m\right| \leq \dfrac{N^2-1}{8N}.
    \label{bound_max_nu}
\end{equation}
The bound in \eqref{eqn_jump_amt} is thus readily established using \eqref{E: Size of N}.

\end{document}